\documentclass[oneside, a4paper,10pt,reqno, english]{amsart}
\usepackage{adjustbox}
\usepackage[a4paper,heightrounded]{geometry}
\usepackage[T1]{fontenc}
\usepackage{extsizes}
\usepackage{anysize,color}
\usepackage[utf8]{inputenc}
\usepackage[dvipsnames]{xcolor}
\usepackage[colorlinks=true,urlcolor={YellowOrange},linkcolor={YellowOrange},citecolor={YellowOrange}]{hyperref}  
\usepackage{float}
\usepackage{amsfonts,amssymb,amsmath}
\usepackage{amsthm}    
\usepackage{url}
\usepackage{paralist}
\usepackage{tikz}
\usetikzlibrary{decorations.markings,arrows,shapes.callouts}
\usepackage{rotating} 
\usepackage[mathscr]{euscript}
\usepackage{verbatim}
\usepackage{tikz-cd}
\usepackage{mathtools}
\usepackage{dsfont}
\usepackage{amsrefs}
\usepackage{thmtools}
\usepackage{thm-restate}
\usepackage{booktabs}
\usepackage{subcaption}

\usepackage{charter}

\theoremstyle{plain}
\newtheorem{thm}{Theorem}[section]
\newtheorem{lem}[thm]{Lemma}
\newtheorem{cor}[thm]{Corollary}

\newtheorem{thmx}{Theorem}

\newtheorem{prop}[thm]{Proposition}

\newtheorem{thmro}{Theorem}

\makeatletter
\newcommand{\neutralize}[1]{\expandafter\let\csname c@#1\endcsname\count@}
\makeatother

\theoremstyle{definition}

\newtheorem{ex}[thm]{Example}
\newtheorem{rem}[thm]{Remark}

\newtheorem{que}[thm]{Question}
\newtheorem{conj}[thm]{Conjecture}

\newtheorem{notate}[thm]{Notation}

\newcommand \bb \mathbb

\newcommand \rarr \rightarrow
\newcommand \larr \leftarrow
\newcommand \uarr \uparrow
\newcommand \darr \downarrow
\newcommand \Rarr \Rightarrow
\newcommand \Larr \Leftarrow

\newcommand \emb \hookrightarrow
\newcommand \surj \twoheadrightarrow

\DeclareMathOperator{\id}{id}

\DeclareMathOperator{\Ext}{Ext}

\DeclareMathOperator{\co}{H}

\newcommand{\PProj} {\mathbb{P}}

\newcommand\rsmraise[1]{%
  \ifx#1\displaystyle .8\else
    \ifx#1\textstyle .8\else
      \ifx#1\scriptstyle .6\else
        .45%
      \fi
    \fi
  \fi}

\def\ker{\mathrm{ker}}

\def\pspan{\mathrm{pspan}}
\def\span{\mathrm{span}}

\def\cspan{\mathrm{span}_{\C}}
\def\cpspan{\mathrm{pspan}_{\C}}

\newcommand{\be}{\begin{enumerate}}
\newcommand{\ee}{\end{enumerate}}
\newcommand{\R}{\mathbb{R}}
\newcommand{\Z}{\mathbb{Z}}
\newcommand{\C}{\mathbb{C}}

\DeclareMathOperator{\lobs}{\mathscr{O}}

\usetikzlibrary{calc}

\usetikzlibrary{decorations.pathmorphing}

\tikzset{curve/.style={settings={#1},to path={(\tikztostart)
    .. controls ($(\tikztostart)!\pv{pos}!(\tikztotarget)!\pv{height}!270:(\tikztotarget)$)
    and ($(\tikztostart)!1-\pv{pos}!(\tikztotarget)!\pv{height}!270:(\tikztotarget)$)
    .. (\tikztotarget)\tikztonodes}},
    settings/.code={\tikzset{quiver/.cd,#1}
        \def\pv##1{\pgfkeysvalueof{/tikz/quiver/##1}}},
    quiver/.cd,pos/.initial=0.35,height/.initial=0}

\tikzset{tail reversed/.code={\pgfsetarrowsstart{tikzcd to}}}
\tikzset{2tail/.code={\pgfsetarrowsstart{Implies[reversed]}}}
\tikzset{2tail reversed/.code={\pgfsetarrowsstart{Implies}}}

\tikzset{no body/.style={/tikz/dash pattern=on 0 off 1mm}}

\def\keywords{\xdef\@thefnmark{}\@footnotetext}

\begin{document}

\title{Complex line fields on almost-complex manifolds}


\author[Sadovek]{Nikola Sadovek}
\address{Max Planck Institute of Molecular Cell Biology and Genetics, Center for Systems Biology Dresden,\newline
Faculty of Mathematics, Technische Universit\"at Dresden, Germany}
\email{sadovek@mpi-cbg.de}

\author[Schutte]{Baylee Schutte} 
\address{Institute of Mathematics, Freie Universit\"at Berlin, Germany}
\email{bschutte@zedat.fu-berlin.de, bayleeschutte.math@gmail.com}

\thanks{\hspace{-4mm}\textit{2020 Mathematics Subject Classification.} Primary 55S40, 32Q60; Secondary 55S35, 57R20, 55S45. 
The research of N.\ Sadovek is funded by the Deutsche Forschungsgemeinschaft (DFG, German Research Foundation) under Germany's Excellence Strategy--The Berlin Mathematics Research Center MATH+ (EXC-2046/1, project ID 390685689, BMS Stipend). B.\ Schutte received no funding to assist with the preparation of this manuscript.}


\begin{abstract}
    We study linearly independent complex line fields on almost-complex manifolds, which is a topic of long-standing interest in differential topology and complex geometry. A necessary condition for the existence of such fields is the vanishing of appropriate virtual Chern classes. We prove that this condition is also sufficient for the existence of one, two, or three linearly independent complex line fields over certain manifolds. More generally, our results hold for a wider class of complex bundles over CW complexes.
\end{abstract}

\maketitle



\setcounter{tocdepth}{1}
\tableofcontents

\section{Introduction}

Let $\xi$ be a complex vector bundle of rank $m \ge 1$ over a $2m$-dimensional CW complex \(X\). For an integer $1 \le r \le m$ and line bundles $\ell_1, \dots, \ell_r$ over $X$, the Whitney sum formula for Chern classes provides a necessary cohomological condition for the existence of an embedding of $\ell_1 \oplus \cdots \oplus \ell_r$ into $\xi$. By way of elaboration, such an embedding is equivalent to the existence of an isomorphism $\xi \cong \ell_1 \oplus \cdots \oplus \ell_r \oplus \eta$, for some rank $m-r$ bundle $\eta$. Hence the total Chern class factors as $c(\xi) = c(\ell_1) \cdots c(\ell_r) c(\eta) \in \co^*(X;\Z)$, and we deduce that
\begin{equation} \label{eq:necessary condition}
    \ell_1 \oplus \dots \oplus \ell_r \subseteq \xi ~ \Longrightarrow ~ c(\xi) = c(\ell_1) {\cdots}  c(\ell_r)   x \in \co^*(X;\Z), ~\text{for some}~x \in \co^{\le 2(m-r)}(X;\Z).
\end{equation}
A priori, it is unclear if the right hand side of \eqref{eq:necessary condition} is also a sufficient condition, which leads to the following fundamental question:

\begin{que} \label{que:fundamental}
    Is there an algebraic condition that ensures the existence of an embedding of $\ell_1 \oplus \dots \oplus \ell_r$ into $\xi$? Specifically, does the inverse implication of \eqref{eq:necessary condition} hold, possibly under certain assumptions on $X$?
\end{que}

As the main result of this paper, we use Moore--Postnikov theory to positively answer Question \ref{que:fundamental}, for $r \le 3$, with the proviso that we make certain assumptions about the base space $X$. Namely, for certain $X$, we prove that the implication in \eqref{eq:necessary condition} is an equivalence. Thus, in some sense, we obtain an equivalence between the geometric world (i.e. the existence of line bundle embeddings) and the algebraic world (i.e. Chern classes). We summarize our main result below, but  we state our results in full generality (that is, under milder conditions on the base space) in Section \ref{sec:main_thms}.

\begin{thm}[Main theorem] \label{thm:main}
Let $M$ be a smooth closed connected $2m$-manifold and $\xi$ a complex $m$-plane bundle over $M$.
\begin{enumerate}[\normalfont(I)]
    \item Suppose that \(m \ge 1\). If $\ell$ is a complex line bundle over $M$, then
    \begin{equation*}
        \ell \subseteq \xi ~ \iff ~ c(\xi) = c(\ell)\cup x \in \co^*(M;\Z), ~\text{for some}~x \in \co^{\le 2(m-1)}(M;\Z).
    \end{equation*}
    \item Suppose that \(m \ge 3.\) If \(m\) is even, assume that \(w_2(M) \neq 0\). Then 
    \begin{equation*}
        \ell_1 \oplus \ell_2 \subseteq \xi ~ \iff ~ c(\xi) = c(\ell_1) \cup c(\ell_2)\cup x \in \co^*(M;\Z), ~\text{for some}~x \in \co^{\le 2(m-2)}(M;\Z).
    \end{equation*}
        \item Suppose that \(m \geq 5\), \(\pi_1(M) = 0\), and \(\co_2(M;\Z)\) has no \(2\)-torsion. If \(m\) is odd, assume that \(\co^{2m-2}(M;\Z/2)= Sq^2 \co^{2m-4}(M;\Z/2)\) and \(\co_3(M;\Z) =0\). Then 
    \begin{equation*}
        \ell_1 \oplus \ell_2 \oplus \ell_3 \subseteq \xi ~\iff~  c(\xi) = c(\ell_1)\cup c(\ell_2)\cup c(\ell_3)\cup x \in \co^*(M;\Z), ~\text{for some}~x \in \co^{\le 2(m-3)}(M;\Z).
    \end{equation*}
\end{enumerate}
\end{thm}

 It is readily seen that the assumptions on the base space in Theorem \ref{thm:main} are satisfied for e.g., complex projective spaces of appropriate dimension. Therefore, complex projective spaces represent a first family of examples for which our results hold. Moreover, coupled with the Schwarzenberger condition \cite[Theorem A]{thomas1974almost}---which provides a description of the total Chern class of complex $m$-plane bundles over $\C P^m$---our work refines the understanding of the extent to which complex vector bundles over \(\C P^m\) with certain line bundle splitting properties are described by their Chern classes. For a precise statement, see Theorem \ref{thm:cp^m}.

\subsection*{History and related work}

When the $\ell_i$ are all trivial line bundles $\varepsilon^1 \colon~ \C \to \C \times X \to X$, Question \ref{que:fundamental} has been extensively studied. Hopf's theorem for almost-complex manifolds is a first example where purely algebraic data (e.g., cohomological invariants) completely characterizes some geometric property (e.g., non-vanishing vector fields).
Recall that a smooth \(2m\)-manifold is \emph{almost-complex} if the tangent bundle \(TM\) of \(M\) is isomorphic to the underlying real bundle \(\omega^\R\) of some complex \(m\)-plane bundle \(\omega\) over \(M\). Let $X = M$ be an almost-complex $2m$-manifold and $\xi = TM$ the corresponding tangent bundle. Recalling that the top Chern class of an almost-complex manifold is equal to the Euler class, Hopf's theorem states that the Euler characteristic $\chi(M)$ of \(M\) determines the existence of a nonvanishing \emph{complex vector field} on $M$, see \cite{hopf27vector}. For example, notice that the almost-complex spheres \(S^i\), for \(i = 2,6\), and complex projective spaces $\C P^n$, for \(n \geq 0\), have zero complex span by virtue of their non-vanishing Euler characteristics.

More generally, the search for cohomological conditions guaranteeing the existence of linearly independent sections of real vector bundles has a rich history, as surveyed for example in \cite{thomas1969vectorfields,korbas1994survey,sankaran2019survey}. In this case, for \(1 \leq r \leq m\), the implication \eqref{eq:necessary condition} takes the following form:
\begin{equation} \label{eq:lin-indep-sections}
    \varepsilon^r \subseteq \xi ~ \Longrightarrow ~ c_m(\xi) = \dots = c_{m+1-r}(\xi) = 0.
\end{equation}
Hence the Chern classes of \(\xi\) obstruct the existence of $r$ linearly independent sections of $\xi$. The above leads to the study of a numerical invariant of \(\xi\) call the \emph{(complex) span}, defined as
\begin{equation*}
    \cspan(\xi) \coloneqq \max\{r \in \Z_{\geq 0} \mid \varepsilon^{r} \subseteq \xi\}.
\end{equation*}

In the complex case, when \(m \ge 3\) is odd, Thomas \cite[Corollary 3.6]{thomas1967postnikov} proved that \(\cspan(M) \geq 2\) if and only if \(c_m(M) = c_{m-1}(M) = 0\). In the case that \(m \ge 3\) is even, he proved in \cite[Theorem 1.5]{thomas1967real} that $c_{m}(\xi) = c_{m-1}(\xi)=c_{m-2}(\xi) = 0$ is a sufficient condition for \(\cspan(M) \geq 2.\) Additionally, Thomas \cite[Corollary 1.6]{thomas1967complex} obtained a slightly weaker version of a result by Gilmore \cite[Theorem 1]{gilmore1967complex}, who proved that 
\begin{equation*}
    \cspan(M) \ge 3 ~ \iff ~ c_{m-2}(M) = c_{m-1}(M) = c_{m}(M) = 0,
\end{equation*}
where $m \ge 6$ is an even integer, $M$ is simply connected, and $\co_2(M;\Z)$ has no 2-torsion. Furthermore, Gilmore \cite[Theorems 2 \& 3]{gilmore1967complex} identified situations in which there exist necessary and sufficient conditions for \(\cspan(M) \geq 4 \text{ or } 5\). As noted in Section \ref{sec:main_thms}, the main results of this paper not only specialize to recover several of these classically known facts, but also yield a new result about the complex span of non-spin almost-complex manifolds, see Corollary \ref{cor:cpspan2_even}. 

We also mention recent work of Nguyen \cite{nguyen2024cobordism}, who defined certain cobordism obstructions whose vanishing dictates whether or not a given almost-complex manifold is ``complex section cobordant'' to an almost-complex manifold with complex span at least \(r>0\).

Chiefly, this paper furthers the examination of the ``projectivization'' of span initiated by Grant and the second author in \cite{grant2024projective}. Indeed, we study another numerical invariant of \(\xi\) called the \emph{(complex) projective span}, defined as 
\begin{equation*}
    \cpspan(\xi) \coloneqq \max\left\{r \in \Z_{\geq 0} \mid \text{there exists complex line bundles } \{\ell_i\}_{i = 1}^r \text{ such that } \bigoplus_{i=1}^r\ell_i \subseteq \xi\right\}.
\end{equation*}
In the language of sections, since a line subbundle of $\xi \colon \C^m \to E \to X$ can be identified with a section of the projectivization $\PProj(\xi) \colon \C P^{m-1} \to \PProj(E) \to X$ of $\xi$, we surmise that the value of $\cpspan(\xi)$ represents the maximal number of linearly independent sections of $\PProj(\xi)$. Of particular interest in this paper is the case when $X=M$ is a smooth $2m$-dimensional almost-complex manifold and $\xi=TM$ its tangent bundle. In this case, a complex line subbundle \(\ell \subseteq TM\), or equivalently a smooth section of the projectivized tangent bundle \(\mathbb{P}TM\), is called a \emph{complex line field} on \(M\).

For example, it follows from a theorem of  Glover--Homer--Stong \cite[Theorem 1.1(ii)]{glover1982splitting}  that \[\cpspan(T\C P^{2k}) = 0 \text{ and } \cpspan(T\C P^{2k+1}) = 1,\] for \(\C P^n\) equipped with its standard almost-complex structure. Note however that complex projective spaces are stably line element parallelizable. That is, \[\pspan_{\mathbb{C}}(T\mathbb{C} P^n\oplus \varepsilon^1) = n,\] for all \(n\), see \cite[p.\ 170]{milnor2005characteristic}.

We briefly discuss key research related to real projective span.  For more information, we suggest the introduction of \cite{grant2024projective}, where Grant and the second author compute the real projective span of all the Wall manifolds. It is a classical result due to Markus \cite{markus1955line} and Samelson \cite{samelson1951diff} that the real projective span of a smooth manifold is strictly greater than zero if and only if the Euler characteristic vanishes. In agreeable situations, necessary and sufficient conditions for the existence of monomorphisms \(\alpha \xhookrightarrow{} \beta\) of vector bundles over a real (non)orientable smooth manifold were identified by Mello \cite{mello1987two}, for \(\mathrm{rank}(\alpha)=2\), and Mello--da Silva \cite{mello2000note}, for \(\mathrm{rank}(\alpha)=3\). The former inquiry utilizes obstruction theory while the latter appeals to Koschorke's \cite{koschorke1981book} normal bordism and singularity theory; principally, however, one may extract information about real projective span from these results. Further work on decompositions of real vector bundles using obstruction theory was carried out by Leslie--Yue \cite{leslie2003vector}. When it comes to complex projective span, we can mention the work of Iberkleid \cite[Theorem 3.4]{iberkleid1974splitting}, who proved that an almost-complex \(2m\)-manifold \(M\) is unoriented cobordant to an almost-complex \(2m\)-manifold \(N\) with \(\cpspan(N) = m\) if and only if the Euler characteristic \(\chi(M)\) is even. 

Notice that if \(\xi\) is a complex \(m\)-plane bundle with \(\cspan(\xi) \geq r\), then \(\span_{\R}(\xi^\R) \geq 2r\) because the underlying real bundle of a trivial complex line bundle is itself trivial. However, if \(\cpspan(\xi)\geq r,\) it is not generally true that \(\pspan_\R(\xi^{\R}) \geq 2r\) since the underlying real bundle of a nontrivial complex line bundle may not split into two real line bundles (take for example the Hopf line bundle over complex projective space). Moreover, Koschorke \cite{koschorke1999complex} studies the relationship between the existence and homotopy classification theories of embeddings \(\zeta \xhookrightarrow{}_{\C} \eta\) of complex bundles and embeddings \(\zeta_\R \xhookrightarrow{}_\R \eta_\R\) of the underlying real bundles---see in particular \cite[\S 4 Complex line fields vs.\ real plane fields]{koschorke1999complex}. Especially of note, he gives examples of complex line bundles \(\alpha\) over \(S^1 \times S^1 \times S^{2n}\), where \(n>2\), for which there exist infinitely many real embeddings \(\alpha_{\R} \xhookrightarrow{}_{\R} \beta_\R\) not homotopic to complex embeddings \(\alpha \xhookrightarrow{}_{\C} \beta\), for certain \(\beta\) \cite[Example 4.10]{koschorke1999complex}.

To conclude, since a one-dimensional complex subbundle of the tangent bundle of a compact connected complex manifold is tangent to a holomorphic one-dimensional foliation, our results may interact with K\"{a}hler geometry. More specifically, suppose that \(M\) is a compact connected K\"{a}hler manifold with \(\cpspan(M) \geq r\), for some \(r >0\). Then \cite{brunella2006kahler} relates the splitting of the tangent bundle of \(M\) with splitting properties of the universal cover of \(M\). 

\subsection*{Organization}

The remainder of the paper is organized as follows. In Section \ref{sec:applications} we provide the first example of Theorem \ref{thm:main} in the case of complex projective spaces. Then we state our main results in complete detail in Section \ref{sec:main_thms}. Next, we recall in Section \ref{sec:prelims} some basic elements of Moore-Postnikov theory as well as various classical results needed throughout the paper. We formulate the appropriate lifting problem in Section \ref{sec:primaryobs}, where we compute the primary obstruction to splitting off \(r\)-line bundles. Finally, Sections \ref{sec:thm_cpspan2_odd}--\ref{sec:thm_cpspan2_even} and Section \ref{sec:thm_pspan3_even}--\ref{sec:thmpspan3_odd} contain the largest parts of our proofs, which make use of Moore-Postnikov theory. Due to considerable technical differences, we treat the cases of splitting off two and three line bundles separately.

\section{Application}
\label{sec:applications}

In this section, we apply our main result, Theorem \ref{thm:main}, to the case of complex $m$-plane bundles $\xi$ over the complex projective space $\C P^m$. For integers $c_1, \dots, c_m$ and a class
\begin{equation} \label{eq:class-c}
    c \coloneqq 1 + c_1u + c_2u^2 + \dots + c_mu^m \in \co^*(\C P^m;\Z) \cong \Z[u]/(u^{m+1}), \hspace{5mm} \text{where}~ \lvert u \rvert= 2,
\end{equation}
the Schwarzenberger condition \cite[Theorem A]{thomas1974almost} detects whether $c$ is the total Chern class of a complex $m$-bundle over $\C P^m$. More specifically, the  class $c$ factors as
\begin{equation*} \label{eq:factorization-of-c}
    c=(1+z_1u)\cdots(1+z_mu) \in \C[u], \hspace{5mm} \text{for some}~ z_1, \cdots, z_m \in \C.
\end{equation*}
The Schwarzenberger condition says that \(c\) is the total Chern class of a complex vector bundle over \(\C P^m\) if and only if
\begin{equation} \label{eq:schwarzenberger}
    \binom{z_1}{k} + \dots + \binom{z_m}{k}  \in \Z, \hspace{5mm} \text{for all}~ k = 2, \dots, m.
\end{equation}
Here $\binom{z}{k} = z(z-1) \dots (z-k+1)/k! \in \C$, for $z \in \C$.
Using this notation, we provide a complete description of the cohomology classes in $\co^*(\C P^m;\Z)$ which are the total Chern classes of complex $m$-bundles over $\C P^m$ admitting a sum of at most three line bundles as a subbundle.

\begin{thm} \label{thm:cp^m}
    Let \(m \geq 1\) and \(r \leq 3\) be integers. Further let $\ell_1, \dots, \ell_r$ be given line bundles over $\C P^m$ with total Chern classes $c(\ell_i) = 1+z_i u \in \co^*(\C P^m;\Z)$, where $z_i \in \Z$. For $m > 2r-1$, the class
    \[
    c \coloneqq (1+z_1u) \cdots (1+z_ru)(1+z_{r+1}u) \cdots (1+z_mu) \in \C[u], \hspace{5mm} \text{for some}~ z_{r+1}, \dots, z_{m} \in \C,
    \]
    is the total Chern class of a complex \(m\)-plane bundle \(\xi\) over \(\C P^m\) for which \(\ell_1 \oplus \cdots \oplus \ell_r \subseteq \xi\) if and only if 
    \begin{equation*}
        \binom{z_{r+1}}{k} + \dots + \binom{z_m}{k}  \in \Z, \hspace{5mm} \text{for all}~ k = 2, \dots, m.
    \end{equation*}
\end{thm}
\begin{proof}
    The claim follows by applying the Schwarzenberger condition \eqref{eq:schwarzenberger} and Theorem \ref{thm:main}. To be able to apply the theorem, we verify the conditions.
    
    If $m$ is even, we have $c(\C P^m) = (1+u)^{m+1} \in \co^*(\C P^m;\Z)$, so $c_1(\C P^m) = m+1$. Thus, the mod 2 reduction of $c_1(\C P^m)$ is non-zero in $\co^2(\C P^m;\Z/2)$, so the condition in Theorem \ref{thm:main} (II) is satisfied.
    
    Next, $\C P^m$ is simply connected and $\text{H}_2(\C P^m;\Z) \cong \Z$ does not have 2-torsion. Moreover, if $m$ is odd, Steenrod square acts on the mod 2 reduction of the generator $u^{m-2} \in \co^{2m-4}(\C P^m;\Z)$ by the rule
        \begin{equation*}
            Sq^{2}(u^{m-2}) = (m-2) u^{m-1} \neq 0 \in \co^{2m-2}(\C P^m; \Z/2).
        \end{equation*}
    We conclude that the conditions in Theorem \ref{thm:main} (III) are also satisfied, which finishes the proof.
\end{proof}

\section{Main Theorems}\label{sec:main_thms}

In this section, we state the most comprehensive form of our main results. Due to technical differences, we split the statements according to the parity of the complex rank.

Let $X$ be a finite $2m$-dimensional CW complex. For a rank $m$ complex bundle $\xi$ and line bundles $\ell_1, \dots, \ell_r$ over $X$, the total Chern class of the stable bundle $\xi-\ell_1 \oplus \dots \oplus \ell_r$ is given by
\begin{equation*}
    c(\xi-\ell_1 \oplus \dots \oplus \ell_r) = c(\xi) \cdot c(\ell_1)^{-1} \cdot {\dots} \cdot c(\ell_r)^{-1} \in \co^*(X;\Z).
\end{equation*}
Notice that $c(\ell_i)=1+c_1(\ell_i)$ is indeed invertible with $c(\ell_i)^{-1} = \sum_{j\ge 0} (-c_1(\ell_i))^j$ since $c_1(\ell_i)$ is nilpotent. In other words, we can write the the $n$-th virtual Chern class as
\[
    c_n(\xi - \ell_1 \oplus \cdots \oplus \ell_r) = \sum_{j=0}^n c_{n-j}(\xi)h_j(-c_1(\ell_1), \dots, -c_1(\ell_r)), \text{ for } n \ge 0,
\]
where $h_j$ denotes the complete homogeneous symmetric polynomial of degree $j$. That is, there is always a factorization $c(\xi) = c(\ell_1) \cdots c(\ell_r)x \in \co^*(X;\Z)$, where $x=c(\xi-\ell_1 \oplus \dots \oplus \ell_r)$. Moreover, in view of \eqref{eq:necessary condition} and Question \ref{que:fundamental}, if we are given such a factorization of \(c(\xi)\), then  \(x \in \co^{\le 2(m-r)}(X;\Z)\) if and only if
\begin{equation*}
    c_{m-r+1}(\xi-\ell_1 \oplus \dots \oplus \ell_r) = \dots = c_{m}(\xi-\ell_1 \oplus \dots \oplus \ell_r) = 0 \in \co^*(X;\Z).
\end{equation*}
In the rest of the section, we will formulate our main results for splitting off up to three line bundles in terms of the vanishing of the aforementioned top virtual Chern classes. 

\subsection{Splitting off a single complex line bundle}
We state a generalization of the complex analogue of Hopf's theorem, which states that the Euler characteristic $\chi(M)$ of a manifold $M$ determines the existence of a non-vanishing vector field on $M$. Our result also provides a necessary and sufficient condition for $\cpspan(\xi) \ge 1$.

\begin{prop}\label{prop:complexhopf}
    Let \(\xi \) be a rank \(m\) complex vector bundle over a \(2m\)-dimensional CW complex \(X\), where \(m \ge 1.\) If $\ell$ is a complex line bundle over $M$, then
    \begin{equation*}
        \ell \text{ embeds in } \xi \text{ if and only if } c_m(\xi-\ell) =0.
    \end{equation*}
    In particular, if \(X\) is an almost-complex manifold \(M\), then \(\cspan(M) \geq 1\) if and only if the Euler characteristic \(\chi(M)\) of \(M\) vanishes.
\end{prop}

\subsection{Splitting off two complex line bundles}\label{subsec:twolines}
In this section we state necessary and sufficient conditions for splitting off two complex line bundles from a complex vector bundle, thus providing (under certain restrictions) necessary and sufficient conditions for $\cpspan(\xi) \ge 2$. In what follows, \[\rho_k \colon \co^*(X;\Z)\to \co^*(X;\Z/k)\] will denote the map induced by reduction mod \(k\) (see also Section \ref{sec:prelims}); also, we write \(\delta\) for the Bockstein associated to the coefficient sequence \(0 \to \Z \to \Z \to \Z/2 \to 0\), c.f., Notation \ref{notate:cohomologyops}(ii).

We first consider the case when the complex rank of the bundle is odd. 

\begin{thm}\label{thm:cpspan2_odd}
Let \(X\) be a \emph{CW} complex of dimension \(2m\), where \(m\ge 3\) is an odd integer. Let \(\xi\) be a complex \(m\)-plane bundle over \(X\) and let \(\ell_1,\ell_2\) be complex line bundles over \(X.\) 

\begin{enumerate}[\normalfont(i)]
\item Suppose that 
\(
\co^{2m}(X;\Z) = \delta Sq^2 \rho_2 \co^{2m-3}(X;\Z).
\)
Then 
\begin{equation*}
    \ell_1 \oplus \ell_2 \text{ embeds in } \xi \text{ if and only if } c_{m-1}(\xi-\ell_1\oplus \ell_2)=0.
\end{equation*} 
\item Suppose that \(\co^{2m}(X;\Z)\) has no \(2\)-torsion. Then
\begin{equation*}
    \ell_1 \oplus \ell_2 \text{ embeds in } \xi \text{ if and only if } c_{m+1-i}(\xi-\ell_1\oplus \ell_2)=0, \hspace{3mm} \textrm{for}~i=1,2.
\end{equation*}
\end{enumerate}
\end{thm}

In the case of a manifold, the following is a direct application of Theorem \ref{thm:cpspan2_odd}(ii).

\begin{cor}\label{cor:cpspan2_odd}
Let \(M\) be a smooth closed connected manifold of dimension \(2m\) with an almost-complex structure, where \(m\ge 3\) is an odd integer. Let \(TM\) be the tangent bundle of \(M\) and let \(\ell_1,\ell_2\) be complex line bundles over \(M.\) Then 
\begin{equation*}
    \ell_1 \oplus \ell_2 \text{ embeds in } TM \text{ if and only if } c_{m+1-i}(TM-\ell_1\oplus \ell_2)=0, \hspace{3mm} \textrm{for}~i=1,2.
\end{equation*}
\end{cor}

Moreover, taking \(\ell_1\) and \(\ell_2\) to be trivial complex line bundles, Theorem \ref{thm:cpspan2_odd} specializes to a result about complex span of manifolds, recovering \cite[Corollary 3.6, p.\ 192]{thomas1967postnikov}.

\begin{cor}
Let \(M\) be a smooth closed connected manifold of dimension \(2m\) with an almost-complex structure, where \(m \ge 3\) is an odd integer. Then
\begin{equation*}
    \span_{\C}(M) \geq 2 \text{ if and only if } c_{m+1-i}(M)=0, \hspace{3mm} \textrm{for}~i=1,2.
\end{equation*}
\end{cor}

We now state our results in the case when the complex rank of the bundle is even.

\begin{thm}\label{thm:cpspan2_even}
Let \(X\) be a \emph{CW} complex of dimension \(2m\), where \(m \ge 3\) is an even integer. Let \(\xi\) be a complex \(m\)-plane bundle over \(X\) and let  \(\ell_1,\ell_2\) be complex line bundles over \(X.\) Suppose that 
\begin{itemize}
    \item \(\co^{2m-1}(X;\Z/2) = Sq^2 \rho_2 \co^{2m-3}(X;\Z)\); 
    \item \(\co^{2m}(X;\Z/2) = Sq^2\co^{2m-2}(X;\Z/2)\); and 
    \item \(\co^{2m}(X;\Z)\) has no \(2\)-torsion.
\end{itemize}
Then \begin{equation*}
    \ell_1 \oplus \ell_2 \text{ embeds in } \xi \text{ if and only if } c_{m+1-i}(\xi-\ell_1\oplus \ell_2)=0, \hspace{3mm} \textrm{for}~i=1,2.
\end{equation*}
\end{thm}

In the case of a manifold, we have the following simpler statement.

\begin{cor}\label{cor:cpspan2_even}
Let \(M\) be a smooth closed simply-connected manifold of dimension \(2m\) with an almost-complex structure, where \(m \ge 3\) is an even integer. Fix line bundles \(\ell_1,\ell_2\) over \(M\). If \(w_2(M)\neq 0\), then 
\begin{equation*}
    \ell_1 \oplus \ell_2 \text{ embeds in } TM \text{ if and only if } c_{m+1-i}(TM-\ell_1\oplus \ell_2)=0, \hspace{3mm} \textrm{for}~i=1,2.
\end{equation*}
\end{cor}
\begin{proof}
Poincar\'e duality and the connectivity assumption yield $\co^{2m-1}(M;\Z/2) \cong \mathrm{H}_1(M;\Z/2) \cong 0$, so the first condition in Theorem \ref{thm:cpspan2_even} holds.
We have that \(Sq^2 x = v_2(M) x\) for all \(x \in \co^{2m-2}(M;\Z/2),\) where \(v_2(M) \in \co^{2}(M;\Z/2)\) is the second Wu class of \(M.\) But \(M\) is orientable by assumption, so \(v_2(M) = w_2(M).\) Hence by Theorem \ref{thm:cpspan2_even}, we require that \[Sq^2 \co^{2m-2}(M;\Z/2) = w_2(M) \co^{2m-2}(M;\Z/2)\] is isomorphic to \(\co^{2m}(M;\Z/2) \cong \Z/2\); but this happens precisely when \(w_2(M) \neq 0.\) Lastly, since \(\co^{2m}(M;\Z)\) is torsion free,
the corollary follows.
\end{proof}

Moreover, taking \(\ell_1\) and \(\ell_2\) to be trivial complex line bundles, Theorem \ref{thm:cpspan2_even} specializes to a (seemingly new) result about the complex span of non-spin almost-complex manifolds.

\begin{cor}\label{cor:cpspan2_even}
Let \(M\) be a smooth closed simply-connected manifold of dimension \(2m\) with an almost-complex structure, where \(m \ge 3\) is an even integer. If \(w_2(M) \neq 0\), then 
\begin{equation*}
    \span_{\C}(M) \geq 2 \text{ if and only if } c_{m+1-i}(M)=0, \hspace{3mm} \textrm{for}~i=1,2.
\end{equation*}
\end{cor}

\subsection{Splitting off three complex line bundles}\label{subsec:threelines}
In this section we state necessary and sufficient conditions for splitting off three complex line bundles from a complex vector bundle. We first consider the case of bundles with even complex rank.

In what follows, we will denote by $\pi_{2m-2}$ the $(2m-2)$-th homotopy group of the Stiefel manifold $W(m,3)$. By Table (\ref{table:pi2m-2W(m,3)}), the group $\pi_{2m-2}$ is finite abelian with order dividing 12, so $12/
\lvert \pi_{2m-2} \rvert \in \Z$.

\begin{thm}\label{thm:cpspan3_even}
Let \(X\) be a \emph{CW} complex of dimension \(2m\), where \(m \geq 5\) is an even integer. Let \(\xi\) be a complex \(m\)-plane bundle over \(X\) and let \(\ell_1,\ell_2,\ell_3\) be complex line bundles over \(X.\)
\begin{enumerate}[\normalfont(I)]
\item Suppose that the following conditions \emph{(i)--(iii)} hold.  
\begin{enumerate}[\normalfont(i)]
    \item \(
    \co^{2m-2}(X;\Z) = \delta Sq^2 \rho_2 \co^{2m-5}(X;\Z)
    \);
    \item One of the following conditions \emph{(a)--(d)} holds: 
\begin{enumerate}[\normalfont(a)]
\item \(m \equiv 0~\mathrm{mod}~8\) and \(\co^{2m-1}(X;\Z/4) = 0\); or 
\item \(m \equiv 2~\mathrm{mod}~8\); or 
\item \(m \equiv 4~\mathrm{mod}~8\) and \(\co^{2m-1}(X;\Z/2) = 0\); or 
\item \(m \equiv 6~\mathrm{mod}~8\).
\end{enumerate}
    \item \(\co^{2m}(X;\Z)\) has no \(n\)-torsion for \(n = 12/ \lvert \pi_{2m-2} \rvert \in \Z\).
\end{enumerate}
Then
\vspace{-1mm}
\begin{equation*}
    \ell_1 \oplus \ell_2 \oplus \ell_3 \text{ embeds in } \xi \text{ if and only if } c_{m+1-i}(\xi-\ell_1\oplus \ell_2 \oplus \ell_3)=0, \hspace{3mm} \textrm{for}~i=1,3.
\end{equation*}
\item Suppose that the following conditions \emph{(i)--(iii)} hold.  
\begin{enumerate}[\normalfont(i)]
    \item \(\co^{2m-2}(X;\Z)\) has no \(2\)-torsion;
    \item One of the following conditions \emph{(a)--(d)} holds: 
\begin{enumerate}[\normalfont(a)]
\item \(m \equiv 0~\mathrm{mod}~8\) and \(\co^{2m-1}(X;\Z/4) = 0\); or 
\item \(m \equiv 2~\mathrm{mod}~8\); or 
\item \(m \equiv 4~\mathrm{mod}~8\) and \(\co^{2m-1}(X;\Z/2) = 0\); or 
\item \(m \equiv 6~\mathrm{mod}~8\).
\end{enumerate}
    \item \(\co^{2m}(X;\Z)\) has no \(n\)-torsion for \(n = 12/ \lvert \pi_{2m-2} \rvert \in \Z\). 
\end{enumerate}
Then
\vspace{-1mm}
\begin{equation*}
    \ell_1 \oplus \ell_2 \oplus \ell_3 \text{ embeds in } \xi \text{ if and only if } c_{m+1-i}(\xi-\ell_1\oplus \ell_2 \oplus \ell_3)=0, \hspace{3mm} \textrm{for}~i=1,2,3.
\end{equation*} 
\end{enumerate}
\end{thm}

We note that Theorem \ref{thm:cpspan3_even}(II) specializes to recover a result about the complex span of manifolds in \cite[Theorem 1 \& Corollary, p.\ 633]{gilmore1967complex}.\\

\begin{cor}
Let \(M\) be a smooth closed manifold of dimension \(2m\) with an almost-complex structure, where \(m \geq 5\) is even. Further let \(\ell_1,\ell_2,\ell_3\) be complex line bundles over \(M\).
\begin{enumerate}[\normalfont(i)]
    \item Suppose that \(m \equiv 2~ \mathrm{mod}~4\). If \(\co_2(X;\Z)\) has no \(2\)-torsion, then 
\[
    \ell_1 \oplus \ell_2 \oplus \ell_3 \text{ embeds in } TM \text{ if and only if } c_{m+1-i}(TM-\ell_1\oplus \ell_2 \oplus \ell_3)=0, \hspace{3mm} \textrm{for}~i=1,2,3.
\]
    \item Suppose that \(m \equiv  0 ~\mathrm{mod}~4.\) Assume that \(M\) is simply-connected and that \(\co_{2}(X;\Z)\) has no \(2\)-torsion. Then
\[
    \ell_1 \oplus \ell_2 \oplus \ell_3 \text{ embeds in } TM \text{ if and only if } c_{m+1-i}(TM-\ell_1\oplus \ell_2 \oplus \ell_3)=0, \hspace{3mm} \textrm{for}~i=1,2,3.
\]
\end{enumerate}
In particular, if \(M\) is simply connected and \(\co_2(M;\Z)\) has no \(2\)-torsion, then
\[
\cspan(M) \geq 3 \text{ if and only if } c_{m+1-i}(M) = 0, \text{ for } i = 1,2,3.
\]
\end{cor}

In the case when the complex rank of the bundle is odd, we prove the following.

\begin{thm}\label{thm:cpspan3_odd}
Let \(X\) be a \emph{CW} complex of dimension \(2m\), where \(m \geq 5\) is an odd integer. Let \(\xi\) be a complex \(m\)-plane bundle over \(X\), and let \(\ell_1,\ell_2,\ell_3\) be complex line bundles over \(X.\) Further suppose that 
\begin{enumerate}[\normalfont(i)]
    \item \(\co^{2m-3}(X;\Z/2) = Sq^2 \rho_2 \co^{2m-5}(X;\Z)\);
    \item \(\co^{2m-2}(X;\Z/2) = Sq^2 \co^{2m-4}(X;\Z/2)\)
    \item \(\co^{2m-2}(X;\Z)\) has no \(2\)-torsion; 
    \item \(\co^{2m-1}(X;\Z)\) is finite abelian with no \(2\)-torsion; and 
    \item \(\co^{2m}(X;\Z)\) is torsion free.
\end{enumerate} Then 
\begin{equation*}
    \ell_1 \oplus \ell_2 \oplus \ell_3 \text{ embeds in } \xi \text{ if and only if } c_{m+1-i}(\xi-\ell_1\oplus \ell_2 \oplus \ell_3)=0, \hspace{3mm} \textrm{for}~i=1,2,3.
\end{equation*}
\end{thm}

Together, Theorems \ref{thm:cpspan3_even} and \ref{thm:cpspan3_odd} recover a result of Thomas \cite[Corollary 1.6]{thomas1967real} and Gilmore \cite[Theorem 1]{gilmore1967complex} about complex non-vanishing vector fields on almost-complex manifolds. Namely, orientability of the manifold and the connectivity assumption yield the following result.

\begin{cor}\label{cor:cspan3_manifold}
Let \(M\) be a \(2m\)-dimensional, \(3\)-connected almost-complex manifold, where \(m \geq 5.\) Then \(\span_{\C}(M) \geq 3\) if and only if \(c_{m-2}(M) = \chi(M)= 0\).
\end{cor}

\section{Preliminaries}\label{sec:prelims}
In this section we introduce some notation and develop notions used throughout the text.

\subsection{Some essentials of Moore--Postnikov theory }\label{subsec:mptheory}
For a classical reference, we refer the reader to \cite{thomas2006seminar}. Consider the lifting problem 
\begin{equation}\label{eq:liftingprob}
\begin{tikzcd}
	& E \\
	X & B
	\arrow["{q_0}", from=1-2, to=2-2]
	\arrow["f"', from=2-1, to=2-2]
\end{tikzcd}\end{equation}
subject to the following assumptions:
\begin{enumerate}[(i)]
    \item The space \(X\) is a CW complex. 
    \item The map \(q_0 \colon E \to B\) is a (Hurewicz) fibration with simply connected fiber \(F\) and simply connected base \(B.\)
\end{enumerate}
Write \(\pi_i \coloneqq \pi_{n_i}(F)\) for the nonzero homotopy groups of \(F\), where \((n_i)_{i \in \mathbb{N}}\) is a sequence of integers such that \(n_1 >1\) and \(n_i < n_{i+1}\) for all \(i \in \mathbb{N}\).

To solve the lifting problem (\ref{eq:liftingprob}), we construct a \emph{Moore--Postnikov tower} for \(q_0 \colon E \to B\), which is a homotopy commutative diagram
\begin{equation}\label{eq:factorwithkinvs}
\begin{tikzcd}
	&& \vdots \\
	&& {E[2]} & {K(\pi_3,n_3+1)} \\
	&& {E[1]} & {K(\pi_2,n_2+1)} \\
	F & E & B & {K(\pi_1,n_1+1)} \\
	&& X
	\arrow[hook, from=4-1, to=4-2]
	\arrow["{q_0}"{description}, from=4-2, to=4-3]
	\arrow["{k_1}", from=4-3, to=4-4]
	\arrow["{k_2}", from=3-3, to=3-4]
	\arrow["{p_1}", from=3-3, to=4-3]
	\arrow["{p_2}", from=2-3, to=3-3]
	\arrow["{q_1}"{description}, from=4-2, to=3-3]
	\arrow["{q_2}"{description}, from=4-2, to=2-3]
	\arrow["{k_3}", from=2-3, to=2-4]
	\arrow[from=1-3, to=2-3]
	\arrow["f", from=5-3, to=4-3]
\end{tikzcd}\end{equation}
where 
\begin{itemize}
    \item each \(k_i\) is the \emph{characteristic class of the fibration} \(q_{i-1}\)---i.e., each \(k_i\) is the transgression of the fundamental class in the fibration \(q_{i-1}\); 
    \item each \(E[i]\) is the homotopy fiber of \(k_i\); and 
    \item each \(q_i \colon E \to E[i]\) is an \(n_{i +1}\)-equivalence. 
\end{itemize}

The following proposition illustrates the upshot of such a factorization (\ref{eq:factorwithkinvs}).

\begin{prop}\label{prop:upshot}
Let \(q_0 \colon E \to B\) and \(f \colon X \to B\) be as in \emph{(\ref{eq:liftingprob})}. Let \emph{(\ref{eq:factorwithkinvs})} be a Moore--Postnikov factorization of \(q_0\), and assume that \(X\) is a CW complex with \({\mathrm{dim}}(X) \leq n_r\). Then \(f \colon X \to B\) lifts to \(E\) in \emph{(\ref{eq:liftingprob})} if and only if \(f\) lifts to \(E[r]\) in \emph{(\ref{eq:factorwithkinvs})}. 
\end{prop}
\begin{proof}
Suppose that \(g \colon X \to E\) is a solution to (\ref{eq:liftingprob}). Then \(q_r \circ g \colon X \to E[r]\) is a lift of \(f\) to \(E[r]\). Conversely suppose that \(g \colon X \to E[r]\) is a map such that \[p_1 \circ p_2 \circ \cdots \circ p_{r-1}\circ p_r \circ g \simeq f.\] Then it follows from \cite[Corollary 7.6.23]{spanier1995algebraic} that the induced map \[q_r \circ - \colon [X,E] \to [X,E[r]]\] is bijective (and in particular surjective) since \(\mathrm{dim}(X) < n_{r +1}\) and \(q_r\) is an \((n_{r + 1})\)-equivalence.
\end{proof}

We now set notation to describe the various obstructions to lifting \(f\) to \(E[r-1].\)
\begin{notate}
Let \(f \colon X \to B\) and \(q_0 \colon E \to B\) be given as in (\ref{eq:liftingprob}). Given a factorization (\ref{eq:factorwithkinvs}) for \(q_0\), we write 
\[
\lobs^{n_1+1}(f,q_0,k_1) = \{f^*k_1\} \subseteq \co^{n_1+1}(X;\pi_1)
\] for the \emph{primary obstruction} to lifting \(f\). 
Assuming that \(\lobs^{n_1+1}(f,q_0,k_1) = \{0\}\), there exists a map \(g \colon X \to E[1]\) such that \(p_1 \circ g \simeq f.\) We write 
\[
\lobs^{n_2+1}(f,q_0,k_2) \coloneqq \bigcup_{g \colon X \to E[1]}g^*(k_2) \subseteq \co^{n_2+1}(X;\pi_2)
\]
for the \emph{generalized secondary obstruction} to lifting \(f.\)
Inductively, we denote by
\[
\lobs^{n_i+1}(f,q_0,k_{i}) \coloneqq \bigcup_{g \colon X \to E[i-1]}g^*(k_i) \subseteq \co^{n_i+1}(X;\pi_i)
\]
the \emph{generalized \(i\)-th obstruction} to lifting \(f\), assuming that \(0 \in \lobs^{n_j+1}(f,q_0,k_j)\) for all \(j=1, \dots, i-1.\)
\end{notate}

Of course, computing higher obstructions is subtle, difficult, and even impossible most of the time. Nonetheless, some hope remains in the form of the proceeding proposition. Essentially, since the vertical fibrations \(p_i\) in (\ref{eq:factorwithkinvs}) are actually principal fibrations, any two (homotopy) lifts \(f_1 \colon X \to E[i]\) and \(g_1\colon X \to E[i]\) of \(f\) differ by an action of \(\Omega K(\pi_i,n_i + 1)\). More precisely, we can understand the difference between lifts of \(f\) to \(E[i]\) using the following result, c.f., \cite[Chapter II]{thomas2006seminar} or \cite[Proposition, p.\ 218]{pollina1982tangent}.

\begin{prop}\label{prop:difflifts}
Let \(f_1, g_1 \colon X \to E[i]\) be two lifts of \(f\) such that \[p_1 \circ p_2\circ \cdots \circ p_i \circ f_1 \simeq f \simeq p_1 \circ p_{2}\circ \cdots \circ p_i \circ g_1.\] Then there exists a map \(\alpha \colon X \to \Omega K(\pi_i,n_i+1) \cong K(\pi_i, n_i)\) such that the composition
\[
X \overset{\Delta} \longrightarrow X \times X \overset{f_1 \times \alpha}\longrightarrow E[i] \times K(\pi_i,n_i) \overset{\mu}\longrightarrow E[i]
\]
is homotopic to \(g_1\). Here, we denote by \(\Delta\) the diagonal map and by \(\mu\) the principal action map associated to the total space of the principal fibration \(K(\pi_i,n_i) \longrightarrow E[i] \overset{p_i}\longrightarrow E[i-1].\)
\end{prop}

Finally, we will freely make use of the following information without further mention, as it is often useful for computing the second \(k\)-invariant.

\begin{rem}\label{rem:mpinvrestrict}
Let \(F \xhookrightarrow{} E \overset{q_0}\to B\) be as in (\ref{eq:factorwithkinvs}). Consider the fibration \(p_1 \colon E[1] \to B\) with fiber \(K = K(\pi_1, n_1)\) in the Moore--Postnikov factorization (\ref{eq:factorwithkinvs}) of \(q_0\). Write \(i_1 \colon K \xhookrightarrow{}E[1]\) for the fiber inclusion. Then the Moore--Postnikov invariant \(k_2\) is characterized by the following properties, see \cite[(3.2), p.\ 190]{thomas1967postnikov} and \cite[Theorem 4.1]{hermann1959secondary}. 
\begin{enumerate}
    \item[(i)] \(q_1^*(k_2) = 0\); and 
    \item[(ii)] \(i_1^*(k_2) \in \co^{n_2+1}(K;\pi_2)\) coincides with the bottom Postnikov invariant in the Postnikov tower of the space \(F\). 
\end{enumerate}
\end{rem}

\subsection{A lemma about homotopy fibers}
We will frequently make use of the following fact, see \cite[Lemma 6.1]{thomas1967fields}.

\begin{lem}[Thomas diagram]\label{lem:thomas_diagram}
Consider the following commutative diagram of pointed spaces. 
\begin{equation}\label{eq:commutative_square}\begin{tikzcd}
	A & B \\
	C & D
	\arrow["{u_1}", from=1-1, to=1-2]
	\arrow["{v_0}"', from=1-1, to=2-1]
	\arrow["{v_1}", from=1-2, to=2-2]
	\arrow["{u_0}", from=2-1, to=2-2]
\end{tikzcd}\end{equation}
Denote by \(\mathrm{hofib}(x)\) the standard homotopy fiber of \(x \in \{u_0,u_1,v_0,v_1\}\). Then there are canonical continuous functions \(u \colon \mathrm{hofib}(v_0) \to \mathrm{hofib}(v_1)\) and \(v \colon \mathrm{hofib}(u_1) \to \mathrm{hofib}(u_0)\) between the homotopy fibers, as well as a natural homeomorphism \(\mathrm{hofib}(u) \approx \mathrm{hofib}(v).\)
Thus there is a homotopy commmutative diagram
\begin{equation}\label{eq:thomas_diagram}
\begin{tikzcd}
	F & {\mathrm{hofib}(v_0)} & {\mathrm{hofib}(v_1)} \\
	{\mathrm{hofib}(u_1)} & A & B \\
	{\mathrm{hofib}(u_0)} & C & D
	\arrow[from=1-1, to=1-2]
	\arrow[from=1-1, to=2-1]
	\arrow["u", from=1-2, to=1-3]
	\arrow[from=1-2, to=2-2]
	\arrow[from=1-3, to=2-3]
	\arrow[from=2-1, to=2-2]
	\arrow["v"', from=2-1, to=3-1]
	\arrow["{u_1}", from=2-2, to=2-3]
	\arrow["{v_0}"', from=2-2, to=3-2]
	\arrow["{v_1}", from=2-3, to=3-3]
	\arrow[from=3-1, to=3-2]
	\arrow["{u_0}", from=3-2, to=3-3]
\end{tikzcd}
\end{equation}
where \(F\) denotes the homotopy fiber of \(u\) and \(v\).
\end{lem}
Henceforth, we will refer to the diagram (\ref{eq:thomas_diagram}) as the \emph{Thomas diagram associated to} (\ref{eq:commutative_square}). 

\subsection{Cohomology of certain Eilenberg--MacLane Spaces}
In order to carry out the obstruction theory in subsequent sections, we will require (co)homological information about certain Eilenberg--MacLane spaces with various coefficients. We collect the essential results in this section and set notation. \sloppy

\begin{notate}\label{notate:cohomologyops}
We set the following notation. 
\begin{enumerate}[(i)]
    \item The mapping \(\rho_k \colon \co^*(X;\Z)\to \co^*(X;\Z/k)\) is induced by reduction mod \(k.\) For fixed \(k \geq2\), we will write only \(\rho\) instead of \(\rho_k.\)
    \item Let \(k \geq 2.\) We write \(\delta_k\) for the Bockstein homomorphism associated with the short exact sequence 
    \[
    0 \longrightarrow \Z \overset{\cdot k}\longrightarrow \Z \overset{\rho_k}\longrightarrow \Z/k \longrightarrow 0.
    \]
    For fixed \(k \geq2\), we will write only \(\delta\) instead of \(\delta_k.\)
    
    Additionally, given also \(\ell \geq 2\), we write \(\delta_{k}^{\ell}\) for the composition \(\rho_{2^\ell} \circ \delta_k \colon \co^*(\cdot; \Z/k) \to \co^{*+1}(\cdot;\Z/2^\ell)\).
    \item The mapping \(\theta_2^k \colon \co^*(\cdot;\Z/2) \to \co^*(\cdot;\Z/2^k)\) is induced from the inclusion \(\Z/2 \to \Z/2^k\).
    \item For all \(i \geq 0\), we write 
    \[
    Sq^i \colon \co^*(\cdot;\Z/2) \to \co^{*+i}(\cdot;\Z/2)
    \]
    for the Steenrod squares. 
    \item Let \(p > 2\) be prime. For all \(i\geq 0\), we write 
    \[
    P_p^i \colon \co^{*}(\cdot;\Z/p) \to \co^{*+2i(p-1)}(\cdot;\Z/p)
    \]
    for the reduced \(p\)-th power operations.

    \item Let \(p \geq 2\) be prime. We will write \(\widetilde{\beta}_p\) for the Bockstein homomorphism associated to the short exact sequence 
    \[
    0 \longrightarrow \Z/p \longrightarrow \Z/p^2 \longrightarrow \Z/p \longrightarrow 0.
    \] Moreover, we denote by \(\beta_p\) the operation given by \(\beta_p(x) = (-1)^{\mathrm{dim}(x)}\widetilde{\beta}_p(x)\).
    \item We set the following notation for the fundamental classes (see e.g., \cite{mosher2008cohomology}) of Eilenberg--MacLane spaces to be used throughout the paper.
    \begin{itemize}
        \item Write \(\iota_{q} \in \co^q(K(\Z,q);\Z)\) for the integral fundamental class.
        \item Let \(\iota_q^k \coloneqq \rho_k \iota_q \in \co^q(K(\Z,q);\Z/k)\) be the mod \(k\) reduction of \(\iota_q\). When the context is clear, we omit the subscript (emphasizing the dimension) and/or the superscript (emphasizing the coefficients).
        \item Furthermore, denote by \(\kappa_q \in \co^q(K(\Z/k,q);\Z/k)\) the mod \(k\) fundamental class.
    \end{itemize}
\end{enumerate}
\end{notate}

\begin{prop}\label{prop:cohomK(G,q)}
Let \(G\) be any abelian group. The first six nonzero integral homology and cohomology groups of \(K(G,q)\) satisfy \emph{Table \ref{table:cohomK(G,q)}.} Here, \({_pG}\) denotes the subgroup of \(G\) consisting of all elements of order \(p\). 
\end{prop}

\begin{table}[!h]
\centering
\begin{tabular}{@{}cclcc@{}}
\toprule
\(i\)       & \(\co_i(K(G,q);\Z)\)  &                                                 & \(\co^i(K(G,q);\Z)\)                          & \(q\) \\ \midrule
\(i= q\)    & \(G\)                 & \cite[Theorem 20.4]{eilenbergmaclane1954groups} & \(\mathrm{Hom}_{\Z}(G,\Z)\)                   & \(q >0\)\\
\(i= q+1\)  & \(0\)                 & \cite[Theorem 20.5]{eilenbergmaclane1954groups} & \(\mathrm{Ext}_{\Z}^1(G,\Z)\)                & \(q >1\) \\
\(i= q+2\)  & \(G/2G\)              & \cite[Theorem 23.1]{eilenbergmaclane1954groups} & \(0\)                                        & \(q >2\) \\
\(i= q+3\)  & \({_2G}\)             & \cite[Theorem 24.1]{eilenbergmaclane1954groups} & \(\mathrm{Ext}_{\Z}^1(G/2G,\Z)\)    & \(q >3\)          \\
\(i = q+4\) & \(G/2G\oplus G/3G\)   & \cite[Theorem 25.1]{eilenbergmaclane1954groups} & \(\mathrm{Ext}_{\Z}^1({_2G},\Z)\)    &   \(q >4\)      \\
\(i = q+5\) & \({_2G}\oplus {_3G}\) & \cite[Theorem 25.3]{eilenbergmaclane1954groups} & \(\mathrm{Ext}_{\Z}^1(G/2G \oplus G/3G, \Z)\)& \(q >5\)\\ \bottomrule
\end{tabular}
\vspace{2mm}
\caption{The integral (co)homology of \(K(G,q)\).}
\label{table:cohomK(G,q)}
\end{table}
\begin{cor}\label{cor:cohomEMS}
The first seven nonzero integral homology and cohomology groups of \(K(\Z,q)\) satisfy \emph{Table \ref{table:cohomEMS}.}
\end{cor}

\begin{proof}
The first, second, third, and fifth rows of Table \ref{table:cohomEMS} follow from immediately from Proposition \ref{prop:cohomK(G,q)}. The fourth (resp.\ sixth) row follows from Proposition \ref{prop:cohomK(G,q)} together with \cite[Theorem 27.1]{eilenbergmaclane1954groups} (resp.\ \cite[Theorem 27.5]{eilenbergmaclane1954groups}. One obtains the last row from \cite[p.\ 662]{eilenberg1950cohomology}.
\end{proof}

\begin{table}[h!]
\centering
\begin{tabular}{@{}cccc@{}}
\toprule
\(i\)       & \(\co_i(K(\Z,q);\Z)\) & \(\co^i(K(\Z,q);\Z)\)                                                                  & \(q\) \\ \midrule
\( q\)    & \(\Z\)                & \(\Z\langle \iota_q\rangle\)   & \(q >0\)                                                           \\
\( q+1\)  & \(0\)                 & \(0\)                                                                                  & \(q >1\) \\
\( q+2\)  & \(\Z/2\)              & \(0\)                                                                                  & \(q >2\) \\
\( q+3\)  & \(0\)                 & \(\Z/2\langle \delta_2 Sq^2\iota_q^2 \rangle\)                                                & \(q \geq3\)\\
\(q+4\) & \(\Z/2\oplus \Z/3\)   & \(0\)                                                                                   & \(q >4\)\\
\(q+5\) & \(0\)                 & \(\Z/2\langle \delta_2 Sq^4\iota_q^2\rangle \oplus \Z/3\langle \delta_3 P_3^1\iota_q^3\rangle\) & \(q \geq5\)\\ 

\(q +6\)&\(\Z/2 \oplus \Z/2\) & \(0\) & \(q>6\)\\ \bottomrule
\end{tabular}
\vspace{2mm}
\caption{The integral (co)homology of \(K(\Z,q)\).}
\label{table:cohomEMS}
\end{table}

We will also need partial information about the mod \(k\) cohomology of \(K(\Z,q)\). By the universal coefficient theorem and Corollary \ref{cor:cohomEMS} we make the following computations.

\begin{prop}\label{prop:modpcohomEMS}
For \(q\) sufficiently large (as in \emph{Proposition \ref{prop:cohomK(G,q)}}), the first six nonzero cohomology groups \(\co^i(K(\Z,q);\Z/k)\) satisfy \emph{Table \ref{table:modpcohomEMS}}, for \(k \in \{2,3,4,8\}\)
\end{prop}

\begin{table}[h!]
\centering
\begin{tabular}{@{}cccccc@{}}
\toprule
\(i\)   & \(k=2\)                              & \(k=3\)                                        & \(k=4\)                                          & \(k=8\)                                         & \(q\)       \\ \midrule
\(q\)   & \(\Z/2\langle \iota_q^2\rangle\)     & \(\Z/3\langle \iota_q^3\rangle\)               & \(\Z/4\langle \iota_q^4\rangle\)                 & \(\Z/8\langle \iota_q^8\rangle\)                & \(q>0\)     \\
\(q+1\) & \(0\)                                & \(0\)                                          & \(0\)                                            & \(0\)                                           & \(q>1\)     \\
\(q+2\) & \(\Z/2\langle Sq^2\iota_q^2\rangle\) & \(0\)                                          & \(\Z/2\langle \theta_2^2Sq^2\iota_q^2\rangle\)   & \(\Z/2\langle \theta_2^3Sq^2\iota_q^2\rangle\)  & \(q>2\)     \\
\(q+3\) & \(\Z/2\langle Sq^3\iota_q^2\rangle\) & \(0\)                                          & \(\Z/2\langle \delta_2^2Sq^2\iota_q^2\rangle\)   & \(\Z/2\langle \delta_2^3Sq^2 \iota_q^2\rangle\) & \(q\geq 3\) \\
\(q+4\) & \(\Z/2\langle Sq^4\iota_q^2\rangle\) & \(\Z/3\langle P_3^1\iota_q^3\rangle\)          & \(\Z/2\langle \theta_2^2Sq^4\iota_q^2\rangle\)   & \(\Z/2\langle\theta_2^3Sq^4 \iota_q^2\rangle\)  & \(q>4\)     \\
\(q+5\) & \(\Z/2\langle Sq^5\iota_q^2\rangle\) & \(\Z/3\langle \beta_3 P_1^3 \iota_q^3\rangle\) & \(\Z/2\langle \delta_2^2 Sq^4 \iota_q^2\rangle\) & \(\Z/2\langle\delta_2^3Sq^4 \iota_q^2\rangle\)  & \(q\geq5\)  \\ \bottomrule
\end{tabular}
\caption{The cohomology groups \(\co^i(K(\Z,q);\Z/k)\) for \(k \in \{2,3,4,8\}\).}
\label{table:modpcohomEMS}
\end{table}

We remark that the \(k =2\) column of Table \ref{table:modpcohomEMS} was classically computed by Serre, see e.g., \cite[Theorem 3, p.\ 90]{mosher2008cohomology}; moreover, the \(k=3\) column was computed by Cartan, see e.g., \cite[p.\ 414]{fomenko2016homotopical}. 

Finally, we will need partial information about the integral cohomology of \(K(\Z/k,q)\). We compute the following groups using Proposition \ref{prop:cohomK(G,q)} and the universal coefficient theorem.

\begin{prop}\label{prop:cohomK(Z/p,q)}
Let \(G= \Z/k\) for \(k \in \{2,3,4,8\}.\) For \(q\) sufficiently large (as in \emph{Proposition \ref{prop:cohomK(G,q)}}), the first three integral (co)homology groups of \(K(G,q)\) satisfy \emph{Table \ref{table:cohomK(Z/p,q)}}. 
\end{prop}

\begin{table}[h!]
\centering
\begin{tabular}{@{}ccccc@{}}
\toprule
\(\co_i(K(G,q);\Z)\) & \(G= \Z/2\)                           & \(G=\Z/3\)                            & \(G= \Z/4\)                           & \(G=\Z/8\)                             \\ \midrule
\(i=q\)              & \(\Z/2\)                              & \(\Z/3\)                              & \(\Z/4\)                              & \(\Z/8\)                               \\
\(i=q+1\)            & \(0\)                                 & \(0\)                                 & \(0\)                                 & \(0\)                                  \\
\(i=q+2\)            & \(\Z/2\)                              & \(0\)                                 & \(\Z/2\)                              & \(\Z/2\)                               \\ \midrule
\(\co^i(K(G,q);\Z)\) & \(G= \Z/2\)                           & \(G=\Z/3\)                            & \(G= \Z/4\)                           & \(G=\Z/8\)                             \\
\(i=q\)              & \(0\)                                 & \(0\)                                 & \(0\)                                 & \(0\)                                  \\
\(i=q+1\)            & \(\Z/2\langle \delta_2 \kappa_q\rangle\) & \(\Z/3\langle \delta_3 \kappa_q\rangle\) & \(\Z/4\langle \delta_4 \kappa_q\rangle\) & \(\Z/8\langle \delta_8 \kappa_q \rangle\) \\
\(i=q+2\)            & \(0\)                                 & \(0\)                                 & \(0\)                                 & \(0\)                                  \\ \bottomrule
\end{tabular}
\vspace{2mm}
\caption{The integral (co)homology groups of \(K(\Z/k,q)\), for \(k \in \{2,3,4,8\}\).}
\vspace{-4mm}
\label{table:cohomK(Z/p,q)}
\end{table}

\subsection{Some homotopy groups of complex Stiefel manifolds}
For \(m \geq r\), let \(W(m,r)\) denote the complex Stiefel manifold of \(r\)-frames in \(\C^m.\) Recalling that \(W(m,r)\) is \(2(m-r)\)-connected, we collect relevant homotopy groups of \(W(m,r)\).

\begin{prop}\label{prop:piW(m,r)}
Set \(\pi_{2m-2} \coloneqq \pi_{2m-2}(W(m,3)).\)
\begin{enumerate}[\normalfont(i)]
\item For \(r = 2,3\), the homotopy groups \(\pi_{2(m-r)+i_r}(W(m,r))\) satisfy \emph{Table \ref{table:pi(W(m,2))}} and \emph{Table \ref{table:pi(W(m,3))}}, where \(i_r =1,2,\dots, 2r-1\).
\item Additionally, the homotopy group \(\pi_{2m-2}\) satisfies \emph{Table \ref{table:pi2m-2W(m,3)}}, for \(m >4.\)
\end{enumerate}
\end{prop}

\begin{proof}
Proposition \ref{prop:piW(m,r)}(i) follows from \cite{gilmore1967complex} and we demonstrate Proposition \ref{prop:piW(m,r)}(ii). To this end, the \(p\)-primary part of \(\pi_{2m-2}\) has been computed by \cite{gilmore1967complex} for prime \(p.\) By \cite[Example 2.40, p.\ 71]{oprea2008algmodels} the minimal model of the complex Stiefel manifold \(W(m,3)\) is given by the exterior algebra 
\[
\Lambda_{\mathbb{Q}}(e_{2m-5}, e_{2m-3}, e_{2m-1}),
\]
over the rationals. Hence it follows from \cite[Theorem 2.50, p.\ 75]{oprea2008algmodels} that \(\pi_{2m-2}(W(m,3))\) has no \(\Z\)-summand.
\end{proof}

\begin{table}[h!]
\centering
\begin{tabular}{@{}ccc@{}}
\toprule
\(\pi_{2(m-2)+i_2}(W(m,2))\) & \(m > 2\) odd & \(m >2\) even     \\ \midrule
\(i_2 = 1\)        & \(\Z\)       & \(\Z\)            \\
\(i_2 = 2\)        & \(0\)        & \(\Z/2\)          \\
\(i_2 = 3\)        & \(\Z\)       & \(\Z\oplus \Z/2\) \\ \bottomrule
\end{tabular}
\vspace{2mm}
\caption{Some homotopy groups of \(W(m,2)\), for \(m >2\).}
\label{table:pi(W(m,2))}
\end{table}

\begin{table}[h!]
\centering
\begin{tabular}{ccc}
\toprule
\(\pi_{2(m-3)+i_3}(W(m,3))\) & \(m >2\) odd      & \(m >4\) even  \\ \midrule
\(i_3 = 1\)        & \(\Z\)            & \(\Z\)         \\
\(i_3 = 2\)        & \(\Z/2\)          & \(0\)          \\
\(i_3 = 3\)        & \(\Z\oplus \Z/2\) & \(\Z\)         \\
\(i_3 = 4\)        & \(\pi_{2m-2}\)    & \(\pi_{2m-2}\) \\
\(i_3 = 5\)        & \(\Z\) $\oplus \begin{cases}
    \Z/2 & m=5,\\
    0 & \text{otherwise}
\end{cases}$           & \(\Z\)         \\ \bottomrule
\end{tabular}
\vspace{2mm}
\caption{Some homotopy groups of \(W(m,3)\), for \(m >4\).}
\label{table:pi(W(m,3))}
\end{table}

\begin{table}[h!]
\centering
\begin{tabular}{@{}ccc@{}}
\toprule
\(m ~\mathrm{mod}~ 3\) & \(m ~\mathrm{mod}~ 8\) & \(\pi_{2m-2}\) \\ \midrule
\(\neq 0\)               & \(2 \text{ or } 6\)      & \(0\)                   \\
\(\neq 0\)               & \(1,4 \text{ or } 5\)    & \(\Z/2\)               \\
\(\neq 0\)               & \(0 \text{ or } 7\)      & \(\Z/4\)               \\
\(\neq 0\)               & \(3\)                    & \(\Z/8\)               \\
\(0\)                    & \(2 \text{ or } 6\)      & \(\Z/3\)               \\
\(0\)                    & \(1,4 \text{ or } 5\)    & \(\Z/3 \oplus \Z/2\)   \\
\(0\)                    & \(0 \text{ or } 7\)      & \(\Z/3 \oplus \Z/4\)   \\
\(0\)                    & \(3\)                    & \(\Z/3 \oplus \Z/8\)   \\ \bottomrule
\end{tabular}
\vspace{2mm}
\caption{The homotopy group \(\pi_{2m-2}(W(m,3))\), for \(m>4\).}
\label{table:pi2m-2W(m,3)}
\end{table}

\subsection{Relevant cohomology rings}
In this section, we establish notation for the rest of the paper, and, due to frequent use throughout the paper, we collate various information about relevant cohomology rings.

\begin{notate}\label{notate:1^k}
Let \(m \geq r\). We set the following notation.
\begin{enumerate}[(i)]
    \item As before, we write \(W(m,r)\) for the complex Stiefel manifold of \(r\)-frames in \(\C^m.\) Recall that \(W(m,r)\) may be viewed as the homogeneous space \(U(m)/U(m-r)\).
    \item We denote by 
    \[
        F(m-r,1^r) \coloneqq U(m)/(U(m-r)\times U(1)^r),
    \]
    the flag manifold,
    where $U(m-r) \times U(1)^r \subseteq U(m)$ is the diagonal inclusion.
    \item We write \(B(m,1^r)\) for the product space \(BU(m)\times BU(1)^r.\) 
    \item Let \(\gamma_m\) be the universal bundle over \(BU(m),\) and write \(c_i \coloneqq c_i(\gamma_m)\) for the universal Chern classes.
    \item We will denote by 
    \[
    a_i = c_i(\gamma_m \times 1^{r} - 1 \times \gamma_1^{\times r}) \in \co^{2i}(B(m,1^r);\Z)
    \]
    the Chern classes of the stable bundle 
    \[
        \gamma_m \times 1^{r} - 1 \times \gamma_1^{\times r} \colon B(m,1^r) \longrightarrow BU.
    \] 
    Finally, we will write \(b_i^{(p)} \coloneqq \rho_p a_i\), for the mod \(p\) reductions, where \(p >1\) is an integer. When the value of \(p\) is understood, we will omit the superscript \((p)\) from the notation.
\end{enumerate}
\end{notate}

The content of the following  proposition can be found in \cite{hsiang1984transformation}.

\begin{prop}\label{prop:cohomology_rings} Let \(m \geq r >0\) be positive integers.
\begin{enumerate}[\normalfont(i)]
    \item The integral cohomology of \(BU(m)\) satisfies \[\co^*(BU(m);\Z) \cong \Z[c_1, \dots, c_m],\] where $\lvert c_i\rvert= 2i$ and the colimit $U \coloneqq \bigcup_{m \ge 0} U(m)$ satisfies
    \[
        \co^*(BU;\Z) \cong \Z[c_1, c_2, \dots].
    \]
    \item Let \(p\) be a positive integer. Then \[\co^*(BU(m);\Z/p) \cong \Z/p[\rho_p c_1, \dots, \rho_p c_{m}].\] In particular, if \(p = 2\), then \(\rho_2 c_i = w_{2i}(\gamma_m^{\R})\), where \(\gamma_m^{\R}\) denotes the underlying real bundle of the universal bundle \(\gamma_m.\)
    \item Let \(p\) be a positive integer and \(R\) a finitely generated abelian group. Then \(\co^*(B(m,1^r);R)\) is a polynomial algebra. Additionally, if \(i\) is odd, then \(\co^i(B(m,1^r);R) = 0\).
    \item  The integral cohomology of the complex Stiefel manifold \(W(m,r)\) satisfies
\[
    \co^*(W(m,r);\Z) \cong \Lambda_{\Z}[e_{2(m-r+1)-1},e_{2(m-r+2)-1},\dots, e_{2m-1}];
\]
where $\lvert e_i \rvert = i$. 
    \item Moreover, 
\[
    \co^*(W(m,r);\Z/2) \cong \Lambda_{\Z/2}[f_{2(m-r+1)-1},f_{2(m-r+2)-1},\dots, f_{2m-1}];
\]
where \(f_i \coloneqq \rho_2 e_i\). Furthermore,
\begin{equation}\label{eq:action_steenrod_W(m,r)}
Sq^{2j}(f_{2i+1}) = \binom{i}{j} f_{2i+2j+1}~\mathrm{mod}~2,  \text{ for } j \le i,~i+j \le m-1,
\end{equation}
and is zero otherwise. 
\item Finally, given \(m \geq r \geq 1\), the Wu formula holds for the mod 2 characteristic classes \(b_i\). That is, $Sq^{2i-1}(b_j) = 0$ and
\[
Sq^{2i}(b_j) = \sum_{t=0}^{i}\binom{j + t - i - 1}{t}b_{i-t}b_{j+t},
\]
for \(i \leq j.\)
\end{enumerate}
\end{prop}

\section{The lifting problem and the primary obstruction}\label{sec:primaryobs}

Towards a proof of the theorems in Section \ref{sec:main_thms}, we now formulate the relevant lifting problem and begin the construction of the associated Moore-Postnikov tower. Chiefly, we compute the primary obstruction to the lifting problem, consequently proving Proposition \ref{prop:complexhopf}.

For integers \(m \geq r \ge 1\), let us denote by 
\begin{equation*}
    \begin{tikzcd}
        d_{m,r} \colon U(m-r) \times U(1)^r \arrow[r, hook] & U(m)
    \end{tikzcd}
\end{equation*}
the diagonal inclusion. Then, the inclusion of topological groups
\begin{equation*}
    \begin{tikzcd}
        (d_{m,r}, \text{id}_{U(1)^r}) \colon U(m-r) \times U(1)^r \arrow[r, hook] & U(m) \times U(1)^r
    \end{tikzcd}
\end{equation*}
induces a map of classifying spaces
\begin{equation} \label{eq:fibq_1}
    q_{m,r} \coloneqq B(d_{m,r}, \text{id}_{U(1)^r}) \colon B(m-r, 1^r) \longrightarrow B(m,1^r).    
\end{equation}
Then $q_{m,r}$ satisfies
\[
    q_{m,r}^*(\gamma_m \times 1^{\times r}) = \gamma_{m-r}\times \gamma_1^{\times r} \text { and } q_{m,r}^*(1 \times \gamma_1^{\times r}) = 1 \times \gamma_1^{\times r}.
\]
Let us identify the fiber of the homotopy fibration $q_{m,r}$.

\begin{prop}\label{prop:fiberqmk}
For integers $m \ge r \ge 1$, the map \[q_{m,r} \colon B(m-r,1^r) \longrightarrow B(m,1^r)\] defined in \eqref{eq:fibq_1} is homotopy equivalent to a fibration with fiber \(W(m,r)\).
\end{prop}
\begin{proof}
    The inclusion $(d_{m,r}, \text{id}_{U(1)^r})$ of topological groups induces a homotopy fibration on classifying spaces
    \begin{equation*}
        (U(m) \times U(1)^r)/(U(m-r)\times U(1)^r) \longrightarrow B(m-r, 1^r) \xrightarrow{~q_{m,r}~} B(m,1^r).
    \end{equation*}
    Next, one can check that the inclusion $(\text{id}, 1) \colon~ U(m) \to U(m) \times U(1)^r$ induces a continuous map on the quotients
    \begin{equation*}
        U(m)/U(m-r) \longrightarrow (U(m) \times U(1)^r)/(U(m-r)\times U(1)^r),
    \end{equation*}
    which is a bijection of compact spaces and hence a homeomorphism.
\end{proof}

To aid calculations in the sequel, we now compute the map in cohomology induced by \(q_{m,r}\), where \(m \geq r \geq 1.\)

\begin{prop}\label{prop:surjective_q_{m,r}}
Let $m \ge r \ge 1$. Write cohomologies with integer coefficients as
\begin{itemize}
    \item \(\co^*(B(m,1^r)) \cong \Z[c_1,\dots, c_m, y_1, \dots, y_r]\), where \(\lvert c_i\rvert = 2i\) and \(\lvert y_i\rvert =2\), for each \(i = 1,\dots, r\); and 
    \item \(\co^*(B(m-r,1^r)) \cong \Z[c_1',\dots, c_{m-r}', x_1, \dots, x_r]\), where \(\lvert c_i\rvert = 2i\) and \(\lvert x_i\rvert=2\), for each \(i = 1,\dots, r\).
\end{itemize}
Let the map
\[q_{m,r} \colon B(m-r,1^r) \longrightarrow B(m,1^r)\] be as defined in \eqref{eq:fibq_1}. 
\begin{enumerate}
    \item[\emph{(i)}] Then the induced map in integral cohomology
\[
q_{m,r}^* \colon \co^*(B(m,1^r)) \longrightarrow \co^*(B(m-r,1^r)),
\]
satisfies the following: 
\begin{equation}\label{eq:q{m,r}*(y_i)}
q_{m,r}^*(y_i) = x_i;\end{equation}
and 
\begin{equation}\label{eq:q{m,r}*(c_i)}
q_{m,r}^*(c_k) = \sum_{i=\mathrm{max}(0,k-(m-r))}^{\mathrm{min}(k,r)}c'_{k-i}\sigma_i(x_1,\dots, x_r),
\end{equation}
where \(c_0' =1\) and \(\sigma_i(x_1,\dots, x_r)\) is the \(i\)-th elementary symmetric polynomial in the variables \(x_1, \dots, x_r,\).
    \item[\emph{(ii)}] Moreover, \(q_{m,r}^*\) is surjective in all degrees. 
\end{enumerate}
\end{prop}

\begin{proof}
The proof of part (i) is a straightforward computation using the properties of Chern classes.

We proceed with the proof of part (ii). Clearly the \(x_i\) are in the image of \(q_{m,r}^*\) by (\ref{eq:q{m,r}*(y_i)}). Let us prove by induction that each \(c_k'\) is in the image of \(q_{m,r}^*\).

For \(k=1\), it follows from (\ref{eq:q{m,r}*(y_i)}) and (\ref{eq:q{m,r}*(c_i)}) that \[q_{m,r}^*(c_1) = c_1' +x_1 + \dots +x_r = c_1'+q^*_{m,r}(y_1 + \dots + y_r) \in \mathrm{im}(q_{m,r}^*).\]
For the induction hypothesis, suppose for some \(1 < k \leq m-r\) that \[c_1', \dots, c_{k-1}' \in \mathrm{im}(q_{m,r}^*).\] Then by (\ref{eq:q{m,r}*(c_i)}), we obtain
\begin{equation}\label{eq:solve_for_c_s'}
q^*_{m,r}(c_k) = \sum_{i=0}^{\mathrm{min}(k,r)}c_{k-i}'\sigma_i(x_1,\dots,x_r) = c_k' + \sum_{i=1}^{\mathrm{min}(k,r)}c_{k-i}'\sigma_i(x_1,\dots,x_r).
\end{equation}
But \(\sum_{i=1}^{\mathrm{min}(k,r)}c_{k-i}'\sigma_i(x_1,\dots,x_r) \in \mathrm{im}(q_{m,r}^*)\) by the induction hypothesis. Hence rearranging (\ref{eq:solve_for_c_s'}) gives a description of \(c_k'\) in terms of elements of \(\mathrm{im}(q_{m,r}^*)\), as desired. This concludes the proof.
\end{proof}

Henceforth fix a complex \(m\)-plane bundle \(\xi \colon X \to BU(m)\) and complex line bundles \(\ell_1,\dots, \ell_r\) over \(X\). This gives rise to a classifying map 
\begin{equation}\label{eq:xilinesclassifyingmap}
f = (\xi, \ell_1,\dots, \ell_r)\colon X \longrightarrow B(m,1^r) 
\end{equation}
defined by 
\[f^*(1 \times \gamma_1^{\times r}) = \ell_1 \oplus \cdots \oplus \ell_r~ \text{ and }~ f^*(\gamma_m \times 1^{\times r}) = \xi.\]
With this notation, we may formulate the next lemma (c.f., \cite{mello1987two}).

\begin{lem}
    For a complex rank $m$ bundle \(\xi\) and line bundles $\ell_1, \dots, \ell_r$ over $X$, we have that $\xi$ admits \(\ell_1 \oplus \cdots \oplus \ell_r\) as a subbundle if and only if there exists a map \(g \colon X \to B(m-r,1^r)\) such that \(q_{m,r}\circ g \simeq f\).
\end{lem}
\begin{proof}
    Given a splitting $\xi \cong \eta \oplus \ell_1 \oplus \dots \oplus \ell_r$, for some $(m-r)$-bundle $\eta$, let $g$ be the the classifying map $(\eta, \ell_1, \dots, \ell_r) \colon X \to B(m-r,1^r)$. Then, \(q_{m,r}\circ g \simeq f\) follows by uniqueness of the classifying map up to homotopy \cite[Thm.~14.6]{milnor2005characteristic}. On the other hand, given such a map $g$, we may define $\eta = g^*\gamma_{m-r}$ and $\ell_i = g^*\gamma_1^{(i)}$, for $i=1, \dots, r$, where $\gamma_1^{(i)}$ denotes the canonical line bundle corresponding to the $i$-th factor of $BU(1)$ in $B(m-r,1^r)$. This induces the bundle splitting
    \[
        \xi \cong f^*(\gamma_m) \cong g^*(\gamma_{m-r}\times \gamma_1^{1} \times \dots \times \gamma_1^{r}) \cong \eta \oplus \ell_1 \oplus \dots \oplus \ell_r,
    \]
    as claimed.
\end{proof}

Subsequently, we seek obstructions to the existence of a map \(g \colon X \to B(m-r,1^r)\) making the following diagram homotopy commutative.
\begin{equation} \label{eq:lifting-problem-q_m,r}
    \begin{tikzcd}
        {} &&& B(m-r,1^r) \arrow[d, "q_{m,r}"]\\
        X \arrow[urrr, dashed] \arrow[rrr, "\,\,\,(\xi{,}\ell_1{,}\dots{,}\ell_r)"] & & & B(m,1^r)
    \end{tikzcd}
\end{equation}
As the main result of this section, we will use Proposition \ref{prop:fiberqmk}  to compute the first Moore-Postnikov invariant in the Moore--Postnikov tower associated to the lifting problem \eqref{eq:lifting-problem-q_m,r}. Diagrammatically, we have
\begin{equation} \label{eq:moore-postnikov-tower-beginning}
    \begin{tikzcd}
        W(m,r) \arrow[r] & B(m-r, 1^r) \arrow[r, "q_{m,r}"] & B(m,1^r) \arrow[r, "k_1"] & K(\Z, 2(m-r+1)) \\
        {} & {} & X \arrow[u, "f=(\xi{,} \ell_1{,} \dots{,} \ell_r)"'] & {}
    \end{tikzcd}
\end{equation}
where we have used the fact that \(\pi_{2(m-r)+1}(W(m,r))\cong \Z\), for \(m>0\), is the first non-zero homotopy group of \(W(m,r)\), see \cite{gilmore1967complex}.

\begin{prop}[Primary obstruction]\label{prop:primaryobs}
Let \(m\geq 1\) be an integer. Let \(f=(\xi,\ell_1,\dots,\ell_r) \colon X \to B(m,1^r)\) be as in \eqref{eq:xilinesclassifyingmap}, and let \[k_1 \in \co^{2(m-r+1)}(B(m,1^r);\Z)\] be the characteristic class in the fibration \(q_{m,r}\). Then the primary obstruction to the lifting problem \eqref{eq:lifting-problem-q_m,r} is the singleton set
\[
    \lobs^{2(m-r+1)}(f, q_{m,r},k_2) = \{c_{m-r+1}(\xi - \ell_1 \oplus \cdots \oplus \ell_r)\} \subseteq \co^{2(m-r+1)}(X;\Z).
\]
\end{prop}

The proof of Proposition \ref{prop:primaryobs} is a straightforward consequence of the following lemma, whose proof is a modification of \cite[pp.\ 269--270]{mello1987two}.

\begin{lem}\label{lem:diffsofq1}
Consider the fibration 
    \[
        W(m,r) \longrightarrow B(m-r, 1^r) \xrightarrow{~~q_{m,r}~~} B(m,1^r).
    \]
defined as in \emph{(\ref{eq:fibq_1})} and let $(E_*^{*,*}, d_*)$ denote the associated Leray--Serre spectral sequence.
 Then, the differential 
 \[
    d \coloneqq d_{2m-2i+2} \colon E_{2m-2i+2}^{0,2m-2i+1} \longrightarrow E_{2m-2i+2}^{2m-2i+2,0}
\]
satisfies 
    \[
        d(e_{2m-2i+1}) = a_{m-i+1}=c_{m-i+1}(\gamma_m \!\times 1^{\times r}\! -\! 1 \times\! \gamma_1^{\times r}), \hspace{5mm} \textrm{for } i= 1,\dots, r.
    \]
\end{lem}

\begin{proof}
Let \(U \coloneqq \bigcup_{j=1}^\infty U(j)\) be the infinite unitary group and consider a commutative diagram of topological groups
\begin{equation*}
    \begin{tikzcd}
        U(m-r) \times U(1)^r \arrow[rrr,"(d_{m,r}{,} \text{id}_{U(1)^r})"] \arrow[d, "\text{proj}"] & & & U(m) \times U(1)^r \arrow[d]\\
        U(m-r) \arrow[rrr, hook] & & & U,
    \end{tikzcd}
\end{equation*}
where the right vertical map is given by
\[(A, a_1, \dots, a_r) \longmapsto \begin{pmatrix}
A & 0 & 0 & 0 \\
0 & 1 & 0 & 0 \\
0 & 0 & 1 & 0 \\
0 & 0 & 0 & \ddots
\end{pmatrix}\cdot 
\begin{pmatrix}
\mathbb{I}_{m-r} & 0 & 0  & 0 & 0 & 0\\
0 & a_1^{-1} & 0 & 0 & 0 & 0\\
0 & 0 & \ddots & 0 & 0 & 0\\
0 & 0 & 0 & a_r^{-1}  & 0 & 0\\
0 & 0 & 0 & 0  &1 & 0 \\
0 &0 & 0 & 0 & 0  &\ddots
\end{pmatrix}.\] 
Passing to classifying spaces, the square induces a map of homotopy fibrations (c.f.~\ \cite{mello1987two})
\begin{equation} \label{eq:morphism-q-q_m,r}
    \begin{tikzcd}
	{W(m,r)} & {B(m-r,1^r)} & {B(m,1^r)} \\
	{W(m-r)} & {BU(m-r)} & BU
	\arrow[from=1-1, to=1-2]
	\arrow["{j}",from=1-1, to=2-1]
	\arrow["{q_{m,r}}", from=1-2, to=1-3]
	\arrow["\mathrm{proj}", from=1-2, to=2-2]
	\arrow["h", from=1-3, to=2-3]
	\arrow[from=2-1, to=2-2]
	\arrow["q", from=2-2, to=2-3]
\end{tikzcd}    
\end{equation}
where $W(m-r)\coloneqq U/U(m-r)$.
Moreover, \(q\) is the classifying map of the stable bundle \(\gamma_{m-r}\), and \(h\) is the classifying map of the stable bundle \(\gamma_m \times 1^{\times r} - 1 \times \gamma_1^{\times r}\). Finally, the map \(j \colon W(m,r)\to W(m-r)\) is induced by the inclusion of \(U(m) \xhookrightarrow{} U\). We claim that \(j\) is \((2m-1)\)-connected, whence 
\[
j^* \colon \co^*(W(m-r);\Z) \longrightarrow \co^*(W(m,r);\Z)
\]
is an isomorphism for \(* = 1,\dots, 2m-2\) and a surjection for \(* = 2m-1.\)
To see this, consider the following commutative diagram of topological groups:
\begin{equation*}
    \begin{tikzcd}
        U(m-r) \arrow[r, hook] \arrow[d, equal] & U(m) \arrow[d, hook]\\
        U(m-r) \arrow[r, hook] & U.
    \end{tikzcd}
\end{equation*}
Passing to classifying spaces, there is the Thomas diagram associated to \((\blacktriangle)\):
\[\begin{tikzcd}
	{\Omega (U/U(m))} & {*} & {U/U(m)} \\
	{W(m,r)} & {BU(m-r)} & {BU(m)} \\
	{W(m-r)} & {BU(m-r)} & BU
	\arrow[from=1-1, to=1-2]
	\arrow[from=1-1, to=2-1]
	\arrow[from=1-2, to=1-3]
	\arrow[from=1-2, to=2-2]
	\arrow[from=1-3, to=2-3]
	\arrow[from=2-1, to=2-2]
	\arrow["j", from=2-1, to=3-1]
	\arrow[from=2-2, to=2-3]
	\arrow[equals, from=2-2, to=3-2]
	\arrow["{(\blacktriangle)}"{description}, draw=none, from=2-2, to=3-3]
	\arrow[from=2-3, to=3-3]
	\arrow[from=3-1, to=3-2]
	\arrow["q", from=3-2, to=3-3]
\end{tikzcd}\]
We point out that the map $j$ between the fibers is the same as the one in \eqref{eq:morphism-q-q_m,r}, since it is again induced by the inclusion $U(m-r) \hookrightarrow U$.
The homotopy fiber of \(* \to U/U(m)\) is \(\Omega(U/U(m)).\) Furthermore, by Lemma \ref{lem:thomas_diagram}, the homotopy fiber of \(j\) is \(\Omega (U/U(m))\), which is \((2m-1)\)-connected. The claim follows.

Since the map \(j^*\) is surjective, particularly in degrees \(1, 3, \dots, 2m-1\), we can lift the odd dimensional generators of \(\co^*(W(m,r);\Z)\), described in Proposition \ref{prop:cohomology_rings}(iv), to obtain elements \[e_{2m-2i+1} \in \co^{2m-2i+1}(W(m-r);\Z),\] for \(i = 1,2,\dots, 2m-1\). Recall that the transgression of
\[
    e_{2m-2i+1}\in \co^{2m-2i+1}(W(m-r);\Z)
\]
in the Leray--Serre spectral sequence of the fibration $q$ is classically known to be the Chern class
\[
    c_{m-i+1} \in \co^*(BU;\Z) \cong \Z[c_1, c_2, \dots],
\]
where $|c_n| = 2n$. This fact is a straightforward extension of \cite[Theorem 5.5(5), p.\ 343]{mim1991book}. Finally, the claim follows by naturality of the transgression via the morphism induced by \eqref{eq:morphism-q-q_m,r}, as illustrated in Figure \ref{fig:ss_q_m,r_naturality}, and the fact that $h^*(c_{m-i+1}) = a_{m-i+1} \in \co^*(B(m,1^r);\Z)$.
\end{proof}

\begin{figure}[h]
  \centering
  \includegraphics[width=0.95\textwidth]{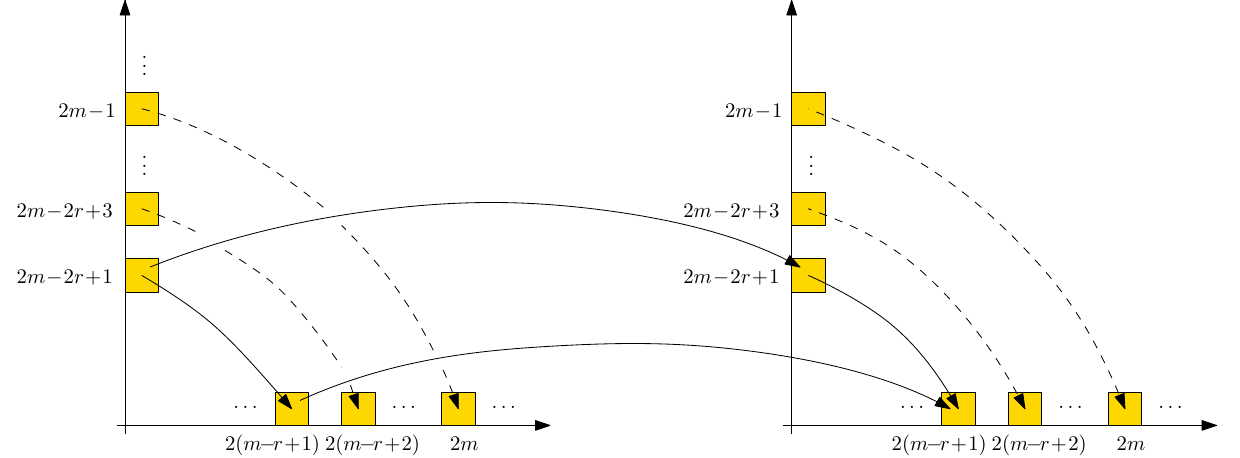}
  \caption{An illustration of the morphism of Leray--Serre spectral sequences between $q$ and $q_{m,r}$ induced by \eqref{eq:morphism-q-q_m,r}. Dashed differentials appear on later pages.}
  \label{fig:ss_q_m,r_naturality}
\end{figure}

We obtain the following result immediately from Lemma \ref{lem:diffsofq1} and the Serre cohomological exact sequence.
\begin{cor}\label{cor:kerq_{m,r}^*}
For \(m \geq r \geq 1\), let the homotopy fibration 
    \[
        W(m,r) \longrightarrow B(m-r, 1^r) \xrightarrow{~~q_{m,r}~~} B(m,1^r)
    \]
be as defined in \emph{(\ref{eq:fibq_1})}. Then 
\[
    \mathrm{ker}(q_{m,r}^*)\cap \co^{\leq 2m}(B(m,1^r);\Z) = (a_{m-i+1}\colon~ i=1, \dots, r) \cap \co^{\leq 2m}(B(m,1^r);\Z),
\]
where the \(a_{m-i+1}\) are defined as in \emph{Lemma \ref{lem:diffsofq1}.} 
\end{cor}
To conclude this section, Proposition \ref{prop:complexhopf} follows immediately from Proposition \ref{prop:primaryobs}.

\section{Two complex line bundles: $m$ odd} \label{sec:thm_cpspan2_odd}

In this section, we prove Theorem \ref{thm:cpspan2_odd}. Namely, let \(X\) be a \(2m\)-dimensional CW complex with \(m \ge 3\) odd. Given complex vector bundles \[\xi \colon X \longrightarrow BU(m) \hspace{3mm} \text{ and } \hspace{3mm} \ell_1, \ell_2 \colon X \longrightarrow BU(1),\] we wish to solve the lifting problem \eqref{eq:lifting-problem-q_m,r} for $r=2$. Using Proposition \ref{prop:piW(m,r)}, we construct the following Moore-Postnikov tower of \(q_{m,2}\) on top of \eqref{eq:moore-postnikov-tower-beginning}, namely
\begin{equation}\label{eq:mp_diag_cpspan2_odd}
\begin{tikzcd}[ampersand replacement=\&]
	\&\& {E[2]} \\
	\&\& {E[1]} \& {K(\Z,2m)} \\
	{W(m,2)} \& {B(m-2,1^2)} \& {B(m,1^2)} \& {K(\Z,2m-2)} \\
	\&\& X
	\arrow["{p_2}", from=1-3, to=2-3]
	\arrow["{k_2}", from=2-3, to=2-4]
	\arrow["{p_1}", from=2-3, to=3-3]
	\arrow[from=3-1, to=3-2]
	\arrow["{q_1}"{description}, from=3-2, to=2-3]
	\arrow["{q_{m,2}}", from=3-2, to=3-3]
	\arrow["{k_1}", from=3-3, to=3-4]
	\arrow["{(\xi, \ell_1,\ell_2)}"', from=4-3, to=3-3]
\end{tikzcd}\end{equation}
where \(k_i\) is the characteristic class in the fibration \(q_{i-1}\) (with \(q_0 \coloneqq q_{m,2}\)), and each \(E[i]\) is constructed as the homotopy fiber of \(k_i\).\\

\subsection{The Primary Obstruction}
Restating Proposition \ref{prop:primaryobs} in this context, we have the following.
\begin{lem}\label{lem:primaryob_cpspan2_odd}
Let all spaces and maps be as in \emph{(\ref{eq:mp_diag_cpspan2_odd})}. Then the primary obstruction 
\[
\lobs^{2m-2}((\xi,\ell_1,\ell_2),q_{m,2},k_1) \subseteq \co^{2m-2}(X;\Z)
\]
to lifting \((\xi,\ell_1,\ell_2)\) along \(q_{m,2}\) is the singleton set \(\{c_{m-1}(\xi - \ell_1 \oplus \ell_2)\}.\)
\end{lem}

\subsection{The Generalized Secondary Obstruction}
Finally, Theorem \ref{thm:cpspan2_odd} is an immediate consequence of the following.
\begin{thm}\label{thm:secondaryob_cpspan2_odd}
Let all spaces and maps be as in \emph{(\ref{eq:mp_diag_cpspan2_odd})}. Suppose that \((\xi,\ell_1,\ell_2) \colon X \to B(m,1^2)\) satisfies \(c_{m-1}(\xi - \ell_1\oplus \ell_2)=0\).
\begin{enumerate}[\normalfont(i)]
    \item Suppose that \(\co^{2m}(X;\Z) = \delta Sq^2\rho_2 \co^{2m-3}(X;\Z).\) Then zero is an element of the generalized secondary obstruction
    \[
\lobs^{2m}((\xi,\ell_1,\ell_2), q_{m,2},k_2) \subseteq \co^{2m}(X;\Z)
    \]
to lifting \((\xi,\ell_1,\ell_2)\) along \(q_{m,2}\).
    \item Suppose that \(\co^{2m}(X;\Z)\) has no \(2\)-torsion. Then the generalized secondary obstruction 
\[
\lobs^{2m}((\xi,\ell_1,\ell_2), q_{m,2},k_2) \subseteq \co^{2m}(X;\Z) \]
to lifting \((\xi,\ell_1,\ell_2)\) along \(q_{m,2}\) is the singleton set \(\{c_m(\xi-\ell_1\oplus \ell_2)\}.\)
\end{enumerate}
\end{thm}

The rest of this section is dedicated to proving Theorem \ref{thm:secondaryob_cpspan2_odd}. We remark that the proof is both more detailed and a generalization of the proof of \cite[Theorem 3.5]{thomas1967postnikov}.\\

Suppose that the classifying map \((\xi,\ell_1,\ell_2) \colon X \longrightarrow B(m,1^2)\) satisfies \(c_{m-1}(\xi - \ell_1 \oplus \ell_2) = 0\). Then by Lemma \ref{lem:primaryob_cpspan2_odd} there is a lift making the following diagram commute
\[\begin{tikzcd}
	& & {B(m-2,1^2)} \\
	X & & {B(m,1^2)}
	\arrow["{q_{m,2}}", from=1-3, to=2-3]
	\arrow[ dashed, from=2-1, to=1-3]
	\arrow["{(\xi,\ell_1,\ell_2)}", from=2-1, to=2-3]
\end{tikzcd}\]
if and only if 
\[
g^*(k_2) = 0 \in \co^{2m}(X;\Z)/\mathrm{Indet},
\]
where \(g\colon X \to E[1]\) is some lift of \((\xi,\ell_1,\ell_2)\) to \(E[1]\), and \(\mathrm{Indet}\) denotes the indeterminancy of such lifts.

We first establish the following auxiliary fact which uses the results of Lemma \ref{lem:F[1]-odd2}.

\begin{lem}\label{lem:q1-in-coh}
	Map $q_1 \colon B(m-2,1^2) \to E[1]$ induces an isomorphism 
	\[
		q_1^* \colon \co^{i}(E[1];\Z) \xrightarrow{~\cong~} \co^i(B(m-2,1^2);\Z), \hspace{3mm} \text{for}~ 0 \le i \le 2m-1,
	\]
	and a split exact sequence
	\begin{equation} \label{eq:ses-q1}
		0 \longrightarrow \Z\langle k_2\rangle \longrightarrow \co^{2m}(E[1];\Z) \xrightarrow{~q_1^*~} \co^{2m}(B(m-2,1^2);\Z) \longrightarrow 0.
	\end{equation}
	Moreover, for $K \coloneqq K(\Z,2m-3)$, the map
	\[
		\id \times q_1 \colon K\times B(m-2,1^2) \longrightarrow K\times E[1]
	\]
	induces the map in degree $2m$ cohomology
	\[
		(\id \times q_1)^* = \id\oplus q_1^* \colon \co^{2m}(K;\Z) \oplus \co^{2m}(E[1];\Z) \to \co^{2m}(K;\Z) \oplus \co^{2m}(B(m-2,1^2);\Z).
	\]
\end{lem}
\begin{proof}
	Throughout the proof, all cohomology is considered with integer coefficients. Let $(E_*^{*.*}, d_*)$ denote the Leray-Serre spectral sequence of 
    \begin{equation*}
        F[1] \longrightarrow B(m-2,1^2) \xrightarrow{~q_1~} E[1].
    \end{equation*}
    By Lemma \ref{lem:F[1]-odd2}, $F[1]$ has trivial cohomology in degrees up to $2m$, except $\co^{2m-1}(F[1])\cong \Z\langle e_{2m-1}/2 \rangle$. Thus, $q_1^*$ is an isomorphism in cohomological degrees up to $2m-1$. Moreover, since $\co^{2m-1}(B(m-2,1^2)) \cong 0$, the only non-trivial differential emanating from any of the first $2m$ diagonals is
	\begin{equation*}
		d_{2m} \colon E^{0,2m-1}_{2m} \cong \Z \langle e_{2m-1}/2 \rangle \longrightarrow \co^{2m}(E[1]) \cong E^{2m,0}_{2m}, ~e_{2m-1}/2 \longmapsto k_2
	\end{equation*}
	and it is a monomorphism. See Figure \ref{fig:ss_q1} for illustration. After passing to the $(2m+1)$-th page, we observe that $E_{2m+1}^{2m,0} \cong \co^{2m}(E[1])/\Z\langle k_2 \rangle$ is the only non-trivial slot on the $2m$-th diagonal, hence we have a short exact sequence \eqref{eq:ses-q1}. Since the cokernel is a free $\Z$-module, we conclude that it splits.
	
	As for the second part of the proof, the claim follows from similar reasoning and K\"unneth's formula for cohomology, which is natural in this case.
\end{proof}

\begin{figure}[h]
  \centering
  \includegraphics[width=0.45\textwidth]{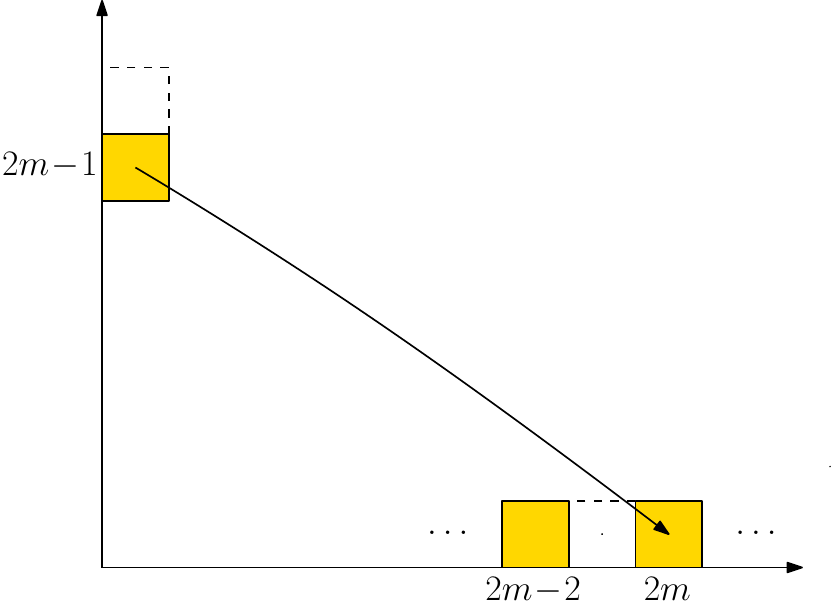}
  \caption[Spectral sequence of $q_1$]{Spectral sequence of $q_1$.}
  \label{fig:ss_q1}
\end{figure}

\begin{lem}\label{lem:mu*k2}
	Let \(k_2\in \co^{2m}(E[1];\Z)\) be as in \eqref{eq:mp_diag_cpspan2_odd}. Further let \[\mu \colon K(\Z,2m-3)\times E[1]\longrightarrow E[1]\] be the principal action map associated to the principal fiber space \(E[1]\). Then
	\begin{equation*}
		\mu^*(k_2) = \delta Sq^2 \rho_2 \iota_{2m-3}^2 \oplus k_2   \in \co^{2m}(K(\Z,2m-3)\times E[1];\Z)
		\cong \co^{2m}(K(\Z, 2m-3); \Z) \oplus \co^{2m}(E[1];\Z).
	\end{equation*}
\end{lem}
\begin{proof}
	Throughout the proof, all cohomology coefficients will be assumed to be integer. By the K\"unneth formula for cohomology, Lemma \ref{lem:q1-in-coh}, and properties of $\co^*(K(\Z, 2m-3))$ from Table \ref{table:cohomEMS}, we have
	\[
		\co^{2m}(K(\Z,2m-3)\times E[1])
		\cong \co^{2m}(K(\Z, 2m-3)) \oplus \co^{2m}(E[1]).
	\]
    We have by \cite[Lemma 4, p.\ 17]{thomas2006seminar} the following homotopy commutative diagram:
\begin{equation}\label{eq:indetdiag2_even}
    \begin{tikzcd}
	{K(\Z,2m-3)} && {K(\Z,2m-3)} \\
	{K(\Z,2m-3)\times B(m-2,1^2)} & {K(\Z,2m-3)\times E[1]} & {E[1]} \\
	{B(m-2,1^2)} && {B(m,1^2)}
	\arrow[Rightarrow, no head, from=1-1, to=1-3]
	\arrow[from=1-1, to=2-1]
	\arrow[from=1-3, to=2-3]
	\arrow["{1 \times q_1}", from=2-1, to=2-2]
	\arrow["\nu"{description}, curve={height=-30pt}, from=2-1, to=2-3]
	\arrow["{{\mathrm{proj}}_2}", from=2-1, to=3-1]
	\arrow["\mu", from=2-2, to=2-3]
	\arrow["{p_1}", from=2-3, to=3-3]
	\arrow["s", curve={height=-18pt}, from=3-1, to=2-1]
	\arrow["{q_{m,2}}", from=3-1, to=3-3]
\end{tikzcd}
\end{equation}
Here \(\nu = \mu \circ (\mathrm{id}\times q_1)\) and \(\bar{s}\) is the canonical section of the projection $\mathrm{proj}_2$.
By Proposition \ref{prop:surjective_q_{m,r}}(ii), we have that \(q_{m,2}^*\) is surjective in degree \(2m\). Moreover, using Corollary \ref{cor:kerq_{m,r}^*}, it is straightforward to check that 
\[
\mathrm{ker}(q_{m,2}^*)\cap \co^{2m}(B(m,1^2)) \subseteq \mathrm{ker}(p_1^*)\cap \co^{2m}(B(m,1^2)).
\] 
Indeed, \(q_{m,2}^* = q_1^*\circ p_1^*\), and by Lemma \ref{lem:q1-in-coh}, \(q_1^*\) in degree \(2m\) is injective modulo the free \(\Z\)-summand generated by \(k_2\).

Combining the above together with the fact that (\ref{eq:indetdiag2_even}) is a pullback diagram, we are justified in applying \cite[Property 5]{thomas2006seminar} to obtain following exact sequence
\[
\dots \longrightarrow \co^{2m}(E[1]) \xrightarrow{~\nu^*~} \co^{2m}(K(\Z,2m-3)\times B(m-2,1^2)) \xrightarrow{~\tau_1~} \co^{2m+1}(B(m,1^2)) \longrightarrow \dots ,
\]
where \(\tau_1\) is the relative transgression (see \cite[page 16]{thomas2006seminar}).
By K\"{u}nneth's formula, \[\co^{2m+1}(B(m,1^2)) \cong 0,\] so the fact that \(\nu \circ s \simeq q_1\) gives rise to an exact sequence
\begin{equation}\label{eq:kerq)1_exact_seq}
\mathrm{ker}(q_1^*) \cap \co^{2m}(E[1]) \xrightarrow{~{\nu^*}\lvert_{\ker q_1^*}~} \mathrm{ker}(s)^*\cap \co^{2m}(K(\Z,2m-3)\times B(m-2,1^2)) \xrightarrow{~\tau_1~} 0,
\end{equation}
where \(\mathrm{ker}(q_1^*) \cap \co^{2m}(E[1]) \cong \Z\langle k_2 \rangle\) by Lemma \ref{lem:q1-in-coh}.
K\"unneth theorem for cohomology together with Table~\ref{table:cohomEMS} yield
\begin{equation*}\label{eq:kunneth1}
	\co^{2m}(K(\Z,2m-3)\times B(m-2,1^2)) \cong  \co^{2m}(K(\Z,2m-3)) \oplus \co^{2m}(B(m-2,1^2))
\end{equation*}
and the intersection with $\mathrm{ker}(s^*)$ is isomorphic to the first direct summand. Therefore, by \eqref{eq:kerq)1_exact_seq} and Table \ref{table:cohomEMS}, we have
\[
	\nu^*(k_2) = \delta Sq^2 \rho_2 \iota_{2m-3}^2 \oplus 0 \in \co^{2m}(K(\Z,2m-3))\oplus \co^{2m}(B(m-2,1^2)).
\]
Furthermore, by Lemma \ref{lem:q1-in-coh}, we have that
\begin{equation*}
	\mu^*(k_2) = \delta Sq^2 \iota_{2m-3}^2 \oplus nk_2 \in \co^{2m}(K(\Z,2m-3))\oplus \co^{2m}(E[1]),
\end{equation*}
for some integer $n$. The principal action map \(\mu\) has a right homotopy inverse \[T \colon E[1] \to K(\Z,2m-3) \times E[1]\] which is a section of the trivial fiber bundle $K(\Z,2m-3) \times E[1] \to E[1]$  (see \cite{pollina1982tangent,thomas2006seminar}). Using Lemma \ref{lem:q1-in-coh} for the description of $\co^{2m}(K(\Z,2m-3) \times E[1])$, these two facts about $T$ imply the first and the last equality:
\begin{equation*}
	2k_2 = (\mu \circ T)^*(2k_2) = T^*(\mu^*(2k_2))=T^*(0\oplus 2nk_2) = 2nk_2 \in \co^{2m}(E[1]),
\end{equation*}
so we conclude that $n=1$, thus finishing the proof.
\end{proof}

Now, we calculate the indeterminancy of lifts. 

\begin{cor}\label{cor:indet_cpspan2_odd}
Let all maps and spaces be defined as in \emph{(\ref{eq:mp_diag_cpspan2_odd})}. Suppose that the classifying map \[(\xi,\ell_1,\ell_2) \colon X \longrightarrow B(m,1^2)\] satisfies \(c_{m-1}(\xi - \ell_1 \oplus \ell_2) = 0\). Then the elements \(g^*(k_2)\), as \(g\) runs over lifts of \((\xi, \ell_1,\ell_2)\) to \(E[1]\), live in a single coset in \(\co^{2m}(X;\Z)\) of the subgroup
\begin{equation}\label{eq:indet2_even}
\mathrm{Indet} \coloneqq \delta Sq^2 \rho_2 \co^{2m-3}(X;\Z).
\end{equation}
\end{cor}
\begin{proof}
Let \(g_1, g_2 \colon X \to E[1]\) be arbitrary lifts of \((\xi,\ell_1,\ell_2) \colon X \to B(m,1^2)\). Then there exists a cohomology class \(\alpha \colon X \to K(\Z,2m-3)\) such that the composition 
\[
\begin{tikzcd}
    X \arrow[r,"\Delta"] & X \times X \arrow[r, "\alpha\times g_1"] & K(\Z,2m-3)\times E[1] \arrow[r, "\mu"] & E[1]
\end{tikzcd}
\]
is homotopic to \(g_2\), see \cite{thomas2006seminar, pollina1982tangent}. By Lemma \ref{lem:mu*k2}, we see that 
\(
g_2^*(k_2) - g_1^*(k_2) = \delta Sq^2\rho_2 \alpha,
\)
for this \(\alpha\) in \(\co^{2m-3}(X;\Z)\), which finishes the proof.
\end{proof}

Now part (i) of Theorem \ref{thm:secondaryob_cpspan2_odd} follows directly from the discussion at the beginning of this section and Corollary \ref{cor:indet_cpspan2_odd}.

To conclude, we now finish the proof of Theorem \ref{thm:secondaryob_cpspan2_odd}(ii). 
By assumption, \(\co^{2m}(X;\Z)\) has no \(2\)-torsion, whence \[\mathrm{Indet} = \delta Sq^2 \rho_2 \co^{2m-3}(X;\Z)= 0.\] In this case, there is a unique secondary obstruction to lifting \((\xi,\ell_1,\ell_2)\) along \(q_{m,2}\), which is identified in the proceeding lemma. In the statement and proof of the result, we use Notation \ref{notate:1^k}(v). Moreover, we will denote by $F[1]$ the fiber of $q_1$ in \eqref{eq:mp_diag_cpspan2_odd} and by $s_1 \colon F[1] \to W(m,2)$ the map between the fibers of $q_{m,2}$ and $q_1$ induced by $p_1$.

\begin{lem}\label{lem:k_2_even}
Let \(k_2 \in \co^{2m}(E[1];\Z)\) be the characteristic class in the fibration \[F[1] \longrightarrow B(m-2,1^2) \xrightarrow{~q_1~} E[1]\] as defined in \emph{(\ref{eq:mp_diag_cpspan2_odd})}. If \(\co^{2m}(X;\Z)\) has no \(2\)-torsion, then, given any lift \(g \colon X \to E[1]\) of \((\xi,\ell_1,\ell_2)\) in \emph{(\ref{eq:mp_diag_cpspan2_odd})}, \(g^*(k_2)\) vanishes if and only if \(c_m(\xi-\ell_1 \oplus \ell_2)\) vanishes.
\end{lem}
\begin{proof}
    Throughout the proof, we will consider only cohomology with integral coefficients. Consider the following map of fibrations
    \begin{equation*}
        \begin{tikzcd}
            F[1] \arrow[r] \arrow[d, "s_1"] & B(m-2,1^2) \arrow[r, "q_1"] \arrow[d, equal] & E[1] \arrow[d, "p_1"]\\
            W(m,2) \arrow[r] & B(m-2,1^2) \arrow[r, "q_{m,2}"] & B(m,1^2),
        \end{tikzcd}
    \end{equation*}
    which is induced by \eqref{eq:mp_diag_cpspan2_odd}.
    The diagram above induces a morphism of Leray-Serre spectral sequences of $q_1$ and $q_{m,2}$. Applying the Lemma \ref{lem:F[1]-odd2} from below, the differentials (i.e., transgressions) on the $2m$-th page fit into the following commutative diagram:
    \begin{equation*}
        \begin{tikzcd}
            \co^{2m-1}(W(m,2)) \arrow[r, "s_1^*"] \arrow[d, "\tau_{q_{m,2}}"] & \co^{2m-1}(F[1]) \arrow[d, "\tau_{q_1}"] & e_{2m-1} \arrow[r, mapsto] \arrow[d, mapsto] & e_{2m-1} = 2 \cdot \frac{1}{2} e_{2m-1} \arrow[d, mapsto]\\
            \co^{2m}(B(m,1^2)) \arrow[r, "p_1^*"] & \co^{2m}(E[1]) & a_{m} \arrow[r, mapsto] & p_1^*(a_m)=2k_2.
        \end{tikzcd}
    \end{equation*}
    Here, \(\tau_h\) stands for the cohomological transgression in the fibration \(h.\) Therefore, every homotopy lift
    \begin{equation*}
        \begin{tikzcd}
            {} & & E[1] \arrow[d, "p_1"]\\
            X \arrow[urr, dashed, "g"] \arrow[rr, "f"] & & B(m,1^2),
        \end{tikzcd}
    \end{equation*}
     of the classifying map $f = (\xi, \ell_1, \ell_2)$ over $p_1$ satisfies 
    \[
        2g^*(k_2) = (p_1 \circ g)^*(a_m)= f^*(a_m) = c_m(\xi-\ell_1\oplus\ell_2) \in \co^{2m}(X).
    \]
    Thus, the claim follows since \(\co^{2m}(X)\) has no \(2\)-torsion by assumption. 
\end{proof}

For the above proof, we made use of the following lemma. For an infinite cyclic group $\Z\langle e \rangle$, we will denote by $\Z\langle e/p\rangle$ the non-trivial extension of $\Z/p$ by $\Z\langle e \rangle$, for an integer \(p \geq 2\).

\begin{lem} \label{lem:F[1]-odd2}
    Let $m \ge 3$ be odd and $s_1 \colon~ F[1] \to W(m,2)$ be the map induced by $p_1$ between the fibers of $q_{m,2}$ and $q_1$ in \eqref{eq:mp_diag_cpspan2_odd}. Then $F[1]$ is \((2m-2)\)--connected, has $\co^{2m}(F[1];\Z) \cong 0$, and $s_1$ induces the inclusion
    \begin{equation*}
            s_1^* \colon \Z\langle e_{2m-1}\rangle \cong \co^{2m-1}(W(m,2)) \longrightarrow \co^{2m-1}(F[1]) \cong \Z \langle e_{2m-1}/2\rangle,~~e_{2m-1} \longmapsto e_{2m-1}.
    \end{equation*}
    Here $e_{2m-1}$ is as introduced in \emph{Proposition \ref{prop:cohomology_rings}}.
\end{lem}
\begin{proof}
By Lemma \ref{lem:thomas_diagram} we have that $s_1$ fits into the following homotopy commutative Thomas diagram associated to (\(\dagger\)). 
    \begin{equation}\label{eq:diage1}
        \begin{tikzcd}
	   {K(\Z,2m-4)} & {*} & {K(\Z,2m-3)} \\
	   {F[1]} & {B(m -2,1^2)} & {E[1]} \\
	   {W(m,2)} & {B(m-2,1^2)} & {B(m,1^2)}
	   \arrow[from=1-1, to=1-2]
	   \arrow[from=1-1, to=2-1]
	   \arrow["{\mathrm{ev}_1}", from=1-2, to=1-3]
	   \arrow[from=1-2, to=2-2]
	   \arrow["{i_1}", from=1-3, to=2-3]
	   \arrow[from=2-1, to=2-2]
	   \arrow["s_1", from=2-1, to=3-1]
	   \arrow["{q_1}", from=2-2, to=2-3]
	   \arrow["{\mathrm{id}~}"', Rightarrow, no head, from=2-2, to=3-2]
	   \arrow["{(\dagger)}"{description}, draw=none, from=2-2, to=3-3]
	   \arrow["{p_1}", from=2-3, to=3-3]
	   \arrow[from=3-1, to=3-2]
	   \arrow["{q_{m,2}}", from=3-2, to=3-3]
    \end{tikzcd}
    \end{equation}
    Here \(\mathrm{ev}_1\) denotes the path-loop fibration over \(K(\Z,2m-3)\) and \(q_{m,2}\) and \(q_1\) are defined as in (\ref{eq:mp_diag_cpspan2_odd}).

    As before, all cohomology groups have integral coefficients. We consider the Leray--Serre spectral sequence \((E_*^{*,*},d_*)\) of the homotopy fibration 
    \begin{equation}\label{eq:s_1}
        K(\Z,2m-4) \longrightarrow F[1]\overset{s_1}\longrightarrow W(m,2),
    \end{equation}
    which is the leftmost column in (\ref{eq:diage1}). Now $F[1]$ is $(2m-2)$-connected and since $m$ is odd (see Table \ref{table:pi(W(m,2))}), we have that
    \[
        \pi_{2m-1}(F[1]) \cong \pi_{2m-1}(W(m,2)) \cong \Z.
    \]
    Hence $\co^{2m-1}(F[1]) \cong \Z$ by the universal coefficient theorem. Therefore, the differential on the $(2m-3)$-th page
    \[
        d_{2m-3} \colon E^{0, 2m-4}_{2m-3} \cong \Z \langle \iota_{2m-4} \rangle \longrightarrow \Z \langle e_{2m-3} \rangle \cong E^{2m-3,0}_{2m-3}
    \]
    is an isomorphism.

    For $m > 3$, according to Table \ref{table:cohomEMS}, the only non-zero terms on the $(2m-1)$-th diagonal are
    \begin{equation*}
        E_2^{0, 2m-1} \cong \Z/2 \langle \delta Sq^2\iota^2_{2m-4} \rangle \hspace{3mm} \text{ and } \hspace{3mm} E_2^{2m-1,0}  \cong \Z \langle e_{2m-1} \rangle,
    \end{equation*}
    and the differentials coming both into and from the diagonal are trivial. See Figure \ref{fig:ss_F1_k=2}(A) for illustration. Therefore, $\co^{2m-1}(F[1]) \cong \Z$ is the non-trivial extension
    \[
        0 \longrightarrow \Z\langle e_{2m-1}\rangle \xrightarrow{~s_1^*~} \co^{2m-1}(F[1]) \longrightarrow  \Z/2\langle \delta Sq^2 \iota_{2m-4}^2\rangle \longrightarrow 0,
    \]
    as claimed. Moreover, since $m>3$ and by Table \ref{table:cohomEMS} we also have that there are no non-zero terms on the $(2m)$-th diagonal, hence $\co^{2m}(F[1]) \cong 0$.

    \begin{figure}[h]
\centering
\begin{subfigure}{.5\textwidth}
  \centering
  \includegraphics[width=.9\linewidth]{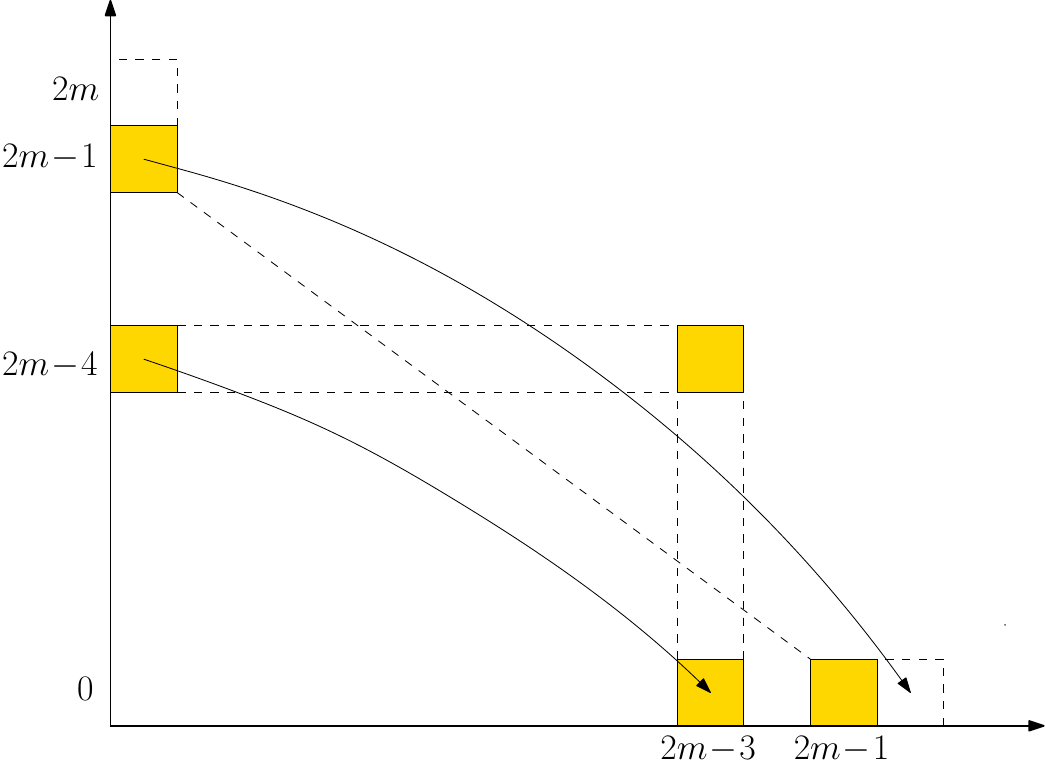}
  \caption{}
  \label{fig:sub1}
\end{subfigure}%
\begin{subfigure}{.5\textwidth}
  \centering
  \includegraphics[width=.83\linewidth]{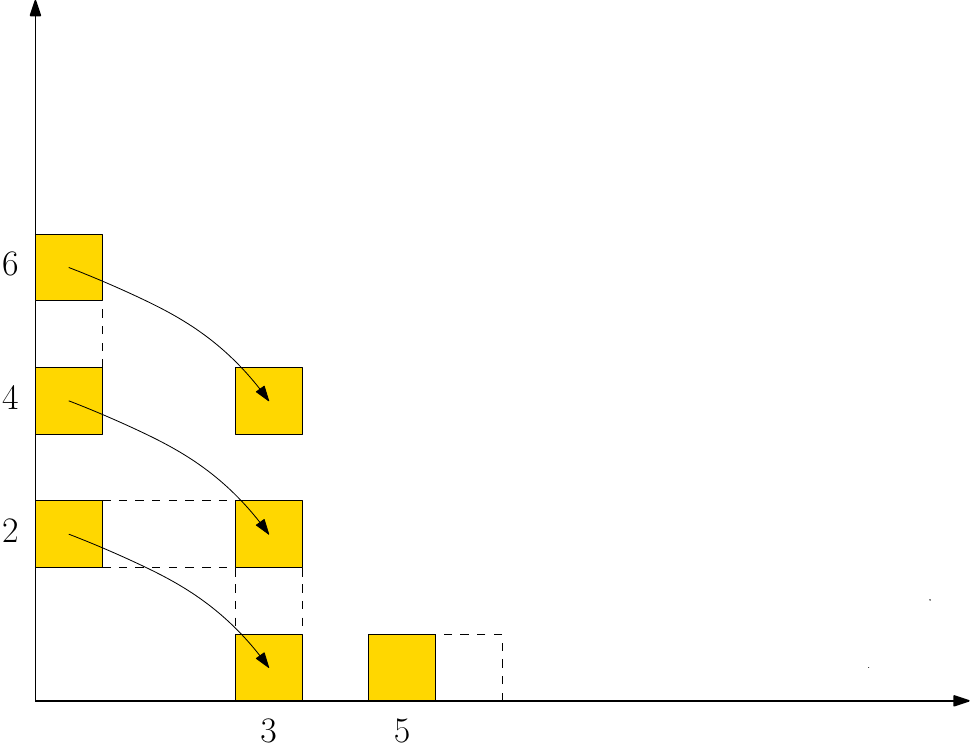}
  \caption{ }
  \label{fig:sub2}
\end{subfigure}
\caption{A portion of the Leray--Serre spectral sequence with $\Z$ coefficients of $s_1$ from \eqref{eq:s_1} for (A) $m>3$, and (B) $m=3$.}
\label{fig:ss_F1_k=2}
\end{figure}

    Finally, for $m=3$, we have that $\C P^{\infty}$ is a model for $K(\Z, 2)$, so the fibration \eqref{eq:s_1} becomes
    \begin{equation*} \label{eq:s_1-m=3}
        \C P^{\infty} \longrightarrow F[1]\overset{s_1}\longrightarrow W(3,2).
    \end{equation*}
    Let $\co^*(\C P^{\infty};\Z) \cong \Z[u]$, where $|u| = 2$ (so that $\iota_{2}$ becomes $u$). As already explained, we have that
    \begin{equation*}
        d_3 \colon~ E_3^{0,2} \cong \Z \langle u \rangle \xrightarrow{~\cong~} \Z \langle e_3 \rangle \cong E_3^{3,0}, \hspace{3mm} u \longmapsto e_3,
    \end{equation*}
    is an isomorphism. Moreover, since the differential $d_3$ is a derivation, we obtain
    \begin{equation*}
        d_3 \colon~ E_3^{0,4} \cong \Z \langle u^2 \rangle \xrightarrow{~\cong~}\Z \langle u \otimes e_3 \rangle \cong  E_3^{3,2}, \hspace{3mm} u^2 \longmapsto 2u \otimes e_3,
    \end{equation*}
    as well as
    \begin{equation*}
        d_3 \colon~ E_3^{0,6} \cong \Z \langle u^3 \rangle \xrightarrow{~\cong~}\Z \langle u^2 \otimes e_3 \rangle \cong  E_3^{4,2}, \hspace{3mm} u^3 \longmapsto 3u^2 \otimes e_3.
    \end{equation*}
    After passing to the $E_4$-page of the spectral sequence, we observe that the fourth and sixth diagonals are trivial, while the non-trivial slots on the fifth diagonal
    \begin{equation*}
        E_4^{3,2} \cong \Z/2 \langle u \otimes e_3 \rangle  \hspace{3mm} \text{ and } \hspace{3mm} E_4^{5,0} \cong \Z \langle e_5 \rangle
    \end{equation*}
    have only zero differentials emanating to or from them at any later page. See Figure \ref{fig:ss_F1_k=2}(B) for illustration. Therefore, we conclude that the fifth row remains unchanged until the $E_{\infty}$-page and that $\co^5(F[1]) \cong \Z$ is the non-trivial extension
    \begin{equation*}
        0 \longrightarrow \Z \langle e_5 \rangle \xrightarrow{~~s_1^*~~} \co^5(F[1]) \longrightarrow \Z/2 \langle u \otimes e_3 \rangle \longrightarrow 0.
    \end{equation*}
    Therefore, the claim follows in this case as well.
\end{proof}

This completes the proof of Theorem \ref{thm:secondaryob_cpspan2_odd}.

\section{Two complex line bundles: $m$ even} \label{sec:thm_cpspan2_even}

In this section we prove Theorem \ref{thm:cpspan2_even}.\\

Similarly as in Section \ref{sec:thm_cpspan2_odd}, let \(X\) be a \(2m\)-dimensional CW complex with \(m > 3\) even. Given complex vector bundles \[\xi \colon X \longrightarrow BU(m) \text{ and }\ell_1, \ell_2 \colon X \longrightarrow BU(1),\] we wish to lift the classifying map \((\xi, \ell_1, \ell_2) \colon X \to B(m,1^2)\) of (\ref{eq:xilinesclassifyingmap}) along the fibration \[W(m,2) \longrightarrow B(m-2,1^2) \overset{q_{m,2}}\longrightarrow B(m,1^2)\] defined in (\ref{eq:fibq_1}). To this end, we use Proposition \ref{prop:piW(m,r)} to construct the following Moore--Postnikov tower of \(q_{m,2}\),
\begin{equation}
\label{eq:mp_diag_cpspan2_even}
\begin{tikzcd}[ampersand replacement=\&]
	\&\& {E[3]} \\
	\&\& {E[2]} \& {K(\Z \oplus \Z/2,2m)} \\
	\&\& {E[1]} \& {K(\Z/2,2m-1)} \\
	{W(m,2)} \& {B(m-2,1^2)} \& {B(m,1^2)} \& {K(\Z,2m-2)} \\
	\&\& X
	\arrow["{p_3}", from=1-3, to=2-3]
	\arrow["{k_3}", from=2-3, to=2-4]
	\arrow["{p_2}", from=2-3, to=3-3]
	\arrow["{k_2}", from=3-3, to=3-4]
	\arrow["{p_1}", from=3-3, to=4-3]
	\arrow[from=4-1, to=4-2]
	\arrow["{q_2}"{description}, from=4-2, to=2-3]
	\arrow["{q_1}"{description}, from=4-2, to=3-3]
	\arrow["{q_{m,2}}", from=4-2, to=4-3]
	\arrow["{k_1}", from=4-3, to=4-4]
	\arrow["{(\xi, \ell_1,\ell_2)}"', from=5-3, to=4-3]
\end{tikzcd}\end{equation}
where each \(k_i\) is the characteristic class in the fibration \(q_{i-1}\) (with \(q_0 = q_{m,2}\)), and each \(E[i]\) is constructed as the homotopy fiber of \(k_i\).

\subsection{The Primary Obstruction}\label{subsec:primaryob_cpspan2_even}

By Proposition \ref{prop:primaryobs}, we obtain the following result. 

\begin{lem}\label{lem:primaryob_cpspan2_even}
Let all spaces and map be as in \emph{(\ref{eq:mp_diag_cpspan2_even})}. Then the primary obstruction 
\[
\lobs^{2m-2}((\xi,\ell_1,\ell_2),q_{m,2},k_1) \subseteq \co^{2m-2}(X;\Z)
\]
to lifting \((\xi,\ell_1,\ell_2)\) along \(q_{m,2}\) is the singleton set \(\{c_{m-1}(\xi- \ell_1\oplus \ell_2)\}.\)
\end{lem}

\subsection{The Generalized Secondary Obstruction}\label{subsec:secondaryob_cpspan2_even}
Throughout this section we will make use of Notation \ref{notate:1^k}(v) and Proposition \ref{prop:cohomology_rings}(vi). 

\begin{lem}\label{lem:secondaryob_cpspan2_even}
Let all spaces and maps be as in \emph{(\ref{eq:mp_diag_cpspan2_even})}. Suppose that the classifying map \[(\xi,\ell_1,\ell_2) \colon X \to B(m,1^2)\] satisfies \(c_{m-1}(\xi-\ell_1 \oplus \ell_2) = 0.\) If \(\co^{2m-1}(X;\Z/2) = Sq^2 \rho_2 \co^{2m-3}(X;\Z)\), then the generalized secondary obstruction 
\[
\lobs^{2m-1}((\xi,\ell_1,\ell_2),q_{m,2},k_2) \subseteq \co^{2m-1}(X;\Z/2)
\]
to lifting \((\xi,\ell_1,\ell_2)\) along \(q_{m,2}\) contains \(0.\)
\end{lem}
\begin{proof}
Throughout the proof, all cohomology will be assumed to be with $\Z/2$ coefficients.
Let the class \(k_2 \in \co^{2m-1}(E[1])\) be as in (\ref{eq:mp_diag_cpspan2_even}).

First, we will compute $\co^{2m-1}(E[1])$ using the mod \(2\) Leray--Serre spectral sequence \((E_*^{*,*},d_*)\) of the fibration
\begin{equation} \label{eq:p_1_even2}
    K(\Z,2m-3) \longrightarrow E[1]\overset{p_1}\longrightarrow B(m,1^2).
\end{equation}
By construction, we get that
\[
    d_{2m-2} \colon~ \Z/2 \langle \iota^2_{2m-3} \rangle \cong E_{2m-2}^{0,2m-3} \longrightarrow E_{2m-2}^{2m-2,0} \cong \co^{2m-2}(B(m,1^2)),~\iota_{2m-3}^2 \longmapsto b_{m-1}
\]
Moreover, since the differentials are derivations, we observe that $d_{2m-2} \colon E_{2m-2}^{2,2m-3} \to E_{2m-2}^{2m-2,0}$ satisfies
\begin{equation} \label{eq:ss_d_(2m-2)_nontriv}
    d_{2m-2}(\iota_{2m-3}^2\otimes u) = d_{2m-2}(\iota_{2m-3}^2)u = b_{m-1}u\neq 0 \in E_{2m-2}^{2m,0} \cong \co^{2m}(B(m,1^2)),
\end{equation}
for all \(u \in \co^{2}(B(m,1^2))\), see Figure \ref{fig:ss_p1-and-p2-_mod_2_even}(A) for an illustration. In particular, we get $E_{\infty}^{2, 2m-3} = 0$.

Next, by Kudo's transgression theorem \cite{kudo1956transgression} and the Wu formula (see also Proposition \ref{prop:cohomology_rings}(vi)), we calculate that $d_{2m} \colon E_{2m}^{0,2m-1} \to E_{2m}^{2m, 0}$ satisfies
\begin{equation} \label{eq:ss_d_(2m)_triv}
    d_{2m}(Sq^2 \iota_{2m-3}^2) = Sq^2 b_{m-1} = b_1b_{m-1}+ (m-2) b_m = 0 \in E_{2m}^{2m, 0} \cong \co^{2m}(B(m,1^2))/(b_{m-1}),
\end{equation}
due to $m$ being even. Thus \[\Z/2 \langle Sq^2 \iota^2_{2m-3}\rangle \cong E_{2m}^{0,2m-1} = E_{\infty}^{0,2m-1}\]
is the only non-zero term on the $(2m-1)$-th diagonal, hence \(\co^{2m-1}(E[1]) \cong \Z/2 \langle Sq^2 \iota_{2m-3}^2\rangle\).
Moreover, from \eqref{eq:ss_d_(2m)_triv} it also follows that $d_{2m} \colon E_{2m}^{2,2m-1} \to E_{2m}^{2m+2, 0}$ satisfies
\begin{equation*}
    d_{2m}(Sq^2 \iota_{2m-3}^2 \otimes u) = d_{2m}(Sq^2 \iota_{2m-3}^2) u = 0 \in E_{2m}^{2m+2, 0},
\end{equation*}
for any $u \in \co^2(B(m,1^2))$. This means that the only non-trivial slots on the $(2m+1)$-th diagonal on the $(2m+1)$-th page are
\[
    E^{0,2m+1}_{2m+1} \cong E^{0,2m+1}_{2} ~ \text{and} ~ E^{2,2m-1}_{2m+1} \cong E^{2,2m-1}_{2}.
\]
Passing to the next page, we can analogously show (since $m$ is even and thus $\tbinom{2m-5}{2}$ is odd) that 
\[
    d_{2m+2} \colon E_{2m+2}^{0,2m+1} \longrightarrow E_{2m+2}^{2m+2, 0} \cong \Big(\co^{*}(B(m,1^2))/(b_{m-1})\Big)^{(2m+2)},
\]
satisfies
\begin{equation*}
    d_{2m+2}(Sq^4 \iota^2_{2m-3}) = Sq^4 b_{m-1} = b_2b_{m-1} + \binom{2m-5}{2}b_1b_m + \binom{2m-3}{4}b_{m+1} \neq 0 \in E_{2m+2}^{2m+2, 0}.
\end{equation*}
Therefore, $E_{2m+3}^{2,2m-1} = E_{\infty}^{2,2m-1}$ remains the only non-trivial slot on the $(2m+1)$-th diagonal (see again Figure \ref{fig:ss_p1-and-p2-_mod_2_even}(A)), so
\begin{equation} \label{eq:H2m+1(E1)}
    \co^{2m+1}(E[1]) \cong \Z/2\langle Sq^2 \iota^2_{2m-3}\rangle \otimes \co^2(B(m,1^2)).
\end{equation}
Now, to show that $k_2$ is the generator, we observe that the transgression homomorphism in the mod 2 Leray--Serre spectral sequence of the homotopy fibration 
\[
F[1] \longrightarrow B(m-2,1^2) \overset{q_1}\longrightarrow E[1]
\]
must be surjective since 
\(\co^{2m-1}(B(m-2,1^2)) \cong 0\). Consequently, 
\begin{equation}\label{eq:k_2_cpspan2_even}
    k_2 = Sq^2 \iota_{2m-3}^2 \in \co^{2m-1}(E[1]),
\end{equation}
as claimed.

It remains to compute the indeterminacy of lifts to \(E[2].\) The proceeding arguments are similar to those in the proof of Theorem \ref{thm:secondaryob_cpspan2_odd}; accordingly, we will be less detailed in our arguments. To begin, consider the following homotopy commutative diagram:
\[\begin{tikzcd}
	{K(\Z,2m-3)} && {K(\Z,2m-3)} \\
	{K(\Z,2m-3)\times B(m-2,1^2)} & {K(\Z,2m-3)\times E[1]} & {E[1]} \\
	{B(m-2,1^2)} && {B(m,1^2)}
	\arrow[Rightarrow, no head, from=1-1, to=1-3]
	\arrow[from=1-1, to=2-1]
	\arrow[from=1-3, to=2-3]
	\arrow["{1\times q_1}", from=2-1, to=2-2]
	\arrow["\nu", curve={height=-30pt}, from=2-1, to=2-3]
	\arrow["{\mathrm{proj}_2}", from=2-1, to=3-1]
	\arrow["\mu", from=2-2, to=2-3]
	\arrow["{p_1}", from=2-3, to=3-3]
	\arrow["{\bar{s}}", curve={height=-6pt}, from=3-1, to=2-1]
	\arrow["{q_{m,2}}", from=3-1, to=3-3]
\end{tikzcd}\]
We also have an exact sequence
\[
\dots \longrightarrow \mathrm{ker}(q_1^*) \cap\co^{2m-1}(E[1]) \overset{\nu^*}\longrightarrow \mathrm{ker}(\bar{s}^*) \cap\co^{2m-1}(K(\Z,2m-3)\times B(m-2,1^2)) \overset{\tau_1}\longrightarrow\co^{2m}(B(m,1^2)) \longrightarrow \dots.
\]
Briefly, we wish to compute \(\mu^*(k_2) = 1 \otimes k_2 + \nu^*(k_2)\), where
\[
\nu^*(k_2) \in \mathrm{ker}(\bar{s}) \cap \mathrm{ker}(\tau_1) \cap \co^{2m-1}(K(\Z,2m-3)\times B(m-2,1^2)).
\]
By direct calculation, we see that
\(\mathrm{ker}(\bar{s}) \cap \co^{2m-1}(K(\Z,2m-3)\times B(m-2,1^2))\) is generated by \(\iota_{2m-3}^2 \otimes u\), for all \(u \in \co^{2}(B(m-1,1^2))\), and \(Sq^2 \iota_{2m-3}\otimes 1.\)
To compute the action of the relative transgression on these elements, we use our recently acquired knowledge of the mod \(2\) Leray--Serre spectral sequence $(E_*^{*,*}, d_*)$ of \(p_1 \colon E[1]\to B(m,1^2)\) together with properties \cite[Property 1 \& 2, p.\ 14]{thomas2006seminar} of the relative transgression \(\tau_1.\) Namely, combining \eqref{eq:ss_d_(2m-2)_nontriv} and \eqref{eq:ss_d_(2m)_triv}, we conclude that
\[
    \mathrm{ker}(\bar{s}) \cap \mathrm{ker}(\tau_1) \cap \co^{2m-1}(K(\Z,2m-3)\times B(m-2,1^2)) \cong \Z/2 \langle Sq^2 \iota_{2m-3}\otimes 1 \rangle.
\]
Finally, as in the proof of Corollary \ref{cor:indet_cpspan2_odd}, we get \(\mathrm{Indet} = Sq^2 \rho_2 \co^{2m-3}(X;\Z).\) The result follows. 
\end{proof}

\begin{figure}
\centering
\begin{subfigure}{.5\textwidth}
  \centering
  \includegraphics[width=.9\linewidth]{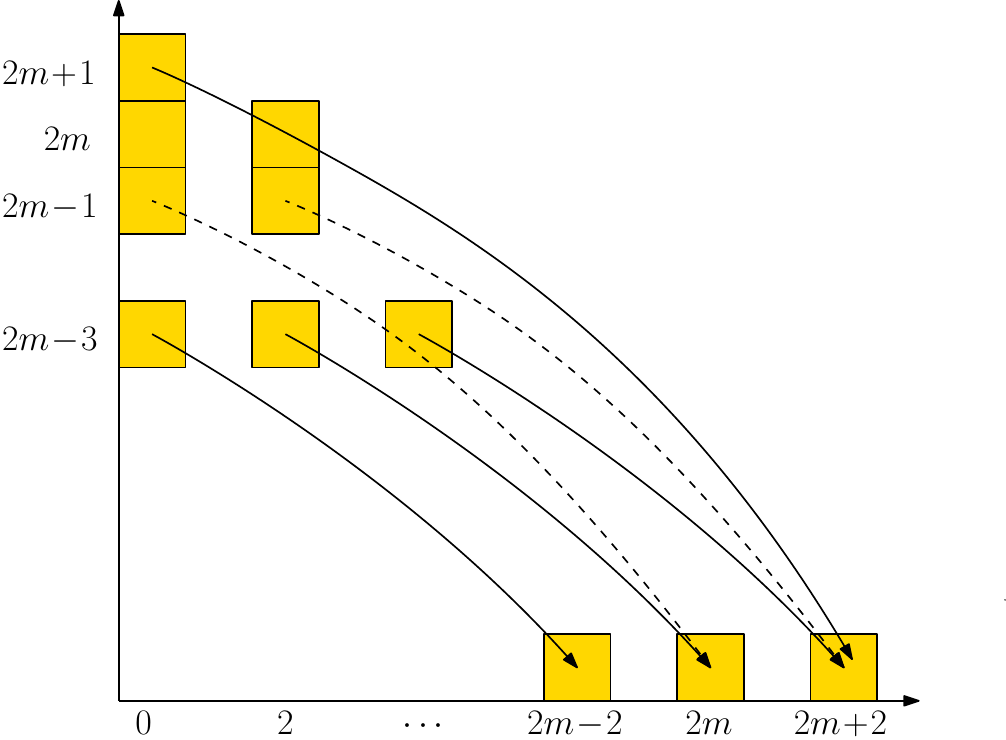}
  \caption{}
  \label{fig:sub1}
\end{subfigure}%
\begin{subfigure}{.5\textwidth}
  \centering
  \includegraphics[width=.9\linewidth]{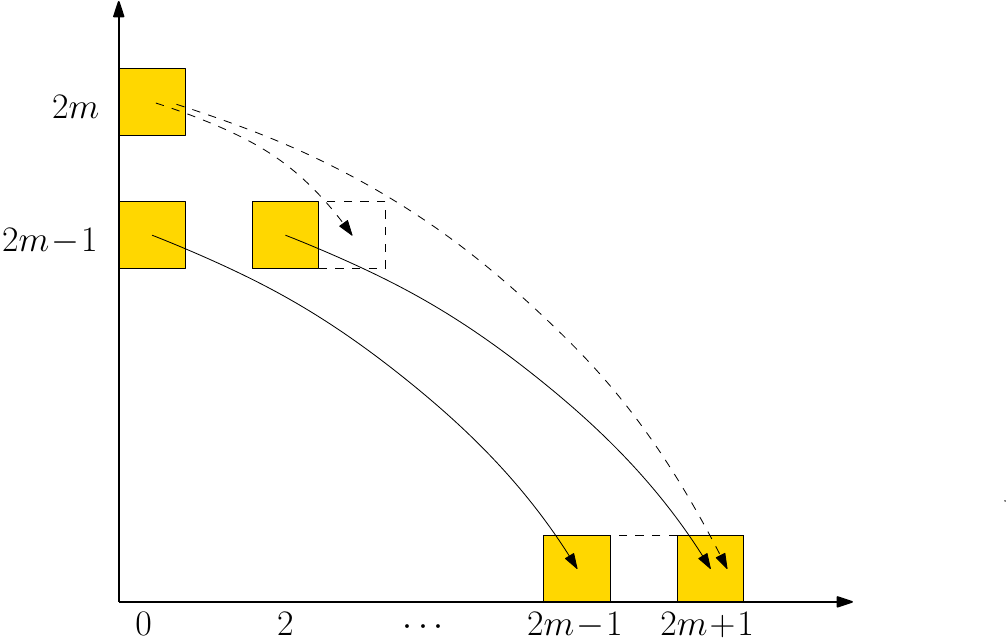}
  \caption{ }
  \label{fig:sub2}
\end{subfigure}
\caption{A portion of the mod 2 Leray--Serre spectral sequence of (A) $p_1$ from \eqref{eq:p_1_even2}, and (B) $p_2$ from \eqref{eq:p_2_even2}, where dashed arrows represent zero differentials.}
\label{fig:ss_p1-and-p2-_mod_2_even}
\end{figure}

\subsection{The Generalized Tertiary Obstruction}\label{subsec:tertiaryob_cpspan2_even}

Now Theorem \ref{thm:cpspan2_even} is an immediate consequence of the following theorem. 

\begin{thm}\label{thm:tertiaryob_cpspan2_even}
Let all spaces and maps be as in \emph{(\ref{eq:mp_diag_cpspan2_even})}. Suppose that the classifying map \[(\xi,\ell_1,\ell_2) \colon X \longrightarrow B(m,1^2)\] satisfies \(c_{m-1}(\xi-\ell_1 \oplus \ell_2) = 0\). Further suppose that all of the following conditions hold.
\begin{itemize}
    \item \(\co^{2m-1}(X;\Z/2) = Sq^2 \rho_2 \co^{2m-3}(X;\Z);\)
    \item \(\co^{2m}(X;\Z/2) = Sq^2\co^{2m-2}(X;\Z/2);\) and 
    \item \(\co^{2m}(X;\Z)\) has no \(2\)-torsion.
\end{itemize}
Then the generalized tertiary obstruction 
\[
\lobs^{2m}((\xi,\ell_1,\ell_2),q_{m,2},k_3) \subseteq \co^{2m}(X;\Z \oplus \Z/2)
\]
to lifting \((\xi,\ell_1,\ell_2)\) along \(q_{m,2}\) is the singleton set \(\{(c_{m}(\xi - \ell_1 \oplus \ell_2),0)\}.\)
\end{thm}

The rest of the section is dedicated to proving Theorem \ref{thm:tertiaryob_cpspan2_even}.\\

By Lemmata \ref{lem:primaryob_cpspan2_even} and \ref{lem:secondaryob_cpspan2_even}, there is a lift making the following diagram commute
\[\begin{tikzcd}
	& & {B(m-2,1^2)} \\
	X & & {B(m,1^2)}
	\arrow["{q_{m,2}}", from=1-3, to=2-3]
	\arrow[dashed, from=2-1, to=1-3]
	\arrow["{(\xi,\ell_1,\ell_2)}", from=2-1, to=2-3]
\end{tikzcd}\]
if and only if
\[
g^*(k_3) = 0 \in \co^{2m}(X;\Z)/\mathrm{Indet}_\Z \oplus \co^{2m}(X;\Z/2)/\mathrm{Indet}_{\Z/2},
\]
where \(g\colon X \to E[2]\) is any lift of \((\xi,\ell_1,\ell_2)\) to \(E[2]\), and \(\mathrm{Indet}_R\) denotes the indeterminacy of such lifts with respect to the coefficient ring \(R \in \{\Z,\Z/2\}.\) Towards a proof of Theorem \ref{thm:tertiaryob_cpspan2_even}, we compute these indeterminacy subgroups.
Consider the usual homotopy commutative diagram:
\begin{equation}\label{eq:diage[2]}
\begin{tikzcd}
	{K(\Z/2,2m-2)} && {K(\Z/2,2m-2)} \\
	{K(\Z/2,2m-2)\times B(m-2,1^2)} & {K(\Z/2,2m-2)\times E[2]} & {E[2]} \\
	{B(m-2,1^2)} && {E[1]}
	\arrow[Rightarrow, no head, from=1-1, to=1-3]
	\arrow[from=1-1, to=2-1]
	\arrow[from=1-3, to=2-3]
	\arrow["{1 \times q_2}", from=2-1, to=2-2]
	\arrow["\nu"{description}, curve={height=-30pt}, from=2-1, to=2-3]
	\arrow["{{\mathrm{proj}}_2}", from=2-1, to=3-1]
	\arrow["\mu", from=2-2, to=2-3]
	\arrow["{p_2}", from=2-3, to=3-3]
	\arrow["{\bar{s}}", curve={height=-18pt}, from=3-1, to=2-1]
	\arrow["{q_2}", from=3-1, to=3-3]
\end{tikzcd}\end{equation}
In this section, we henceforth write \(K \coloneqq K(\Z/2,2m-2)\) and \(B \coloneqq B(m-2,1^2).\) We have an exact sequence,
\begin{equation}\label{eq:tau_1_a}
\dots \longrightarrow \mathrm{ker}(q_2^*) \cap\co^{2m}(E[2]) \overset{\nu^*}\longrightarrow \mathrm{ker}(\bar{s}^*) \cap\co^{2m}(K\times B) \overset{\tau_1}\longrightarrow \co^{2m+1}(E[1]) \longrightarrow \dots ,
\end{equation}
where all cohomology groups have \(\Z \oplus \Z/2\) coefficients. By standard arguments, \[\nu^*(k) \in {\mathrm{ker}}(\bar{s}^*) \cap \mathrm{ker}(\tau_1)\cap \co^{2m}(K\times B;\Z\oplus \Z/2).\] Thereupon, we wish to compute this intersection. As a first step, using the cohomological K\"{u}nneth formula together with Proposition \ref{prop:cohomK(Z/p,q)} and Serre's classical computation of \(\co^*(K(\Z/2,2m-2);\Z/2)\) \cite[Theorem 3, p.\ 90]{mosher2008cohomology}, we obtain the following proposition.
\begin{prop}\label{prop:kunneth}
The cohomology of \(K\times B\) satisfies the following.
\begin{enumerate}[\normalfont(i)]
    \item The group
\(\co^{2m}(K\times B;\Z)\) is isomorphic to \(\co^0(K;\Z) \otimes \co^{2m}(B;\Z).\)
\item The group
\(\co^{2m}(K\times B;\Z/2)\) is isomorphic to \[\co^{0}(K;\Z/2) \otimes \co^{2m}(B;\Z/2) \oplus \co^{2m-2}(K;\Z/2) \otimes \co^2(B;\Z/2) \oplus \co^{2m}(K;\Z/2) \otimes \co^0(B;\Z/2).\]
\end{enumerate}
\end{prop}

Combining the above, we procure the following lemma. Below, $\kappa_{2m-2} \in \co^{2m-2}(K(\Z/2, 2m-2);\Z/2)$ denotes the mod 2 fundamental class (see Notation \ref{notate:cohomologyops}(vii)).
\begin{lem}\label{lem:kernelsbar}
Let \(\overline{s}\) be the canonical section of the projection \(\mathrm{proj}_2 \colon K \times B \to  B\) as defined in \emph{(\ref{eq:diage[2]})}. Further let \(\tau_1\) be the relative transgression of \emph{\eqref{eq:tau_1_a}}. Then  
\[
    \mathrm{ker}(\bar{s})^*\cap \mathrm{ker}(\tau_1) \cap \co^{2m}(K\times B;\Z\oplus \Z/2)
\]
is generated by \((0, Sq^2 \kappa_{2m-2} \otimes 1)\).
\end{lem}

\begin{proof}
It is straightforward to see that \(\mathrm{ker}(\bar{s}^*)\cap \co^{2m}(K \times B;\Z\oplus \Z/2)\) is generated by the elements 
\begin{itemize}
    \item \((0, Sq^2 \kappa_{2m-2} \otimes 1)\), and 
    \item \((0, \kappa_{2m-2} \otimes u)\), for all \(u \in \co^{2}(B(m-2,1^2);\Z/2).\)
\end{itemize}
To compute the action of the relative transgression on these elements, consider the mod \(2\) Leray--Serre spectral sequence \((E_*^{*,*},d_*)\) of the homotopy fibration 
\begin{equation}\label{eq:p_2_even2}
    K(\Z/2,2m-2) \longrightarrow E[2] \overset{p_2}\longrightarrow E[1].
\end{equation}
Using \cite[Property 1 \& 2]{thomas2006seminar}, we compute the following, which is illustrated in Figure \ref{fig:ss_p1-and-p2-_mod_2_even}(B).
Namely, by construction, we have \[d_{2m-1} \colon~ E_{2m-1}^{0,2m-2} \longrightarrow E_{2m-1}^{2m-1,0},~\kappa_{2m-2} \longmapsto k_2.\] From \eqref{eq:k_2_cpspan2_even} and the properties of that spectral sequence it follows that ${d_{2m-1} \colon~} E_{2m-1}^{2,2m-2} \to E_{2m-1}^{2m+1,0}$ satisfies
\[
        d_{2m-1}(\kappa_{2m-2} \otimes u) = k_2  u = (Sq^2 \iota_{2m-3}^2) u \neq 0 \in \co^{2m+1}(E[1];\Z/2),
\]
    for all \(u \in \co^2(B(m-2,1^2);\Z/2)\). The element $(Sq^2 \iota_{2m-3}^2) u$ is indeed non-zero due to \eqref{eq:H2m+1(E1)} (see also Figure \ref{fig:ss_p1-and-p2-_mod_2_even}(A)).
Finally, by Kudo's transgression formula, we see that for $d_{2m+1} \colon~ E_{2m+1}^{0,2m} \to E_{2m+1}^{2m+1,0}$ we get
    \[d_{2m+1}(Sq^2 \kappa_{2m-2}) = Sq^2Sq^2 \iota_{2m-2}^2 =0 \in E_{2m+1}^{*,0},\] since \(Sq^2 Sq^2 \iota_{2m-2}^2 = Sq^3 Sq^1 \iota_{2m-2}^2 = 0\). Thus \(\tau_1(Sq^2 \kappa_{2m-2}\otimes 1) = 0.\)
The result follows. 
\end{proof}

Using Lemma \ref{lem:kernelsbar}, we compute the image of \(k_3\) under the principal action map. 
\begin{cor}\label{cor:3} Let \(\mu \colon K(\Z/2,2m-2)\times E[2]\to E[2]\) be the principal action map associated to the principal fiber space \(E[2]\). Let \(k_3 \in \co^{2m}(E[2];\Z \oplus \Z/2)\) as in \emph{(\ref{eq:mp_diag_cpspan2_even})}. Then \[\mu^*(k_3) = 1 \otimes k_3 + (0, Sq^2 \kappa_{2m-2}  \otimes 1).\]
\end{cor}

Finally, Theorem \ref{thm:tertiaryob_cpspan2_even} follows immediately from the proceeding lemma. 

\begin{lem}\label{lem:indet_cpspan2_even}
Let all maps and spaces be as defined in \emph{(\ref{eq:mp_diag_cpspan2_even})}. Suppose that the classifying map \((\xi, \ell_1, \ell_2) \colon X \to B(m-2,1^2)\) satisfies \(c_{m-1}(\xi-\ell_1 \oplus \ell_2)=0\). Then, as \(g\) runs over lifts of \((\xi,\ell_1,\ell_2)\) to \(E[2]\), the integral part of all the pullbacks \(g^*(k_3)\) are the same element in \(\co^{2m}(X;\Z)\) with zero indeterminacy. Moreover, the mod 2 part of the pullbacks \(g^*(k_3)\) is contained in a single coset in \(\co^{2m}(X;\Z/2)\) of the subgroup
\begin{equation*}\label{eq:indetoddmod2}
{\mathrm{Indet}}_{\Z/2} = Sq^2\co^{2m-2}(X;\Z/2).
\end{equation*}
\end{lem}

\begin{proof}
Let \(g_1, g_2 \colon X \to E[2]\) be two lifts of \((\xi,\ell_1,\ell_2)\) to \(E[2]\). Then there exists an element \(\alpha \in \co^{2m-2}(X;\Z/2)\) such that the difference 
\[
g_2^*(k_3) - g_1^*(k_3) \in \co^{2m}(X;\Z \oplus \Z_2) \cong \co^{2m}(X;\Z) \oplus \co^{2m}(X;\Z/2)
\]
is \((0,Sq^2 \alpha).\) The result follows. 
\end{proof}

Since \(\mathrm{Indet}_{\Z} =0\), there is a unique integral tertiary obstruction to lifting \((\xi, \ell_1, \ell_2)\) along \(q_{m,2}\) which is identified in the proceeding lemma. To state the result, we use Notation \ref{notate:1^k}(v).

\begin{lem}\label{lem:k_3_odd}
Let \(k_3 \in \co^{2m}(E[2];\Z\oplus \Z/2)\) be the characteristic class in the fibration \[q_2 \colon B(m-2,1^2) \to E[2]\] as defined in \emph{(\ref{eq:mp_diag_cpspan2_even})}. Suppose that \(\co^{2m}(X;\Z)\) has no \(2\)-torsion and that \(\co^{2m}(X;\Z/2) \cong Sq^2\co^{2m-2}(X;\Z/2)\). Then, given any lift \(g \colon X \to E[2]\) of \((\xi,\ell_1,\ell_2)\) in \emph{(\ref{eq:mp_diag_cpspan2_even})}, \(g^*(k_3)\) vanishes if and only if \(c_m(\xi-\ell_1 \oplus \ell_2)\) vanishes.
\end{lem}

The proof of Lemma \ref{lem:k_3_odd} is similar to the proof of Lemma \ref{lem:k_2_even}, hence we omit the details for brevity. This completes the proof of Theorem \ref{thm:tertiaryob_cpspan2_even}, which in turn completess the proof of Theorem \ref{thm:cpspan2_even}.
\section{Three complex line bundles: $m$ even}\label{sec:thm_pspan3_even}

In this section we prove Theorem \ref{thm:cpspan3_even}.\\

Let \(X\) be a \(2m\)-dimensional CW complex with \(m > 5\) even. Given complex vector bundles \[\xi \colon X \to BU(m) \text{ and } \ell_1,\ell_2,\ell_3 \colon X \to BU(1),\] we wish to lift the classifying map \((\xi,\ell_1,\ell_2,\ell_3) \colon X \to B(m,1^3)\) of (\ref{eq:xilinesclassifyingmap}) along the fibration \[W(m,3)\longrightarrow B(m-3,1^3)\overset{q_{m,3}}\longrightarrow B(m,1^3)\] defined in (\ref{eq:fibq_1}). Using Proposition \ref{prop:piW(m,r)}, we construct the following Moore--Postnikov factorization of \(q_{m,3},\)

\begin{equation}\label{eq:mp_diag_cpspan3_even}
\begin{tikzcd}[ampersand replacement=\&]
	\&\& {E[4]} \\
	\&\& {E[3]} \& {K(\Z,2m)} \\
	\&\& {E[2]} \& {K(\pi_{2m-2},2m-1)} \\
	\&\& {E[1]} \& {K(\Z,2m-2)} \\
	{W(m,3)} \& {B(m-3,1^3)} \& {B(m,1^3)} \& {K(\Z,2m-4)} \\
	\\
	\&\& X
	\arrow["{p_4}", from=1-3, to=2-3]
	\arrow["{k_4}", from=2-3, to=2-4]
	\arrow["{p_3}", from=2-3, to=3-3]
	\arrow["{k_3}", from=3-3, to=3-4]
	\arrow["{p_2}", from=3-3, to=4-3]
	\arrow["{k_2}", from=4-3, to=4-4]
	\arrow["{p_1}", from=4-3, to=5-3]
	\arrow[from=5-1, to=5-2]
	\arrow["{q_3}"{description}, from=5-2, to=2-3]
	\arrow["{q_2}"{description}, from=5-2, to=3-3]
	\arrow["{q_1}"{description}, from=5-2, to=4-3]
	\arrow["{q_{m,3}}", from=5-2, to=5-3]
	\arrow["{k_1}", from=5-3, to=5-4]
	\arrow["{(\xi,\ell_1,\ell_2,\ell_3)}"', from=7-3, to=5-3]
\end{tikzcd}
\end{equation}
where each \(k_i\) is the characteristic class in the fibration \(q_{i-1}\) (for \(q_0 = q_{m,3}\)); and each \(E[i]\) is the homotopy fiber of \(k_i\). Lastly, recall that \(\pi_{2m-2} = \pi_{2m-2}(W(m,3))\) satisfies Table \ref{table:pi2m-2W(m,3)}.

\subsection{The Primary Obstruction}\label{subsec:primobs_cpspan3_even}

Applying Proposition \ref{prop:primaryobs}, we procure the following result. 
\begin{lem}\label{lem:primarybos_cpspan3_even}
Let all spaces and maps be as in \emph{(\ref{eq:mp_diag_cpspan3_even})}. Then the primary obstruction 
\[
\lobs^{2m-4}((\xi,\ell_1,\ell_2,\ell_3),q_{m,3},k_1) \subseteq \co^{2m-4}(X;\Z)
\]
to lifting \((\xi,\ell_1,\ell_2,\ell_3)\) along \(q_{m,3}\) is the singleton set \(\{c_{m-2}(\xi - \ell_1 \oplus \ell_2 \oplus \ell_3)\}.\)
\end{lem}

\subsection{The Generalized Secondary Obstruction}\label{subsec:secondaryobs_cpspan3_even}

By similar arguments to the proof of Theorem \ref{thm:cpspan2_odd}, \textit{mutatis mutandis},  we achieve the following result. 

\begin{lem}\label{lem:secondaryobs_cpspan3_even}
Let all spaces and maps be as in \emph{(\ref{eq:mp_diag_cpspan3_even})}. Suppose that the classifying map \[(\xi,\ell_1,\ell_2,\ell_3) \colon X \longrightarrow B(m,1^3)\] satisfies \(c_{m-2}(\xi - \ell_1\oplus \ell_2\oplus \ell_3)=0\).
\begin{enumerate}[\normalfont(i)]
    \item Further suppose that \(\co^{2m-2}(X;\Z) = \delta Sq^2\rho_2 \co^{2m-5}(X;\Z).\) Then zero is an element of the generalized secondary obstruction
    \[
\lobs^{2m-2}((\xi,\ell_1,\ell_2,\ell_3), q_{m,3},k_2)\subseteq \co^{2m-2}(X;\Z)
    \]
to lifting \((\xi,\ell_1,\ell_2,\ell_3)\) along \(q_{m,3}\).
    \item Further suppose that \(\co^{2m-2}(X;\Z)\) has no \(2\)-torsion. Then the generalized secondary obstruction 
\[
\lobs^{2m-2}((\xi,\ell_1,\ell_2,\ell_3), q_{m,3},k_2) \subseteq \co^{2m-2}(X;\Z)\]
to lifting \((\xi,\ell_1,\ell_2,\ell_3)\) along \(q_{m,3}\) is the singleton set \(\{c_{m-1}(\xi-\ell_1\oplus \ell_2\oplus \ell_3)\}.\)
\end{enumerate}
\end{lem}

\subsection{The Generalized Tertiary Obstruction}\label{subsec:tertiaryobs_cpspan3_even}

Next, assuming that zero is an element of both \[\lobs^{2m-4}((\xi,\ell_1,\ell_2,\ell_3), q_{m,3},k_1) \text{ and } \lobs^{2m-2}((\xi,\ell_1,\ell_2,\ell_3), q_{m,3},k_2),\] we will determine the generalized tertiary obstruction to lifting.

\begin{lem}\label{lem:2m-1_obs_is_0}
Let \(m > 5\) be an even integer and let all spaces and maps be as in \emph{(\ref{eq:mp_diag_cpspan3_even})}. Suppose that \(c_{m-2}(\xi- \ell_1 \oplus \ell_2 \oplus \ell_3) = 0\) and that condition \emph{Lemma \ref{lem:secondaryobs_cpspan3_even}(i)} or \emph{Lemma \ref{lem:secondaryobs_cpspan3_even}(ii)} is satisfied. If any one of the following conditions \emph{(i)--(iv)} hold, then zero is an element of the generalized tertiary obstruction 
\[
\lobs^{2m-1}((\xi,\ell_1,\ell_2,\ell_3),q_{m,3},k_3)\subseteq \co^{2m-1}(X;\pi_{2m-2}).
\] 
\begin{enumerate}[\normalfont(i)]
\item \(m \equiv 0~\mathrm{mod}~8\) and \(\co^{2m-1}(X;\Z/4) = 0\);
\item \(m \equiv 2~\mathrm{mod}~8\);
\item \(m \equiv 4~\mathrm{mod}~8\) and \(\co^{2m-1}(X;\Z/2) = 0\); or
\item \(m \equiv 6~\mathrm{mod}~8\).
\end{enumerate}
\end{lem}

The rest of this section is dedicated to proving Lemma \ref{lem:2m-1_obs_is_0}.\\

\paragraph{\textbf{Determining the Moore--Postnikov invariant.}} Towards a proof of Lemma \ref{lem:2m-1_obs_is_0}, we compute the degree \((2m-1)\) Moore--Postnikov invariant \(k_3 \in \co^{2m-1}(E[2];\pi_{2m-2})\) as follows.

\begin{prop}\label{prop:2m-1-mpinvariant}
Let \(k_3 \in \co^{2m-1}(E[2];\pi_{2m-2})\) be as in \emph{(\ref{eq:mp_diag_cpspan3_even})}. Then the mod \(p\) part of \(k_3\) is equal to
\[
    \begin{cases}
        Sq^4 \iota_{2m-5}^2  & p=2 \textrm{ and } m = 4~\mathrm{mod}~8\\
        \theta_2^2 Sq^4 \iota_{2m-5}^2 & p=4 \textrm{ and } m = 0 ~\mathrm{mod}~8\\
        b^{(3)}_{m-1}\otimes \iota_{2m-3}^3 & p=3 \textrm{ and } m = 0~\mathrm{mod}~3.
    \end{cases}
\] 
\end{prop}

\begin{proof}
    We study the Leray--Serre spectral sequence $(E_*^{*,*},d_*)$ with $\Z/p$ coefficients of the homotopy fibration
    \begin{equation*}
        F[2] \longrightarrow B(m-3,1^3) \xrightarrow{~q_2~} E[2].
    \end{equation*}
    Recall that the fiber $F[2]$ is $(2m-3)$-connected. The differential
    \[
        d_{2m-1} \colon \co^{2m-2}(F[2];\Z/p) \cong E_{2m-1}^{0,2m-2} \longrightarrow E_{2m-1}^{2m-1, 0} \cong \co^{2m-1}(E[2];\Z/p)
    \]
    on the $(2m-1)$-th page of $E^*_{*,*}$ has to be surjective, since $\co^{2m-1}(B(m-3,1^3);\Z/p) = 0$. 
    Thus, the mod $p$ part of $k_3$ is the generator of $\co^{2m-1}(E[2];\Z/p)$ and the claim follows from Lemma \ref{lem:2m-1-coh-of-E[2]}.
\end{proof}

\begin{lem}\label{lem:2m-1-coh-of-E[2]}
Let \(E[2]\) be as in \emph{(\ref{eq:mp_diag_cpspan3_even})}. Then
\[
    \co^{2m-1}(E[2];\Z/p) = \begin{cases}
        \Z/2 \langle Sq^4 \iota_{2m-5}^2 \rangle  & p=2 \textrm{ and } m = 4~\mathrm{mod}~8\\
        \Z/2\langle \theta_2^2 Sq^4 \iota^2_{2m-5}\rangle & p=4 \textrm{ and } m = 0 ~\mathrm{mod}~8\\
        \Z/3 \langle b^{(3)}_{1}\otimes \iota_{2m-3}^3\rangle. & p=3 \textrm{ and } m = 0~\mathrm{mod}~3,
    \end{cases}
\] 
where \(\theta_2^2\) is the coefficient homomorphism induced by the inclusion \(\Z/2 \xhookrightarrow{}\Z/4\) as in \emph{Notation \ref{notate:cohomologyops}}.
\end{lem}

The first in the proof of Lemma \ref{lem:2m-1-coh-of-E[2]} are the following technical claims regarding the structure of cohomology of $E[1]$. We recall the fibration
\begin{equation} \label{eq:p1}
    K(\Z,2m-5) \overset{h}\longrightarrow E[1] \overset{p_1}\longrightarrow B(m,1^3),
\end{equation}
as defined in \eqref{eq:mp_diag_cpspan3_even}.

\begin{lem}\label{lem:2m-1-coh-of-E[1]}
Let \(E[1]\) be as in \emph{(\ref{eq:mp_diag_cpspan3_even})}. Its cohomology satisfies the following.
\begin{enumerate}[\normalfont(i)]
    \item Let $p=2$ and $m = 4~\mathrm{mod}~8$. We have
    \begin{equation*}
        \co^{2m-2}(E[1];\Z/2) \cong \Z/2 \langle Sq^3 \iota_{2m-5}^2 \rangle \oplus \left(\co^*(B(m,1^3);\Z/2)/(b_{m-2},b_{m-1})\right)^{(2m-2)}
    \end{equation*}
    and the inclusion $h$ of the fiber in \eqref{eq:p1} induces the projection on the direct summand
    \begin{equation} \label{eq:h*}
        h^* \colon \co^{2m-2}(E[1];\Z/2) \longrightarrow \co^{2m-2}(K(\Z,2m-5); \Z/2) \cong \Z/2 \langle Sq^3 \iota^2_{2m-5} \rangle.
    \end{equation}
    Moreover, $\co^{2m-1}(E[1];\Z/2) \cong \Z/2 \langle Sq^4 \iota_{2m-5}^2 \rangle$ and $(Sq^3 \iota_{2m-5}^2 \oplus 0) x \neq 0 \in \co^{2m}(E[1];\Z/2)$, for any $x \in \co^2(E[1];\Z/2) \cong \co^{2}(B(m,1^3);\Z/2)$.
    \item Let $p=4$ and $m = 0 ~\mathrm{mod}~8$. Then, analogously to part \emph{(i)}, we have that the element \[\delta_2^2 Sq^2\iota^2_{2m-5} \neq 0 \in \co^{2m-2}(E[1];\Z/4)\] as well as its multiple by any element from $\co^2(E[1];\Z/4)$ is nonzero in $\co^{2m}(E[1];\Z/4)$. Moreover, we have $\co^{2m-1}(E[1];\Z/4) = \Z/2\langle \theta_2^2 Sq^4 \iota_{2m-5}^2\rangle$.
    \item Let $p=3$ and $m = 0 ~\mathrm{mod}~3$. Then $\co^{2m-1}(E[1];\Z/3) = 0$. Moreover, $b_{m-1}^{(3)} \neq 0 \in \co^{2m-2}(E[1];\Z/3)$, but $b_1^{(3)}b_{m-1}^{(3)} = 0 \in \co^{2m}(E[1];\Z/3)$. Here $b_i^{(3)}$ is the mod $3$ Chern class \[a_i = c_i(\gamma_m \times 1^r - 1 \times \gamma_1^r) \in \co^{2i}(B(m,1^3);\Z)\] (see \emph{Notation \ref{notate:1^k}}).
\end{enumerate}
\end{lem}

\begin{proof}
Let us split the proof with respect to $p$. In each case, $m$ will be assumed to be of the corresponding divisibility. In all of the cases, we will consider the fibration \eqref{eq:p1}. Using Corollary \ref{cor:cohomEMS}, we study the Leray--Serre spectral sequence $(E_*^{*,*},d_*)$ of \(p_1\) with mod \(p\) coefficients. Specifically, we will look at the $(2m-1)$-th diagonal of the $E_{\infty}$-page. 
\begin{enumerate}[\normalfont(i)]
    \item Let first $p=2$. Then the differential on the $(2m-4)$-th page satisfies 
\[
    d_{2m-4}(\iota_{2m-5}^2) = b_{m-2} \in E_{2m-4}^{2m-4,0} \cong \co^{2m-4}(B(m,1^3);\Z/2).
\]
Since $d_{2m-4}$ is a morphism of $\co^*(B(m,1^3);\Z/2)$-modules, we obtain
\begin{equation*}
    d_{2m-4}(\iota_{2m-5}^2 \otimes x) = b_{m-2}x \neq 0 \in E_{2m-4}^{2m,0} \cong \co^{2m}(B(m,1^3);\Z/2),
\end{equation*}
for any nonzero $x \in \co^{4}(B(m,1^3);\Z/2)$. Therefore, \(E_{\infty}^{4,2m-5} = E_{2m-3}^{4,2m-5} = 0\) and
\[
    E_{2m-3}^{*,0} \cong E_{2m-2}^{*,0} \cong \co^*(B(m,1^3);\Z/2)/(b_{m-2}).
\]
Additionally, by Kudo's transgression theorem, the differential on the $(2m-2)$-th page, satisfies
\begin{equation}\label{eq:diffSq^2}
    d_{2m-2}(Sq^2\iota_{2m-5}^2) = Sq^2b_{m-2} = \binom{m-4}{0}b_1b_{m-2} + \binom{m-3}{1} b_0b_{m-1} = b_1b_{m-2}+b_{m-1},
\end{equation}
which is non-zero in $E_{2m-2}^{*,0}$ because $b_1, \dots, b_m$ are algebraically independent. It also follows that \(E_{\infty}^{2,2m-3} = E_{2m-1}^{2,2m-3} = 0\) and
\[
    E_{2m-1}^{*,0} \cong \co^*(B(m,1^3);\Z/2)/(b_{m-2}, b_{m-1}).
\]
Now the differential on the $(2m-1)$-th page is trivial, because, similarly as before, we have $d_{2m-1}(Sq^3 \iota_{2m-5}^2) = Sq^3b_{m-2} = 0$. As for the differential on $(2m)$-th page, we use Wu's formula and the fact that \(m\) is divisible by \(4\) to compute
\[
    d_{2m}(Sq^4\iota_{2m-5}^2) = Sq^4b_{m-2} = \binom{m-5}{0} b_2b_{m-2} + \binom{m-4}{1} b_1b_{m-1} + \binom{m-3}{2} b_0b_m = b_2b_{m-2}.
\]
However, this is zero in $E_{2m}^{*,0}$, so $$E_{2m}^{0,2m-1} = E_{\infty}^{0,2m-1}= \Z/2 \langle Sq^4 \iota_{2m-4}^2\rangle.$$ Finally, since the only non-trivial term on the $(2m-1)$-th diagonal is $E_{\infty}^{0,2m-1}$, we glean the conclusion of the lemma in this case, see Figure \ref{fig:ss_p1_mod_2-3}(A) for an illustration.

\begin{figure}
\centering
\begin{subfigure}{.5\textwidth}
  \centering
  \includegraphics[width=.9\linewidth]{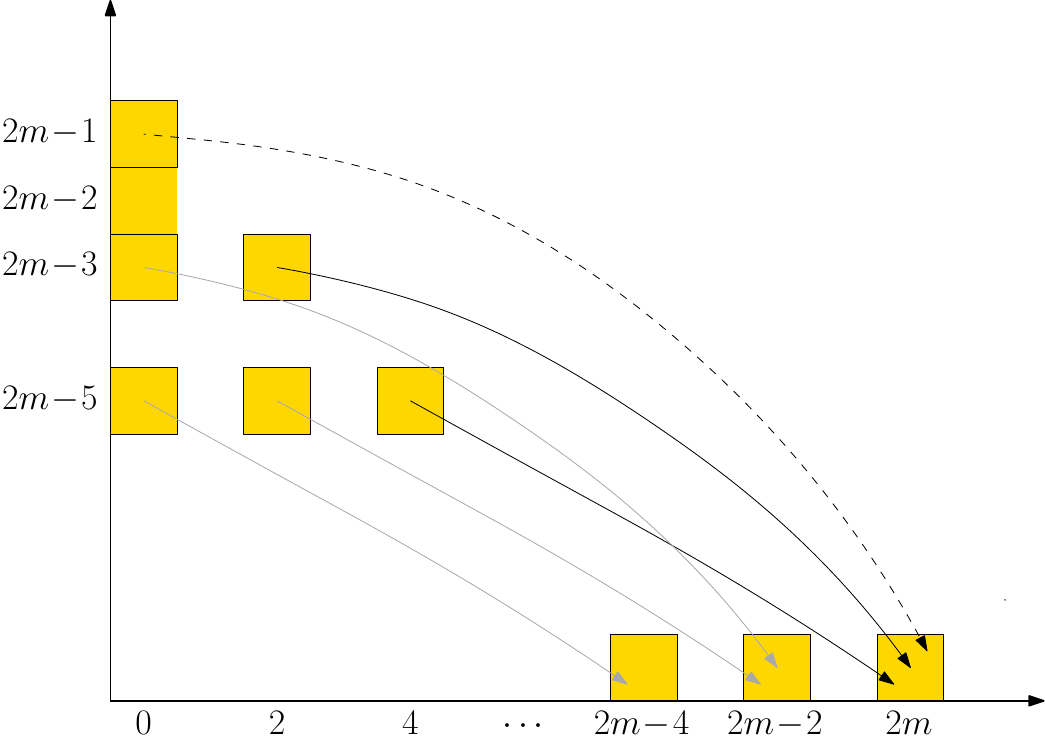}
  \caption{}
  \label{fig:sub1}
\end{subfigure}%
\begin{subfigure}{.5\textwidth}
  \centering
  \includegraphics[width=.9\linewidth]{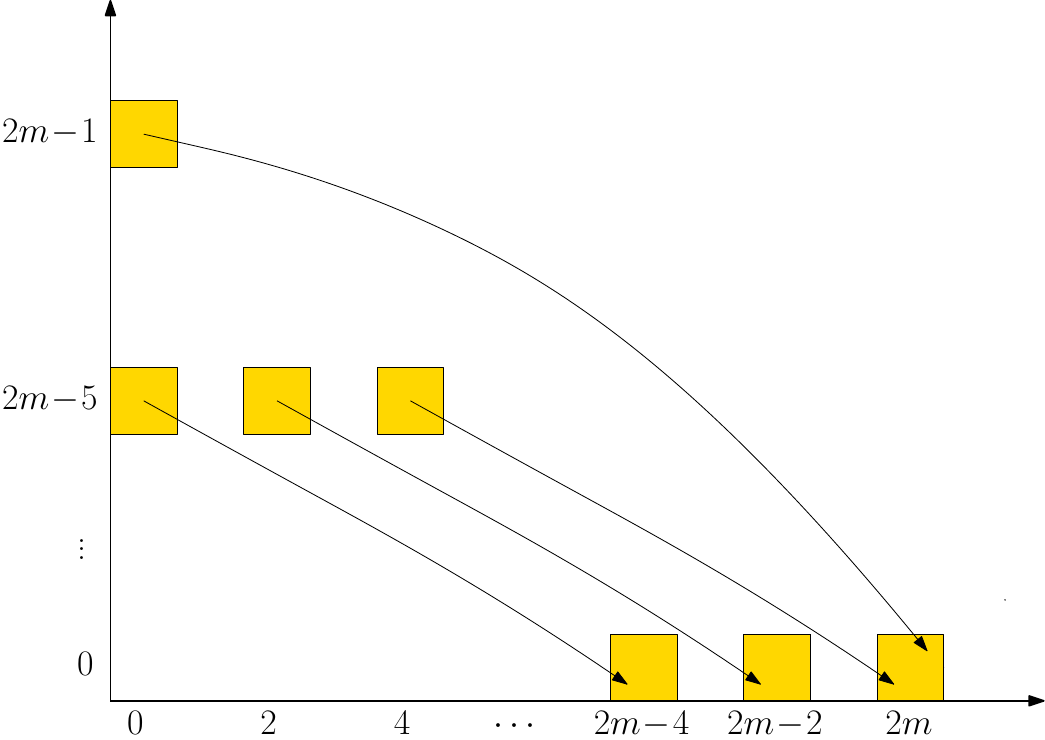}
  \caption{ }
  \label{fig:sub2}
\end{subfigure}
\caption{A portion of the Leray--Serre spectral sequence of homotopy fibration $p_1$ from \eqref{eq:p1} with (A) $\Z/2$, and (B) $\Z/3$ coefficients.}
\label{fig:ss_p1_mod_2-3}
\end{figure}

    \item  Observe that \(\theta \colon \co^*(-;\Z/2) \to \co^*(-;\Z/4)\) satisfies the following formula: 
\[
\text{For any } x \in \co^*(-;\Z/2) \text{ and } y \in \co^*(-;\Z), \text{ we have that }\theta(x \smile \rho_2y) = \theta(x) \smile \rho_4(y).
\]
Using this formula, the case that \(p = 4\) follows from similar arguments to the case that \(p=2\). We omit the details for brevity.

 \item Let $p=3$. Analogous to the above, we study the mod 3 spectral sequence of the same fibration, see Table \ref{table:modpcohomEMS}. Namely, the differential
\[
    d_{2m-4} \colon \co^{2m-5}(K(\Z,2m-5);\Z/3) = E_{2m-5}^{0, 2m-5} \longrightarrow E_{2m-4}^{2m-4, 0} = \co^{2m-4}(B(m,1^3);\Z/3)
\]
is defined by $d(\iota_{2m-5}^3) = b_{m-2}$. Therefore, $E^{2m-5,4}_{\infty} = E^{2m-5,4}_{2m-3} = 0$ and
\[
    E^{*,0}_{2m-3} = \dots = E^{*,0}_{2m} \cong \co^*(B(m,1^3);\Z/3)/(b_{m-2}).
\]
By \cite[Thm.~11.3~\&~Ex.~on~page~429]{borel1953groupes} and Kudo's transgression formula for Steenrod powers, we get
\[
    d_{2m}(P^1_3(\iota_{2m-5}^3)) = P^1_3(b_{m-2}) = b_1^2b_{m-2}-2b_2b_{m-2}-b_1b_{m-1}+mb_m \neq 0 \in E^{*,0}_{2m},
\]
because $b_1b_{m-1}+mb_m = b_1b_{m-1}\notin (b_{m-2})$ due to algebraic independence of $b_i$'s (and the fact that $m$ is divisible by 3).
Therefore, since all of the slots on the $(2m-1)$-th diagonal on the $E_{\infty}$-page are zero, we procure the claim in this case. See Figure \ref{fig:ss_p1_mod_2-3}(B) for an illustration.
\end{enumerate}\end{proof}

Using the above lemma, we may now prove the claimed result:

\begin{proof}[Proof of Lemma \ref{lem:2m-1-coh-of-E[2]}]
Let us look at the mod $p$ Leray--Serre spectral sequence $(E_*^{*,*},d_*)$ of the fibration
\begin{equation} \label{eq:p2}
    K(\Z, 2m-3) \overset{j}\longrightarrow E[2] \xrightarrow{~p_2~} E[1].
\end{equation}
Since $\co^{2m-2}(K(\Z, 2m-3);\Z/p) = 0$ (see Table \ref{table:modpcohomEMS}), we observe that the slot
\[
    E_*^{2m-1,0} \cong \co^{2m-1}(E[1];\Z/p)
\]
is never hit by a differential---see Figure \ref{fig:ss_p2-and-s1-_mod_2_even}(A) for illustration.

Next we consider the ideal in $\co^*(E[1];\Z/p)$ generated by the image of the differential on the $(2m-2)$-th page. Namely, we claim that
\[
    d_{2m-2} \colon \co^{2m-3}(K(\Z,2m-3);\Z/p) \cong E_{2m-2}^{0,2m-3} \longrightarrow E_{2m-2}^{2m-2, 0} \cong \co^{2m-2}(E[1];\Z/p)
\]
satisfies the following rule: 
\begin{equation}\label{eq:diff_iota_{2m-3}}
d_{2m-2}(\iota_{2m-3}) = \begin{cases}
    Sq^3\iota_{2m-5}^2, & \text{ for } p = 2, \\
    \delta_2^2 Sq^2 \iota_{2m-5}^2, & \text{ for } p =4, \\ b_{m-1}^{(3)}, & \text{ for } p =3.
\end{cases}
\end{equation}
Let us complete the proof of the lemma assuming \eqref{eq:diff_iota_{2m-3}} holds.
By Lemma \ref{lem:2m-1-coh-of-E[1]}(i)-(iii) these values are all nonzero. In the case that \(p=3\), by Lemma \ref{lem:2m-1-coh-of-E[1]}(iii), we obtain that \[E_\infty^{2,2m-3} \cong \Z/3 \langle b^{(3)}_{1}\otimes \iota_{2m-3}^3\rangle,\] which is the only non-trivial term on the $(2m-1)$-th diagonal. This completes the proof of the lemma for $p=3$. On the other hand, for $p=2,4$ we have that $d_{2m-2}(x \otimes \iota_{2m-3}) \neq 0$, for any $x \in \co^2(E[1];\Z/p)$, which implies that \(E_\infty^{2,2m-3} \cong 0\). Thus, for $p=2,4$, the only terms on the $(2m-1)$-th diagonal are $(2m-1, 0)$ and $(0,2m-1)$.

First, let $p=2$. By Kudo's transgression theorem, we have that the differential $d_{2m} \colon E_{2m}^{0,2m-1} \to E_{2m}^{2m, 0}$ maps
    \begin{equation} \label{eq:diff-Sq}
        d_{2m} \colon \co^{2m-1}(K(\Z,2m-3);\Z/2) \longrightarrow \co^{2m}(E[1];\Z/2),~Sq^2 \iota^2_{2m-3} \longmapsto Sq^2(Sq^3 \iota^2_{2m-5} \oplus 0).
    \end{equation}
The Adem relation \(Sq^2Sq^3 = Sq^5 + Sq^4Sq^1\) and Lemma \ref{lem:2m-1-coh-of-E[1]}(i) imply
\begin{equation*}
    h^*(Sq^2(Sq^3 \iota^2_{2m-5} \oplus 0)) = Sq^2h^*(Sq^3 \iota^2_{2m-5} \oplus 0) = Sq^5 \iota^2_{2m-5} \neq 0 \in \co^{2m-1}(K(\Z,2m-3);\Z/2),
\end{equation*}
so the differential \eqref{eq:diff-Sq} is non-zero, implying that $E_{\infty}^{2m-1,0} \cong \co^{2m-1}(E[1];\Z/2)$ is the only term on the $(2m-1)$-th diagonal. Thus, the lemma holds for $p=2$.

Similarly, for $p=4$ we have \(d_{2m}(\theta_2^2Sq^2 \iota_{2m-3}^2) = \theta_2^2 Sq^5 \iota_{2m-5}^2 \neq 0\), for \(p =4\), so the lemma holds in this case as well.

\begin{figure}
\centering
\begin{subfigure}{.5\textwidth}
  \centering
  \includegraphics[width=.9\linewidth]{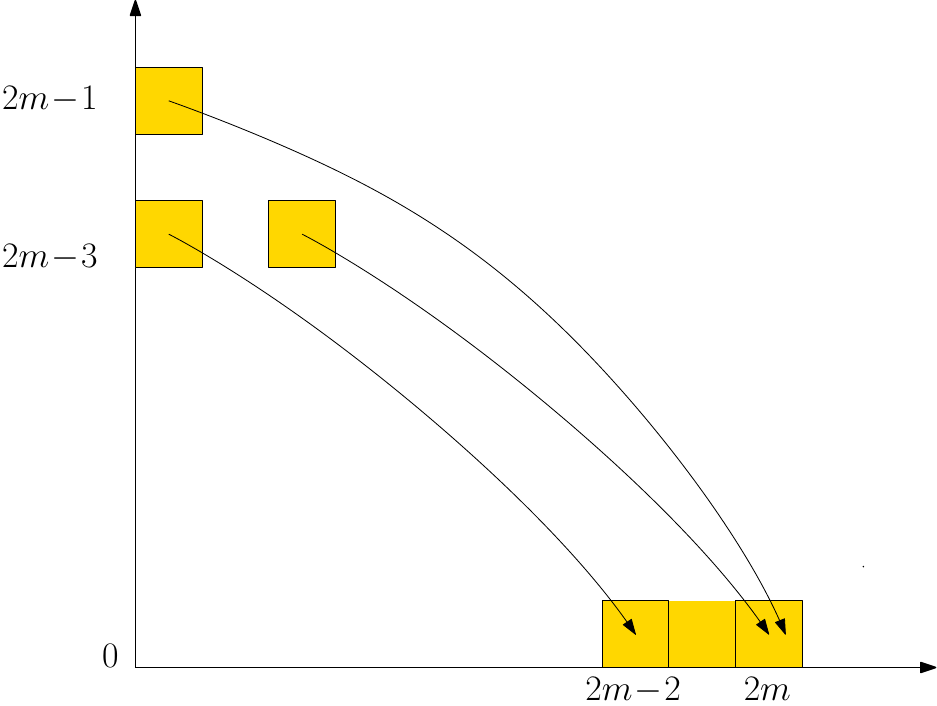}
  \caption{}
  \label{fig:sub1}
\end{subfigure}%
\begin{subfigure}{.5\textwidth}
  \centering
  \includegraphics[width=.9\linewidth]{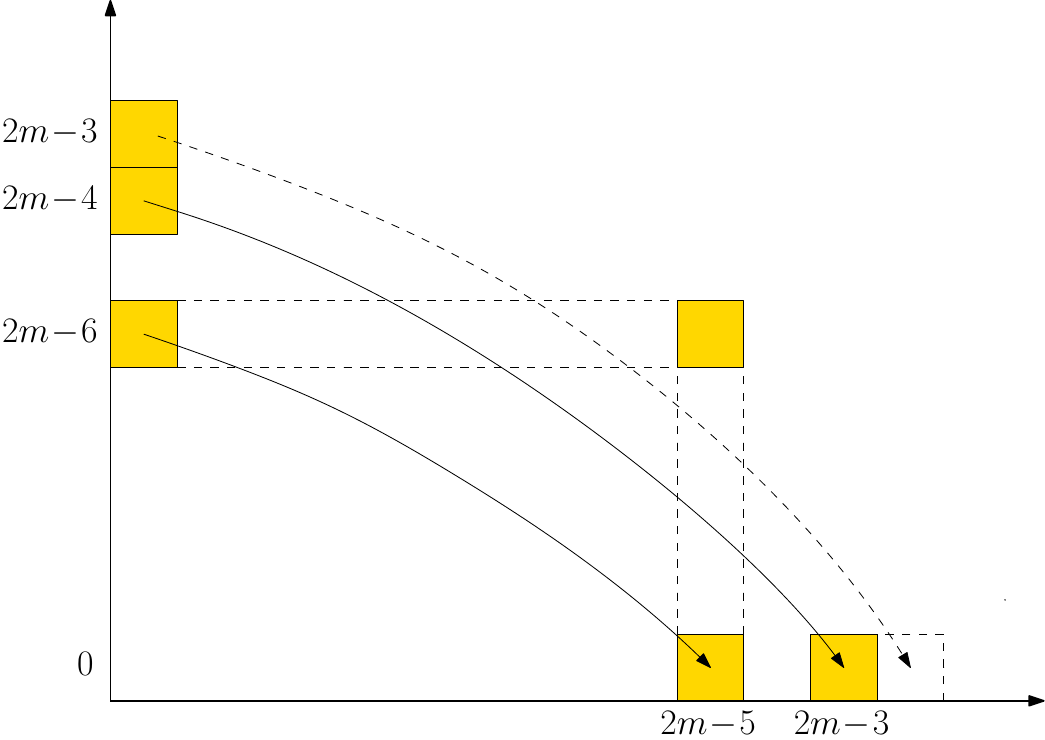}
  \caption{ }
  \label{fig:sub2}
\end{subfigure}
\caption{A portion of the mod 2 Leray--Serre spectral sequence of (A) $p_2$ from \eqref{eq:p2}, and (B) $s_1$ from \eqref{eq:F[1]}, where dashed arrows represent zero differentials.}
\label{fig:ss_p2-and-s1-_mod_2_even}
\end{figure}

\noindent

We now prove the claim (\ref{eq:diff_iota_{2m-3}}). Let \(p = 2\) and consider the mod 2 Leray--Serre spectral sequence \((F_*^{*,*},d_*)\) of the fibration 
\begin{equation}\label{eq:F[1]}
K(\Z,2m-6) \longrightarrow F[1] \overset{s_1}\longrightarrow W(m,3),
\end{equation}
see Figure \ref{fig:ss_p2-and-s1-_mod_2_even}(B) for an illustration. It is straightforward to see that 
\[
d_{2m-5} \colon~ \Z/2 \langle \iota_{2m-6}^2 \rangle \cong F_{2m-5}^{0,2m-6} \longrightarrow F_{2m-5}^{2m-5} \cong \Z/2 \langle f_{2m-5} \rangle,~\iota_{2m-6}^2 \longmapsto f_{2m-5}.
\]
Thus, by Kudo's transgression theorem and the action of the Steenrod squares detailed in Proposition \ref{prop:cohomology_rings}(v), we deduce that $d_{2m-3} \colon~ F_{2m-3}^{0,2m-4} \to F_{2m-3}^{2m-3,0}$ satisfies
\[
d_{2m-3}(Sq^2\iota_{2m-6}^2) = Sq^2 f_{2m-5} = f_{2m-3} \in F_{2m-3}^{2m-3,0} \cong \Z/2 \langle f_{2m-3} \rangle,
\]
thus \(F_\infty^{2m-3,0} = 0\). Passing to the next page, since $m>4$, the only non-zero term on the $(2m-3)$-rd diagonal is \(F_{2m-2}^{0, 2m-3} \cong F_{2}^{0,2m-3}\) However, all the differentials emanating from it land in trivial groups, so this slot remains until the $F_{\infty}$-page. Hence
\[
\co^{2m-3}(F[1];\Z/2) \cong \Z/2\langle Sq^3 \iota_{2m-6}^2\rangle.
\]
Lastly, it follows from the spectral sequence associated to 
\[
F[1] \longrightarrow B(m-3,1^3) \overset{q_1}\longrightarrow E[1],
\]
that the generator \(Sq^3 \iota_{2m-6}^2 \in \co^{2m-3}(F[1];\Z/2)\) transgresses to the nontrivial element \[Sq^3 \iota_{2m-5} \oplus 0 \in \co^{2m-2}(E[1];\Z/2).\] This is because $Sq^3 \iota_{2m-5} \oplus 0 \in \ker(q_1^*) \cap \co^{2m-2}(E[1];\Z/2)$ and \(\co^{2m-3}(B(m-3,1^3);\Z/2) = 0\).

Combining the above yields the claim when \(p = 2.\) Furthermore, the inclusion \(\theta_2^2 \colon \Z/2 \to \Z/4\) induces a map of spectral sequences by which the claim follows for \(p = 4\). Similar arguments yield the claim when \(p =3\), which we omit. This finishes the proof of the claim \eqref{eq:diff_iota_{2m-3}}.
\end{proof}

\paragraph{\textbf{Computing the indeterminacy of lifts.}}
Let \(\pi_{2m-2}\) be as in Proposition \ref{prop:piW(m,r)}. Further let the assumptions of Lemma \ref{lem:2m-1_obs_is_0} be satisfied. Then there is a lift of \((\xi,\ell_1,\ell_2,\ell_3)\) to \(E[3]\) if and only if \(g^*(k_3) = 0\) in \(\co^{2m-1}(X;\pi_{2m-2})/\mathrm{Indet}\) where \(g \colon X \to E[2]\) is any lift of \((\xi,\ell_1,\ell_2,\ell_3)\) to \(E[2]\) and \[\mathrm{Indet} \subseteq \co^{2m-1}(X;\pi_{2m-2})\] denotes the indeterminacy of such lifts. In this section, we calculate this indeterminacy. Note that \(\pi_{2m-2} = \Z/p(m) \oplus \Z/q(m)\) for some numbers \(p(m),q(m)\) that depend on the value of \(m\) mod \(24.\) We write \(\mathrm{Indet}_{\Z/k}\) for the \(\Z/k\) part of the indeterminacy subgroup. 

\begin{lem}
Let all spaces and maps be as in \emph{(\ref{eq:mp_diag_cpspan3_even})}. The indeterminacy \(\mathrm{Indet}\) of lifts of \((\xi,\ell_1,\ell_2,\ell_3)\) to \(E[2]\) satisfies the following. 
\begin{enumerate}[\normalfont(i)]
    \item Let \(m \not\equiv 0~\mathrm{mod}~3\) and \(m \equiv 2,6 ~\mathrm{mod}~8.\) Then \(\mathrm{Indet} =0.\)
    \item Let \(m\equiv 0 ~\mathrm{mod}~3\). Then \(\mathrm{Indet}_{\Z/3} =0.\)
    \item Let \(m\equiv 4 ~\mathrm{mod}~8\). Then \(\mathrm{Indet}_{\Z/2} =0.\)
    \item Let \(m \equiv 0 ~\mathrm{mod}~8\). Then \(\mathrm{Indet}_{\Z/4} = 0.\)
\end{enumerate}
\end{lem}

\begin{proof}
Point (i) follows immediately from Proposition \ref{prop:piW(m,r)} since \(\pi_{2m-2}=0\) in this case. Point (ii) follows immediately from Lemma \ref{lem:2m-1-coh-of-E[2]}. It remains to check points (iii) and (iv). To this end, we have the following homotopy commutative diagram 
\[\begin{tikzcd}
	{K(\Z,2m-3)} && {K(\Z,2m-3)} \\
	{K(\Z,2m-3)\times B(m-3,1^3)} & {K(\Z,2m-3)\times E[2]} & {E[2]} \\
	{B(m-1,1^3)} && {E[1]}
	\arrow[Rightarrow, no head, from=1-1, to=1-3]
	\arrow[from=1-1, to=2-1]
	\arrow[from=1-3, to=2-3]
	\arrow["{1\times q_2}", from=2-1, to=2-2]
	\arrow["\nu", curve={height=-22pt}, from=2-1, to=2-3]
	\arrow["\mathrm{proj}_1", from=2-1, to=3-1]
	\arrow["\mu", from=2-2, to=2-3]
	\arrow["{p_1}", from=2-3, to=3-3]
	\arrow["{\bar{s}}", curve={height=-12pt}, from=3-1, to=2-1]
	\arrow["{q_1}", from=3-1, to=3-3]
\end{tikzcd}\]
and an exact sequence 
\[
\dots \longrightarrow \ \mathrm{ker}(q_2^*)\cap\co^{2m-1}(E[2]) \overset{\nu^*}\longrightarrow \mathrm{ker}(\bar{s}^*)\cap\co^{2m-1}(K(\Z,2m-3)\times B(m-3,1^3)) \overset{\tau_1}\longrightarrow \co^{2m}(E[1]) \longrightarrow \dots,
\]
where all cohomology groups have \(\pi_{2m-2}\) coefficients. By abuse of notation (as previously explained), \(\mu^*(k_3) = 1 \otimes k_3 + \nu^*(k_3)\), where 
\[
\nu^*(k_3) \in \mathrm{ker}(\bar{s}) \cap \mathrm{ker}(\tau_1) \cap \co^{2m-1}(K(\Z,2m-3)\times B(m-3,1^3)).
\]
\noindent{(iii)} Let \(m \equiv 4~ \mathrm{mod}~8\). By direct calculation, one sees that \(\mathrm{ker}(\bar{s})\cap \co^{2m-1}(K(\Z,2m-3)\times B(m-3,1^3;\Z/2)\) is generated by elements of the form
\begin{itemize}
    \item \(\iota_{2m-3}^2 \otimes u\), for all \(u\in \co^{2}(B(m-3,1^3);\Z/2)\); and 
    \item \(Sq^2 \iota_{2m-3}^2 \otimes 1.\)
\end{itemize}
To calculate the action of the relative transgression on these elements, consider the mod \(2\) Leray--Serre spectral sequence \((E_*^{*,*},d_*)\) associated to the homotopy fibration
\[
    K(\Z,2m-3) \longrightarrow E[2] \overset{p_2}\longrightarrow E[1].
\]
Then by \eqref{eq:diff_iota_{2m-3}} we have
\[
d_{2m-2}^{(0,2m-3), (2m-2,0)}(\iota_{2m-3}^2) =Sq^3\iota_{2m-5}^2.
\]
Using the fact that the differentials are derivations, we see that 
\[\mathrm{ker}(\tau_1) \cap \co^{2m}(K(\Z,2m-3)\times B(m-3,1^3))\] contains no linear combination of elements \(\iota_{2m-3}^2 \otimes u\), for \(u \in \co^{2}(B(m-3,1^3);\Z/2)\). 
Moreover, by Kudo's transgression theorem and Wu's formula, we compute 
\[d_{2m}^{(0,2m-1),(2m,0)}(Sq^2 \iota_{2m-3}) = Sq^2Sq^3\iota_{2m-5}^2 = Sq^5\iota_{2m-5}^2 \neq 0 \in E_{2m}^{2m,0}.\]
Hence, \(Sq^2 \iota_{2m-3}^2 \otimes 1\) is not an element of \(\mathrm{ker}(\tau_1) \cap \co^{2m}(K(\Z,2m-3)\times B(m-3,1^3))\). Subsequently, \(\mu^*(k_3) = 1 \otimes k_3\), from which \(\mathrm{Indet}_{\Z/2} = 0\). \\

\noindent{(iv)} Let \(m \equiv 0~\mathrm{mod}~8.\) Similar arguments yield that \(\mathrm{Indet}_{\Z/4} = 0\).

\noindent{(ii)} Finally, let \(m \equiv 0~\mathrm{mod}~3\). Straightforward calculations similar to the above yield that \(\mathrm{Indet}_{\Z/3} = 0.\) We omit these for readability. 
\end{proof}

\subsection{The Generalized Quaternary Obstruction}\label{subsec:quaternaryobs_cpspan3_even}
Finally, the proof of Theorem \ref{thm:cpspan3_even} follows immediately from the proceeding theorem.

\begin{thm}\label{thm:quaternaryobs_even}
Let \(m > 5\) be an even integer and let all spaces and maps be as in \emph{(\ref{eq:mp_diag_cpspan3_even})}. Suppose that the conditions of \emph{Lemma \ref{lem:2m-1_obs_is_0}} are satisfied. Furthermore, if \(\co^{2m}(X;\Z)\) has no $n$-torsion, for $n = 12/\lvert\pi_{2m-2}\rvert \in \Z$ \emph{(see Table \ref{table:pi2m-2W(m,3)})},
    then the generalized quaternary obstruction \[\lobs^{2m}((\xi,\ell_1,\ell_2,\ell_3),q_{m,3},k_4) \subseteq \co^{2m}(X;\Z)\]
    is the singleton set \(\{c_m(\xi- \ell_1 \oplus \ell_2 \oplus \ell_3)\}.\)
\end{thm}

The rest of the section is dedicated to proving Theorem \ref{thm:quaternaryobs_even}.

\vspace{2mm}

\paragraph{\textbf{Computing the indeterminacy of lifts.}}\label{par:indet_cpspan3_even}
Let the assumptions of Theorem \ref{thm:quaternaryobs_even} be satisfied. 
By Lemma \ref{lem:2m-1_obs_is_0}, we know that there is a lift of \((\xi,\ell_1,\ell_2,\ell_3)\) along \(q_{m,3}\) if and only if \(g^*(k_4) = 0\) in \(\co^{2m}(X;\Z)/\mathrm{Indet}\), where \(g \colon X \to E[3]\) is any lift of \((\xi,\ell_1,\ell_2,\ell_3)\) to \(E[3]\) and \(\mathrm{Indet}\) denotes the indeterminacy of such lifts. We first demonstrate that \(\mathrm{Indet} = 0.\)

Write \(K \coloneqq K(\pi_{2m-2}, 2m-2)\), where the group \(\pi_{2m-2} \coloneqq \pi_{2m-2}(W(m,3))\) satisfies Proposition \ref{prop:piW(m,r)}. Then we have the following homotopy commutative diagram

\[\begin{tikzcd}
	K && K \\
	{K\times B(m-3,1^3)} & {K\times E[3]} & {E[3]} \\
	{B(m-3,1^3)} && {E[2]}
	\arrow[Rightarrow, no head, from=1-1, to=1-3]
	\arrow[from=1-1, to=2-1]
	\arrow[from=1-3, to=2-3]
	\arrow["{1\times q_3}", from=2-1, to=2-2]
	\arrow["\nu", curve={height=-18pt}, from=2-1, to=2-3]
	\arrow["\pi", from=2-1, to=3-1]
	\arrow["\mu", from=2-2, to=2-3]
	\arrow["{p_3}", from=2-3, to=3-3]
	\arrow["{\bar{s}}", curve={height=-12pt}, from=3-1, to=2-1]
	\arrow["{q_2}", from=3-1, to=3-3]
\end{tikzcd}\]
and an exact sequence
\[
\dots \longrightarrow \mathrm{ker}(q_3^*)\cap\co^{2m}(E[3])\overset{\nu^*}\longrightarrow \mathrm{ker}(\bar{s}^*)\cap\co^{2m}(K\times B(m-3,1^3))\overset{\tau_1}\longrightarrow \co^{2m+1}(E[2])\longrightarrow \dots,
\]
where all cohomology groups have integer coefficients. By standard arguments, \(\mu^*(k_4) = 1 \otimes k_4 + \nu^*(k_4)\) (by abuse of notation), where
\[
\nu^*(k_4) \in \mathrm{ker}(\bar{s}) \cap \mathrm{ker}(\tau_1) \cap \co^{2m}(K\times E[3]). 
\]
But it follows from Proposition \ref{prop:cohomK(Z/p,q)} that 
\(\co^{2m-2}(K;\Z) = \co^{2m}(K;\Z) = 0\), whence 
\[
    \mathrm{ker}(\bar{s}) \cap \co^{2m}(K\times E[3]) \cong 0.
\]
Subsequently, \(\nu^*(k_4) = 0\)
and 
\begin{equation}\label{eq:mu}
\mu^*(k_4) = 1 \otimes k_4.
\end{equation}

Now let \(g_1,g_2 \colon X \to E[3]\) be two lifts of \((\xi,\ell_1,\ell_2,\ell_3) \colon X \to B(m,1^3)\) in (\ref{eq:mp_diag_cpspan3_even}). Then there exists a cohomology class \(\alpha \colon X\to K\) such that the composition 
\[
X \overset{\Delta}\longrightarrow X \times X \overset{\alpha \times g_1}\longrightarrow K \times E[3]\overset{\mu}\longrightarrow E[3]
\]
is homotopic to \(g_2\). From (\ref{eq:mu}) we see that \(g_2^*(k_4) - g_1^*(k_4) = 0\), whence \(\mathrm{Indet} = 0\), as required.\\

Accordingly, there is a unique quaternary obstruction to lifting \((\xi,\ell_1,\ell_2,\ell_3)\) along \(q_{m,3}\). It remains to identify this unique obstruction. To this end, recall that $k_4 \in \co^{2m}(E[3];\Z)$ is the characteristic element in the fibration
\begin{equation*}
    F[3] \longrightarrow B(m-3, 1^3) \xrightarrow{~q_3~} E[3].
\end{equation*}
Further recall that the fiber $F[3]$ is $(2m-2)$-connected and $\pi_{2m-1}(F[3]) \cong \pi_{2m-1}(W(m,3)) \cong \Z$. Thus, by universal coefficient theorem, we have 
\[
    \co^{2m-1}(F[3];\Z) \cong \Z.
\]
For each $i=1,2,3,$ let us denote by $s_i \colon F[i] \longrightarrow F[i-1]$ the induced map between homotopy fibers of maps $p_i \colon E[i] \longrightarrow E[i-1]$, where $F[0] = W(m,3)$, $E[0] = B(m,1^3)$ and $q_0 = q_{m,3}$, c.f., Lemma \ref{lem:thomas_diagram}. Set \(s \coloneqq s_1 \circ s_2 \circ s_3\) and \(p \coloneqq p_1 \circ p_2 \circ p_3\). 

We also recall
\[
    \co^*(W(m,3);\Z) \cong \Lambda_\Z[e_{2m-5}, e_{2m-3}, e_{2m-1}],
\]
with $\lvert e_i \rvert = i$, see Proposition \ref{prop:cohomology_rings}.

\begin{lem} \label{lem:(f1f2f3)^*-is-mono}
    Let $e \coloneqq e_{2m-1}$ be the generator of $\co^{2m-1}(W(m,3))$. Then, the map $s \colon F[3] \longrightarrow W(m,3)$ induces an inclusion
    \begin{equation*}
        s^* \colon \co^{2m-1}(W(m,3);\Z) \cong \Z \left\langle e \right\rangle \longrightarrow \Z \left\langle e/n\right\rangle \cong \co^{2m-1}(F[3];\Z),~ e \longmapsto e,
    \end{equation*}
    for an integer $n$ which satisfies $\lvert n \rvert = 12/ \lvert \pi_{2m-2} \rvert$.
\end{lem}
\begin{proof}
All cohomology groups throughout the proof have integer coefficients. By Lemma \ref{lem:thomas_diagram}, we have the following homotopy commutative Thomas diagram associated to (\(\ddagger\)).   
    \begin{equation*} \label{eq:thomas-F[1]-k=3}
\begin{tikzcd}
	{K(\Z,2m-6)} & {*} & {K(\Z,2m-5)} \\
	{F[1]} & {B(m -3,1^3)} & {E[1]} \\
	{W(m,3)} & {B(m-3,1^3)} & {B(m,1^3)}
	\arrow[from=1-1, to=1-2]
	\arrow[from=1-1, to=2-1]
	\arrow["{\mathrm{ev}_1}", from=1-2, to=1-3]
	\arrow[from=1-2, to=2-2]
	\arrow[from=1-3, to=2-3]
	\arrow[from=2-1, to=2-2]
	\arrow["{s_1}"', from=2-1, to=3-1]
	\arrow["{q_1}", from=2-2, to=2-3]
	\arrow["{\mathrm{id}~}"', Rightarrow, no head, from=2-2, to=3-2]
	\arrow["{(\ddagger)}"{description}, draw=none, from=2-2, to=3-3]
	\arrow["{p_1}", from=2-3, to=3-3]
	\arrow[from=3-1, to=3-2]
	\arrow["{q_{m,3}}", from=3-2, to=3-3]
\end{tikzcd}
\end{equation*}
    Let us consider the Leray--Serre spectral sequence for the homotopy fibration $s_1$ converging to $\co^*(F[1])$. Since $m > 4$ even, there are no nontrivial terms on the $(2m-4)$-th, $(2m-2)$-th and $2m$-th diagonals. Therefore, $\co^{2m}(F[1]) = 0$, and the only two non-trivial terms on the $(2m-1)$-th diagonal fit into the short exact sequence
    \begin{equation*}
        0 \longrightarrow \Z\langle e \rangle \xrightarrow{~s_1^*~} \co^{2m-1}(F[1]) \longrightarrow \Z/6 \longrightarrow 0.
    \end{equation*}
    For a prime number $p$, we have that $\Z\langle e \rangle \oplus \Z/p$ and
    \begin{equation*}
        0 \longrightarrow \Z \langle e \rangle \longrightarrow \Z \langle e/p \rangle \xrightarrow{e/p \mapsto i} \Z/p \longrightarrow 0, \hspace{3mm} \text{for}~ i=1, \dots, p-1,
    \end{equation*}
    are all possible extensions of $\Z/p$ by $\Z$. Therefore, since Ext-functor is additive on the first entry, $\text{Ext}^1_{\Z}(\Z/6, \Z) \cong \text{Ext}^1_{\Z}(\Z/2, \Z)\! \times \! \text{Ext}^1_{\Z}(\Z/3, \Z)$, we obtain
    \begin{equation} \label{eq:2m-1_co_F[1]}
        \co^{2m-1}(F[1]) \in \left\{\Z\left\langle  e/6 \right\rangle, \Z\left\langle  e/3 \right\rangle \oplus \Z/2, \Z\left\langle e/2 \right\rangle \oplus \Z/3, \Z\langle e \rangle \oplus \Z/6 \right\}.
    \end{equation}
Moreover, in each case, $s_1^*$ is the standard inclusion from $\Z\langle e \rangle$.
    
Again by Lemma \ref{lem:thomas_diagram}, we have the Thomas diagram associated to (\(\ddagger\ddagger\)).
    \begin{equation*} \label{eq:thomas-F[2]-k=3}
\begin{tikzcd}
	{K(\Z,2m-4)} & {*} & {K(\Z,2m-3)} \\
	{F[2]} & {B(m -3,1^3)} & {E[2]} \\
	{F[1]} & {B(m-3,1^3)} & {E[1]}
	\arrow[from=1-1, to=1-2]
	\arrow[from=1-1, to=2-1]
	\arrow["{\mathrm{ev}_1}", from=1-2, to=1-3]
	\arrow[from=1-2, to=2-2]
	\arrow[from=1-3, to=2-3]
	\arrow[from=2-1, to=2-2]
	\arrow["{s_2}"', from=2-1, to=3-1]
	\arrow["{q_2}", from=2-2, to=2-3]
	\arrow["{~\mathrm{id}}"', Rightarrow, no head, from=2-2, to=3-2]
	\arrow["{(\ddagger\ddagger)}"{description}, draw=none, from=2-2, to=3-3]
	\arrow["{p_2}", from=2-3, to=3-3]
	\arrow[from=3-1, to=3-2]
	\arrow["{q_1}"', from=3-2, to=3-3]
\end{tikzcd}\
\end{equation*}
We may now consider the Leray--Serre spectral sequence of $s_2$. Here also, the $(2m-2)$-th and $2m$-th diagonals are trivial, so $\co^{2m}(F[2]) = 0$. Additionally, the only two non-zero terms on the $(2m-1)$-th diagonal fit into a short exact sequence 
    \begin{equation*}
        0 \longrightarrow \co^{2m-1}(F[1]) \xrightarrow{~s_2^*~} \co^{2m-1}(F[2]) \longrightarrow \Z/2 \longrightarrow 0.
    \end{equation*}
    Checking all $\co^{2m-1}(F[1])$ as enumerated in \eqref{eq:2m-1_co_F[1]} and using the fact that functor $\text{Ext}^1_\Z(\Z/2, \cdot)$ respects products, we have the following possibilities:
    \begin{itemize}
        \item if $\co^{2m-1}(F[1])=\Z\left\langle e/6 \right\rangle$, then 
        \[
            \co^{2m-1}(F[2])\in\left\{\Z\left\langle e/6 \right\rangle \oplus \Z/2, \Z\left\langle e/12 \right\rangle\right\};
        \]
        \item if $\co^{2m-1}(F[1])= \Z\left\langle e/3 \right\rangle \oplus \Z/2$, then 
        \[
            \co^{2m-1}(F[2])\in \left\{\Z\left\langle e/3 \right\rangle \oplus \Z/2^{\oplus 2}, \Z\left\langle e/6 \right\rangle \oplus \Z/2, \Z\left\langle e/3 \right\rangle \oplus \Z/4\right\};
        \]
        \item if $\co^{2m-1}(F[1])=\Z\left\langle e/2 \right\rangle \oplus \Z/3$, then 
        \[
            \co^{2m-1}(F[2])\in  \left\{\Z\left\langle e/2 \right\rangle \oplus \Z/6, \Z\left\langle e/4 \right\rangle \oplus \Z/3\right\};
        \]
        \item if $\co^{2m-1}(F[1])=\Z\langle e \rangle \oplus \Z/6$, then 
        \[
            \co^{2m-1}(F[2])\in \left\{\Z\left\langle e/2 \right\rangle \oplus \Z/6, \Z\langle e \rangle \oplus \Z/12\right\}.
        \]
    \end{itemize}
    We may abstract the following general form for $\co^{2m-1}(F[2])$. Namely, we have
    \begin{equation} \label{eq:2m-1_co_F[2]}
        \co^{2m-1}(F[2]) \cong \Z \left\langle  e/n \right\rangle \oplus \pi,
    \end{equation}
    where $\pi \in \{0, \Z/2, \Z/2^{\oplus 2}, \Z/3, \Z/4, \Z/6, \Z/12\}$ and $n \coloneqq 12/\lvert\pi\rvert$. We note that the free part of $\co^{2m-1}(F[1])$ includes in the standard way via $s_2^*$ into $\Z \left\langle e/n \right\rangle$.

Next consider the Thomas diagram associated to (\(\ddagger \ddagger \ddagger\)).
    \begin{equation*} \label{eq:thomas-F[3]-k=3}
\begin{tikzcd}
	{K(\pi_{2m-2},2m-3)} & {*} & {K(\pi_{2m-2},2m-2)} \\
	{F[3]} & {B(m -3,1^3)} & {E[3]} \\
	{F[2]} & {B(m-3,1^3)} & {E[2]}
	\arrow[from=1-1, to=1-2]
	\arrow[from=1-1, to=2-1]
	\arrow["{\mathrm{ev}_1}", from=1-2, to=1-3]
	\arrow[from=1-2, to=2-2]
	\arrow[from=1-3, to=2-3]
	\arrow[from=2-1, to=2-2]
	\arrow["{s_3}"', from=2-1, to=3-1]
	\arrow["{q_3}", from=2-2, to=2-3]
	\arrow["{~\mathrm{id}}"', Rightarrow, no head, from=2-2, to=3-2]
	\arrow["{(\ddagger\ddagger\ddagger)}"{description}, draw=none, from=2-2, to=3-3]
	\arrow["{p_3}", from=2-3, to=3-3]
	\arrow[from=3-1, to=3-2]
	\arrow["{q_2}"', from=3-2, to=3-3]
\end{tikzcd}
\end{equation*}
    We may consider the Leray--Serre spectral sequence $(E_*^{*,*},d_*)$ of $s_3$. Recall that, due to Table \ref{table:cohomK(G,q)}, we have 
    \begin{equation*}
        E_{2m-1}^{0,2m-2} \cong \co^{2m-2}(K(\pi_{2m-2}, 2m-3)) \cong \text{Ext}^1_\Z(\pi_{2m-2},\Z) \cong \pi_{2m-2},
    \end{equation*}
    since $\pi_{2m-2}$ is a finite torsion group described in Proposition \ref{prop:piW(m,r)}. The only differentials mapping into the $(2m-1)$-th diagonal are
    \begin{equation*}
        d_{2m-1} \colon~ \pi_{2m-2} \cong E_{2m-1}^{0,2m-2} \longrightarrow E_{2m-1}^{0,2m-1} \cong \co^{2m-1}(F[2])
    \end{equation*}
    on $(2m-1)$-th page, while there are no non-trivial differentials mapping out from the $(2m-1)$-th diagonal. Recall that $F[3]$ is $(2m-2)$-connected and 
    \[
        \pi_{2m-1}(F[3]) \cong \pi_{2m-1}(W(m,3)) \cong \Z,
    \]
    so the universal coefficients theorem gives that $\co^{2m-1}(F[3]) \cong \Z$.
    Therefore, $d_{2m-1}$ is a monomorphism and since there is no extension problem on the $(2m-1)$-th diagonal, we conclude that
    \[
        s_3^* \colon~ \co^{2m-1}(F[2]) \longrightarrow \co^{2m-1}(F[3]) \cong \co^{2m-1}(F[2])/\pi_{2m-2}
    \]
    is the quotient map.
    However, $\Ext_{\Z}^1(\Z,\cdot)=0$ implies that $\co^{2m-1}(F[2]) \cong \Z \oplus \pi_{2m-2}$. Finally, it follows from (\ref{eq:2m-1_co_F[2]}) that
    \begin{equation*}
        \co^{2m-1}(F[3]) \cong \Z \left\langle e/n \right\rangle,
    \end{equation*}
    where again $n = 12/\lvert\pi_{2m-2}\rvert$. This finishes the proof.
\end{proof}

Next consider the map between homotopy fibrations
\begin{equation*}
    \begin{tikzcd}
        F[3] \arrow[r] \arrow[d, "s"]\ & B(m-3, 1^3) \arrow[d, equal] \arrow[r, "q_3"] & E[3] \arrow[d, "p"]\\
        W(m,3) \arrow[r] & B(m-3, 1^3) \arrow[r, "q_{m,3}"] & B(m,1^3)
    \end{tikzcd}
\end{equation*}
and the induced map between Leray--Serre spectral sequences. Expressly, by Lemma~\ref{lem:diffsofq1} (see also the notation therein), we have the following commutative diagram of the corresponding transgressions,
\begin{equation*}
        \begin{tikzcd}
            \co^{2m-1}(W(m,3)) \arrow[r, "s^*"] \arrow[d, "\tau_{q_{m,3}}"] & \co^{2m-1}(F[3]) \arrow[d, "\tau_{q_3}"] & e_{2m-1} \arrow[r, mapsto] \arrow[d, mapsto] & e_{2m-1} = n \cdot \frac{1}{n} e_{2m-1} \arrow[d, mapsto]\\
            \co^{2m}(B(m,1^3)) \arrow[r, "p^*"] & \co^{2m}(E[3]) & a_{m} \arrow[r, mapsto] & p^*(a_m)=nk_3.
        \end{tikzcd}
    \end{equation*}
    Therefore, every homotopy lift
    \begin{equation*}
        \begin{tikzcd}
            {} & & E[3] \arrow[d, "p"]\\
            X \arrow[urr, dashed, "g"] \arrow[rr, "f"] & & B(m,1^3),
        \end{tikzcd}
    \end{equation*}
     of the classifying map $f = (\xi, \ell_1, \ell_2, \ell_3)$ over $p$ satisfies 
    \[
        ng^*(k_3) = (p \circ g)^*(a_m)= f^*(a_m) = c_m(\xi-\ell_1\oplus\ell_2\oplus\ell_3) \in \co^{2m}(X).
    \]
    This completes the proof of Theorem \ref{thm:quaternaryobs_even}.

\section{Three complex line bundles: $m$ odd}\label{sec:thmpspan3_odd}

In this section we prove Theorem \ref{thm:cpspan3_odd}.\\

Let \(X\) be a \(2m\)-dimensional CW complex with \(m \geq 5\) odd. Fix complex vector bundles \(\xi \colon X\to BU(m)\) and \(\ell_1,\ell_2,\ell_3 \colon X \to BU(1)\). We wish to lift the classifying map \((\xi,\ell_1,\ell_2,\ell_3) \colon X \to B(m,1^3)\) of (\ref{eq:fibq_1}) along the fibration \(W(m,3)\longrightarrow B(m-3,1^3)\overset{q_{m,3}}\longrightarrow B(m,1^3)\) defined in (\ref{eq:xilinesclassifyingmap}). We use Proposition \ref{prop:piW(m,r)} to construct the following Moore--Postnikov factorization of \(q_{m,3}\).
\begin{equation}\label{eq:mp_diag_cpspan3_odd}
\begin{tikzcd}[ampersand replacement=\&]
	\&\& {E[5]} \\
	\&\& {E[4]} \& {K(\pi_{2m-1},2m)} \\
	\&\& {E[3]} \& {K(\pi_{2m-2},2m-1)} \\
	\&\& {E[2]} \& {K(\Z \oplus \Z/2,2m-2)} \\
	\&\& {E[1]} \& {K(\Z/2,2m-3)} \\
	{W(m,3)} \& {B(m-3,1^3)} \& {B(m,1^3)} \& {K(\Z,2m-4)} \\
	\\
	\&\& X
	\arrow["{p_5}", from=1-3, to=2-3]
	\arrow["{k_5}", from=2-3, to=2-4]
	\arrow["{p_4}", from=2-3, to=3-3]
	\arrow["{k_4}", from=3-3, to=3-4]
	\arrow["{p_3}", from=3-3, to=4-3]
	\arrow["{k_3}", from=4-3, to=4-4]
	\arrow["{p_2}", from=4-3, to=5-3]
	\arrow["{k_2}", from=5-3, to=5-4]
	\arrow["{p_1}", from=5-3, to=6-3]
	\arrow[from=6-1, to=6-2]
	\arrow["{q_4}"{description}, from=6-2, to=2-3]
	\arrow["{q_3}"{description}, from=6-2, to=3-3]
	\arrow["{q_2}"{description}, from=6-2, to=4-3]
	\arrow["{q_1}"{description}, from=6-2, to=5-3]
	\arrow["{q_{m,3}}", from=6-2, to=6-3]
	\arrow["{k_1}", from=6-3, to=6-4]
	\arrow["{(\xi,\ell_1,\ell_2,\ell_3)}"', from=8-3, to=6-3]
\end{tikzcd}
\end{equation}

Here each \(k_i\) is the characteristic class in the fibration \(q_{i-1}\) (with \(q_0 = q_{m,3}\)); and each \(E[i]\) is the homotopy fiber of \(k_i\). Lastly, recall that \(\pi_{2m-2} = \pi_{2m-2}(W(m,3))\) satisfies Table \ref{table:pi2m-2W(m,3)} and 
\begin{equation*}
    \pi_{2m-1} = \pi_{2m-1}(W(m,3)) \cong \begin{cases}
        \Z \oplus \Z/2 & m=5,\\
        \Z & \text{otherwise.}
    \end{cases}
\end{equation*}

\begin{rem}
In spectral sequence calculations throughout this section, we will use that \(m >5.\) Minor modifications to the arguments yield the same statements when \(m = 5\), but we omit these details for readability.
\end{rem}

\subsection{The Primary Obstruction}\label{subsec:primaryobs_cpspan3_odd}

By Proposition \ref{prop:primaryobs}, we obtain the following result. 
\begin{lem}\label{lem:primaryob_cpspan3_odd}
Let all spaces and maps be as in \emph{(\ref{eq:mp_diag_cpspan3_odd})}. Then the primary obstruction 
\[
\lobs^{2m-4}((\xi,\ell_1,\ell_2,\ell_3),q_{m,3},k_1) \subseteq \co^{2m-4}(X;\Z)
\]
to lifting \((\xi,\ell_1,\ell_2,\ell_3)\) along \(q_{m,3}\) is the singleton set \(\{c_{m-2}(\xi - \ell_1 \oplus \ell_2 \oplus \ell_3)\}.\)
\end{lem}

\subsection{The Generalized Secondary Obstruction}\label{subsec:secondaryobs_cpspan3_odd}

By similar arguments to the proof of Lemma \ref{lem:secondaryob_cpspan2_even}, one demonstrates the following.

\begin{lem}\label{lem:secondaryobs_cpspan3_odd}
Let all spaces and maps be as in \emph{(\ref{eq:mp_diag_cpspan3_odd})}. Suppose that \(\co^{2m-3}(X;\Z/2) = Sq^2 \rho_2 \co^{2m-5}(X;\Z)\). Then zero is an element of the generalized secondary obstruction 
\[
\lobs^{2m-3}((\xi,\ell_1,\ell_2,\ell_3),q_{m,3},k_2) \subseteq \co^{2m-3}(X;\Z/2)
\]
to lifting \((\xi,\ell_1,\ell_2,\ell_3)\) along \(q_{m,3}\).
\end{lem}

\subsection{The Generalized Tertiary Obstruction}\label{subsec:tertiaryobs_cpspan3_odd}

 Appropriate modifications to the proof of Theorem \ref{thm:tertiaryob_cpspan2_even} yield the following result. 

\begin{lem}\label{lem:tertiaryobs_cpspan3_odd}
Let all spaces and maps be as in \emph{(\ref{eq:mp_diag_cpspan3_odd})}. Suppose that the classifying map \[(\xi,\ell_1,\ell_2,\ell_3) \colon X \longrightarrow B(m,1^3)\] satisfies \(c_{m-2}(\xi - \ell_1\oplus \ell_2\oplus\ell_3)=0\). Further suppose that  
\begin{itemize}
    \item \(\co^{2m-3}(X;\Z/2) = Sq^2 \rho_2 \co^{2m-5}(X;\Z)\); 
    \item \(\co^{2m-2}(X;\Z/2) = Sq^2\co^{2m-2}(X;\Z/2)\), and
    \item \(\co^{2m-2}(X;\Z)\) has no \(2\)-torsion. 
\end{itemize}
Then the generalized tertiary obstruction 
\[
\lobs^{2m-2}((\xi,\ell_1,\ell_2,\ell_3), q_{m,3},k_3) \subseteq \co^{2m-2}(X;\Z\oplus \Z/2)\]
to lifting \((\xi,\ell_1,\ell_2,\ell_3)\) along \(q_{m,3}\) vanishes if and only if \(c_{m-1}(\xi - \ell_1 \oplus \ell_2 \oplus \ell_3) = 0\). 
\end{lem}

\subsection{The Generalized Quaternary Obstruction}\label{subsec:quaternaryobs_cpspan3_odd}

\begin{thm}\label{thm:quaternaryobs_cpspan3_odd}
Let \(m\geq 5\) be an odd integer and let all spaces and maps be as in \emph{(\ref{eq:mp_diag_cpspan3_odd})} Suppose that the conditions of \emph{Lemma \ref{lem:tertiaryobs_cpspan3_odd}} are satisfied and that \(c_{m-1}(\xi - \ell_1 \oplus \ell_2 \oplus \ell_3) =0.\) Further suppose that 
\begin{itemize}
    \item \(\co^{2m-1}(X;\Z)\) is finite abelian with no \(2\)-torsion; and 
    \item \(\co^{2m}(X;\Z)\) is torsion free.
\end{itemize}
Then the generalized quaternary obstruction
\[
\lobs^{2m-1}((\xi,\ell_1,\ell_2,\ell_3),q_{m,3},k_4) \subseteq \co^{2m-1}(X;\pi_{2m-2})
\]
 is the singleton set \(\{0\}\), where \(\pi_{2m-2} = \pi_{2m-2}(W(m,3))\) satisfies \emph{Proposition \ref{prop:piW(m,r)}}.
\end{thm}

In the rest of the section, we provide the details to the proof of Theorem \ref{thm:quaternaryobs_cpspan3_odd}.

\vspace{2mm}

\paragraph{\textbf{Determining the Moore--Postnikov invariants.}}

We compute the degree (\(2m-1\)) Moore--Postnikov invariant \(k_4 \in \co^{2m-1}(E[3];\pi_{2m-2})\) as follows. 

\begin{lem}\label{lem:k_4_cpspan3_odd}
Let \(k_4 \in \co^{2m-1}(E[3];\pi_{2m-2})\) be as in \emph{(\ref{eq:mp_diag_cpspan3_odd})}. Then the mod \(p\) part of \(k_4\) is equal to 
\begin{equation*}
\begin{cases}
Sq^2\iota^2_{2m-3} & p = 2 \text{ and } m \equiv 1,5~\mathrm{mod}~8; \\
\theta_2^2Sq^2\iota^2_{2m-3} & p = 4 \text{ and } m \equiv 7~\mathrm{mod}~8;\\
\theta_2^3Sq^2\iota^2_{2m-3} & p = 8 \text{ and } m \equiv 3~\mathrm{mod}~8; \\
0 & p = 3 \text{ and } m \equiv 0 ~\mathrm{mod}~ 3;
\end{cases}
\end{equation*}
where \(\iota^2_{2m-3} \in \co^{2m-3}(K(\Z, 2m-3);\Z/2)\) is the fundamental class and \(\theta_2^k\) is the homomorphism induced by the inclusion \(\Z/2 \xhookrightarrow{} \Z/2^k\), see \emph{Notation \ref{notate:cohomologyops}}.
\end{lem}

We divide the proof with respect to \(p\), where we assume that \(m\) has the corresponding value mod \(8\). In fact, we will only prove Lemma \ref{lem:k_4_cpspan3_odd} for \(p =2\) and \(p=3\). We omit the details of the proof in the case that \(p = 4,8\) and leave the proofs to the reader. Moreover, for a spectral sequence $(E_*^{*,*}, d_*^{*,*})$, we will denote by
\begin{equation*}
    d^{(p,q),(p+r, q-r+1)}_r \colon~ E^{p,q}_{r} \longrightarrow E^{p+r, q-r+1}_{r}
\end{equation*}
the differential in order to emphasize both its domain and the codomain. \\

\paragraph{\textbf{Proof of Lemma \ref{lem:k_4_cpspan3_odd} for \(p=2\).}}

Let all maps and spaces be as in (\ref{eq:mp_diag_cpspan3_odd}). Further let \(p = 2\) and \(m \equiv 1~ \mathrm{mod}~4\). In this section all cohomology groups have mod 2 coefficients. 

\begin{lem}\label{lem:tau_q_1}
Let \(m\geq 5\) be an odd integer, and let the fibration
\begin{equation*}
    F[1] \longrightarrow B(m-3,1^3) \overset{q_1}{\longrightarrow} E[1]
\end{equation*}
be as defined in \emph{(\ref{eq:mp_diag_cpspan3_odd})}. In the mod 2 Leray--Serre spectral sequence of \(q_1\), the images of the first four last-chance differentials satisfy the following:
\begin{enumerate}[\normalfont(i)]
    \item \(\mathrm{Im}(d_{2m-3}^{(0,2m-4),(2m-3,0)}) = \Z/2 \langle Sq^2 \iota_{2m-5}^2\rangle\); and 
    \item \(\mathrm{Im}(d_{2m-2}^{(0,2m-3),(2m-2,0)}) = \Z/2 \langle b_{m-1}\rangle \oplus \Z/2 \langle Sq^3 \iota_{2m-6}^2\rangle\); 
    \item \(\mathrm{Im}(d_{2m-1}^{(0,2m-2),(2m-1,0)}) = 0\); and 
    \item \(\mathrm{Im}(d_{2m}^{(0,2m-1),(2m,0)}) = \Z/2 \langle Sq^5 \iota_{2m-5}^2\rangle\).
\end{enumerate}
\end{lem}

\begin{figure}[h]
  \centering
  \includegraphics[width=0.5\textwidth]{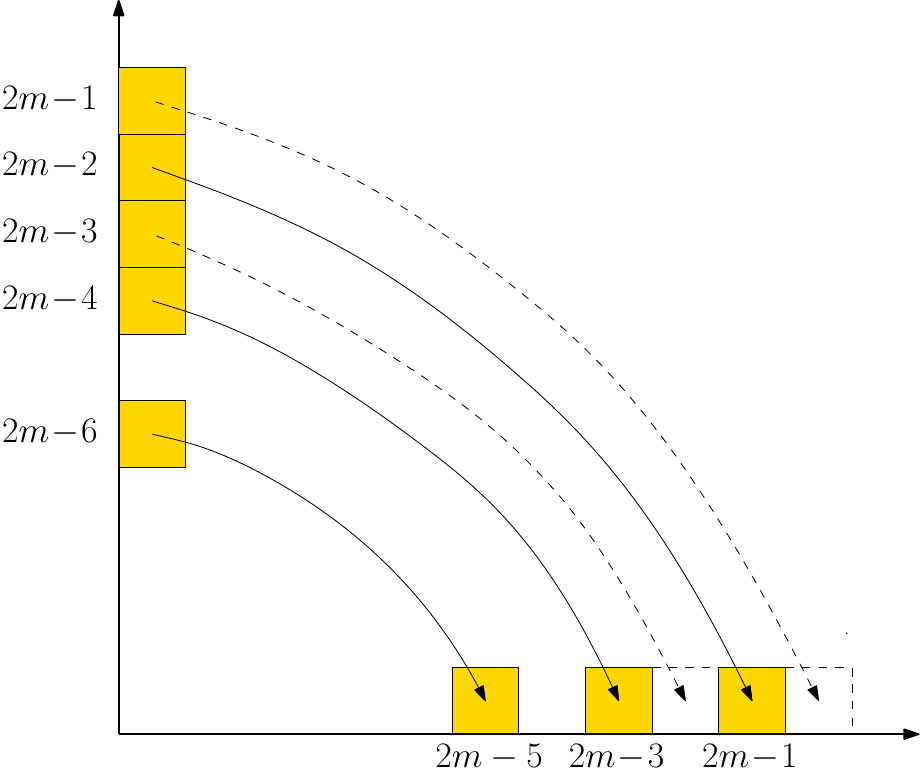}
  \caption{A portion of the Leray--Serre spectral sequence of the fibration \(K(\Z,2m-6)\to F[1]\to W(m,3)\) with mod 2 coefficients.}
  \label{fig:cohomology_F[1]_mod2}
\end{figure}

\begin{proof}
Throughout, all cohomology groups have mod 2 coefficients. To begin, we compute the relevant portion of \(\co^*(F[1];\Z/2).\) To this end, consider the mod 2 Leray--Serre spectral sequence \((E_*^{*,*},d_*)\) associated to the fibration 
\begin{equation}\label{eq:fib_F[1]}
K(\Z, 2m-6) \longrightarrow F[1] \overset{s_1}{\longrightarrow} W(m,3).
\end{equation}
By construction, the differential 
\[
d_{2m-5}^{(0,2m-6),(2m-5,0)} \colon \co^{2m-6}(K(\Z,2m-6)) \cong \Z/2 \langle \iota_{2m-6}^2\rangle \to \Z/2 \langle f_{2m-5}\rangle \cong \co^{2m-5}(W(m,3))
\]
is an isomorphism defined by \(d_{2m-5}(\iota_{2m-6}^2) = f_{2m-5}.\) Moreover, using Kudo's transgression theorem and Proposition \ref{prop:cohomology_rings}(v), we compute
\begin{itemize}
    \item \(d_{2m-3}^{(0,2m-4),(2m-3,0)}(Sq^2 \iota_{2m-6}^2) = Sq^2 f_{2m-5} = 0\) and 
    \item \(d_{2m-1}^{(0,2m-2),(2m-1,0)}(Sq^4 \iota_{2m-6}^2) = Sq^4 f_{2m-5} = f_{2m-1},\) since \(m \equiv 1 ~\mathrm{mod}~ 4.\)
\end{itemize}
See Figure \ref{fig:cohomology_F[1]_mod2} for illustration. 
Consequently, since $m > 5$, we compute the following groups. 
\begin{enumerate}[(a)]
    \item \(\co^{2m-4}(F[1]) \cong \Z/2 \langle Sq^2 \iota_{2m-6}^2\rangle\); 
    \item \(\co^{2m-3}(F[1]) \cong \Z/2 \langle f_{2m-3} \rangle \oplus \Z/2 \langle Sq^3 \iota_{2m-6}^2\rangle\);
    \item \(\co^{2m-2}(F[1]) \cong 0\);
    \item \(\co^{2m-1}(F[1]) \cong \Z/2 \langle Sq^5 \iota_{2m-6}^2\rangle.\)
\end{enumerate}
By similar arguments to the ones in the proof of Lemma \ref{lem:2m-1-coh-of-E[1]}, we compute the following. 

\begin{enumerate}[(a)]\setcounter{enumi}{4}
    \item \(\co^{2m-3}(E[1]) \cong \Z/2 \langle Sq^2 \iota_{2m-5}^2\rangle\);
    \item \(\co^{2m-2}(E[1]) \cong \Z/2 \langle Sq^3 \iota_{2m-5}^2\rangle \oplus (\co^{*}(B(m,1^3))/(b_{m-2}))^{(2m-2)}\); 
    \item \(\co^{2m-1}(E[1]) \cong \Z/2 \langle Sq^2 \iota_{2m-5}^2 \rangle \otimes \co^{2}(B(m,1^3))\);
    \item $\co^{2m}(E[1]) \cong \Z/2 \langle Sq^5 \iota_{2m-5}^2\rangle \oplus (\co^{*}(B(m,1^3))/(b_{m-2}, b_{m-1}b_1 + b_m))^{(2m)}$ 
\end{enumerate}
By considering the Thomas diagram associated to \((\star)\),
\[\begin{tikzcd}
	{K(\Z,2m-6)} & {*} & {K(\Z,2m-5)} \\
	{F[1]} & {B(m-3,1^3)} & {E[1]} \\
	{W(m,3)} & {B(m-3,1^3)} & {B(m,1^3)}
	\arrow[from=1-1, to=1-2]
	\arrow[from=1-1, to=2-1]
	\arrow[from=1-2, to=1-3]
	\arrow[from=1-2, to=2-2]
	\arrow[from=1-3, to=2-3]
	\arrow[from=2-1, to=2-2]
	\arrow[from=2-1, to=3-1]
	\arrow["{q_1}", from=2-2, to=2-3]
	\arrow[Rightarrow, no head, from=2-2, to=3-2]
	\arrow["{(\star)}"{description}, draw=none, from=2-2, to=3-3]
	\arrow["{p_1}", from=2-3, to=3-3]
	\arrow[from=3-1, to=3-2]
	\arrow["{q_{m,3}}", from=3-2, to=3-3]
\end{tikzcd}\]
the claim follows by naturality of spectral sequences. 
\end{proof}

As before, $\kappa_{2m-4} \in \co^{2m-4}(K(\Z/2, 2m-4);\Z/2)$ denotes the mod 2 fundamental class (see also Notation~\ref{notate:cohomologyops}(vii)).

\begin{lem}\label{lem:tau_q_2}
Let \(m\geq 5\) be an odd integer, and let the fibration
\begin{equation*}
    F[2] \longrightarrow B(m-3,1^3) \overset{q_2}{\longrightarrow} E[2]
\end{equation*}
be as defined in \emph{(\ref{eq:mp_diag_cpspan3_odd})}. Then the images of the first three last-chance differentials satisfy the following.
\begin{enumerate}[\normalfont(i)]
    \item \(\mathrm{Im}(d_{2m-2}^{(0,2m-3),(2m-2,0)}) = \Z/2 \langle b_{m-1}\rangle \oplus \Z/2 \langle Sq^2 \kappa_{2m-4}\rangle\);
    \item \(\mathrm{Im}(d_{2m-1}^{(0,2m-2),(2m-1,0)}) = \Z/2 \langle Sq^3 \kappa_{2m-4}\rangle\);
    \item \(\mathrm{Im}(d_{2m}^{(0,2m-1),(2m,0)})=  \Z/2\langle Sq^4\kappa_{2m-4}, Sq^3Sq^1 \kappa_{2m-4} \rangle.\)
\end{enumerate}
\end{lem}

\begin{proof}
First, we compute the relevant portion of \(\co^*(F[2]).\) Consider the mod \(2\) spectral sequence \((E_*^{*,*},d_*)\) associated to the fibration 
\begin{equation}\label{eq:F[2]}
K(\Z/2,2m-5) \longrightarrow F[2] \longrightarrow F[1].
\end{equation}
By construction, the differential 
\[
d_{2m-4}^{(0,2m-5),(2m-4,0)} \colon \co^{2m-5}(K(\Z/2,2m-5)) \longrightarrow \co^{2m-4}(F[1])
\]
is an isomorphism defined by \(d_{2m-4}(\kappa_{2m-5}) = Sq^2 \iota_{2m-6}^2\), whence \[E_{2m-3}^{0,2m-5} = E_{\infty}^{0,2m-5} = 0.\]
Moreover, \(Sq^2(Sq^2\iota_{2m-6}^2)\) necessarily vanishes since \(\co^{2m-2}(F[1]) = 0\). So Kudo's transgression theorem yields that \[ \Z/2 \langle Sq^2 \kappa_{2m-5}\rangle = E_{2m-2}^{0,2m-3} = E_{\infty}^{0,2m-3}.\]
Furthermore,
\[
d_{2m-3}^{(0,2m-4),(2m-3,0)}(Sq^1 \kappa_{2m-5}) = Sq^1 Sq^2 \iota_{2m-6}^2 = Sq^3 \iota_{2m-6}^2,
\]
so that \(E_{2m-2}^{0,2m-5} = E_{\infty}^{0,2m-5} = 0\).

Noting the Adem relations
\begin{itemize}
    \item \(Sq^3 Sq^2 = 0\); and 
    \item \(Sq^2Sq^1Sq^2 = Sq^5 + Sq^4 Sq^1,\)
\end{itemize}
we compute 
\begin{itemize}
    \item \(d_{2m-1}^{(0,2m-2),(2m-1,0)}(Sq^3 \kappa_{2m-5}) = 0,\) and 
    \item \(d_{2m-1}^{(0,2m-2),(2m-1,0)}(Sq^2Sq^1 \kappa_{2m-5}) = Sq^5 \iota_{2m-6}^2.\)
\end{itemize}
Lastly, since \(Sq^4Sq^2\) is admissible and \(Sq^3Sq^1Sq^2 = Sq^3Sq^3 = Sq^5Sq^1\), we compute the differential
\[
d_{2m}^{(0,2m-1),(2m,0)} \colon \co^{2m-1}(K(\Z/2,2m-5)) \longrightarrow \co^{2m}(F[1])
\]
as follows:
\begin{itemize}
    \item \(Sq^4\kappa_{2m-5} \mapsto Sq^4Sq^2\iota_{2m-6}^2\); and 
    \item \(Sq^3Sq^1\kappa_{2m-5} \mapsto 0.\)
\end{itemize}

Thusly, we obtain the following. 
\begin{enumerate}
    \item[(a)] \(\co^{2m-3}(F[2]) \cong \Z/2 \langle f_{2m-3}\rangle \oplus \Z/2 \langle Sq^2 \kappa_{2m-5} \rangle\); 
    \item[(b)] \(\co^{2m-2}(F[2]) \cong \Z/2\langle Sq^3 \kappa_{2m-5}\rangle\); and 
    \item[(c)] \(\co^{2m-1}(F[2]) \cong \Z/2 \langle  Sq^3Sq^1 \kappa_{2m-5}\rangle.\)
\end{enumerate}
Second, by similar arguments applied to the Leray--Serre spectral sequence of the fibration \(p_2 \colon E[2] \to E[1],\) one calculates that 
\begin{enumerate}
    \item[(d)] \(\co^{2m-2}(E[2]) \cong \co^{2m-2}(B(m,1^3))/(b_{m-2}) \oplus \Z/2 \langle Sq^2 \kappa_{2m-4}\rangle\); 
    \item[(e)] \(\co^{2m-1}(E[2]) \cong \Z/2 \langle Sq^3 \kappa_{2m-4}\rangle\); and 
    \item[(f)] \(\co^{2m}(E[2]) \cong \co^{2m}(E[1])/(Sq^2 \iota_{2m-5}^2) \oplus \Z/2 \langle Sq^3Sq^1\kappa_{2m-4} \rangle.\)
\end{enumerate}
Finally, using Lemma \ref{lem:tau_q_1}, we study the map of spectral sequences induced by the Thomas diagram associated to \((\star \star)\). 
\[\begin{tikzcd}
	{K(\Z/2,2m-5)} & {*} & {K(\Z/2,2m-4)} \\
	{F[2]} & {B(m-3,1^3)} & {E[2]} \\
	{F[1]} & {B(m-3,1^3)} & {E[1]}
	\arrow[from=1-1, to=1-2]
	\arrow[from=1-1, to=2-1]
	\arrow[from=1-2, to=1-3]
	\arrow[from=1-2, to=2-2]
	\arrow[from=1-3, to=2-3]
	\arrow[from=2-1, to=2-2]
	\arrow[from=2-1, to=3-1]
	\arrow["{q_2}", from=2-2, to=2-3]
	\arrow[Rightarrow, no head, from=2-2, to=3-2]
	\arrow["{(\star \star)}"{description}, draw=none, from=2-2, to=3-3]
	\arrow["{p_2}", from=2-3, to=3-3]
	\arrow[from=3-1, to=3-2]
	\arrow["{q_1}", from=3-2, to=3-3]
\end{tikzcd}\]
The claim follows by the naturality of spectral sequences.
\end{proof}

\begin{lem}\label{lem:tau_q_3}
Let \(m\geq 5\) be an odd integer, and let the fibration \(q_3 \colon B(m-3,1^3) \to E[3]\) with fiber \(F[3]\) be as defined in \emph{(\ref{eq:mp_diag_cpspan3_odd})}. Then \(Sq^2 \kappa_{2m-3}\) is the characteristic class of \(q_3\). 
\end{lem}

\begin{proof}
To begin, consider the mod \(2\) Leray--Serre spectral sequence \((E_*^{*,*},d_*)\) of the fibration
\[
K(\Z \oplus \Z/2, 2m-4) \longrightarrow F[3] \longrightarrow F[2].
\]
Then the differential
\[
d_{2m-3}^{(0,2m-4),(2m-3,0)} \colon \co^{2m-4}(K(\Z \oplus \Z/2, 2m-4)) \cong \Z/2\langle \iota_{2m-4}^2\rangle \oplus \Z/2 \langle \kappa_{2m-4}\rangle \longrightarrow \co^{2m-3}(F[2]),
\]
is an isomorphism defined by 
\[
d_{2m-3}(\iota_{2m-4}^2, \kappa_{2m-4}) = (f_{2m-3}, Sq^2 \kappa_{2m-5}).
\]
Hence \(E_{2m-2}^{0,2m-4} = E_{\infty}^{0,2m-4} = 0.\)

Moreover, it follows from Kudo's transgression theorem and the Adem relations
\begin{itemize}
    \item \(Sq^1Sq^2 = Sq^3\); and 
    \item \(Sq^2Sq^2 = Sq^3Sq^1\); 
\end{itemize}
that the differential
\(
d_{2m-2}^{(0,2m-3),(2m-2,0)}
\)
is nontrivial and 
\[
d_{2m-1}^{(0,2m-2),(2m-1,0)}(Sq^2 \iota_{2m-4}^2, Sq^2 \kappa_{2m-4}) = (0, Sq^3Sq^1\kappa_{2m-5}).
\]
Hence \(E_{\infty}^{0,2m-3} = 0\) and \(E_{\infty}^{0,2m-2} = \Z/2 \langle Sq^2 \iota^2_{2m-4}\rangle.\) Moreover, since \(E_{\infty}^{0,2m-2}\) is the only nontrivial term on the \((2m-2)\)-th diagonal on the \(E_\infty\)-page, it follows that \[\co^{2m-2}(F[3]) \cong \Z/2 \langle Sq^2 \iota^2_{2m-4}\rangle.\]

Similar arguments applied to the fibration \[K(\Z \oplus \Z/2,2m-3) \longrightarrow E[3]\longrightarrow E[2]\] demonstrate that 
\(
\Z/2 \langle Sq^2 \iota^2_{2m-3}\rangle \subseteq \co^{2m-1}(E[3]). 
\)
To conclude, we study the map of spectral sequences induced by the Thomas diagram associated to \((\star\star \star)\)
\[\begin{tikzcd}
	{K(\Z \oplus\Z/2,2m-4)} & {*} & {K(\Z\oplus \Z/2,2m-3)} \\
	{F[3]} & {B(m-3,1^3)} & {E[3]} \\
	{F[2]} & {B(m-3,1^3)} & {E[2]}
	\arrow[from=1-1, to=1-2]
	\arrow[from=1-1, to=2-1]
	\arrow[from=1-2, to=1-3]
	\arrow[from=1-2, to=2-2]
	\arrow[from=1-3, to=2-3]
	\arrow[from=2-1, to=2-2]
	\arrow[from=2-1, to=3-1]
	\arrow["{q_3}", from=2-2, to=2-3]
	\arrow[Rightarrow, no head, from=2-2, to=3-2]
	\arrow["{(\star \star\star)}"{description}, draw=none, from=2-2, to=3-3]
	\arrow["{p_3}", from=2-3, to=3-3]
	\arrow[from=3-1, to=3-2]
	\arrow["{q_2}", from=3-2, to=3-3]
\end{tikzcd}\]
from which the result follows.
\end{proof}

\paragraph{\textbf{Proof of Lemma \ref{lem:k_4_cpspan3_odd} for \(p=3\).}}
Let all maps and spaces be as in (\ref{eq:mp_diag_cpspan3_odd}). Further let \(p = 3\) and \(m \equiv 0~ \mathrm{mod}~ 3\). In this section, all cohomology groups have mod 3 coefficients. By similar arguments to Case III of Lemma \ref{lem:2m-1-coh-of-E[1]}, we see that \(\co^{2m-1}(E[1]) \cong 0.\) Moreover, it follows from the mod 3 Leray--Serre spectral sequence of the homotopy fibration \(K(\Z/2,2m-4) \to E[2]\to E[1]\) that \(\co^{2m-1}(E[2]) \cong \co^{2m-2}(E[1])\). Indeed \(\co^i(K(\Z/2,2m-3)) \cong 0\), for \(i = 2m-4,2m-3,2m-1\), so there are no non-trivial terms on the \((2m-1)\)-th diagonal on the \(E_{\infty}\)-page. 

Finally, consider the mod 3 Leray--Serre spectral sequence of the homotopy fibration \[K(\Z \oplus \Z/2,2m-3) \to E[3] \to E[2].\]  Now \(\co^{2m-3}(K(\Z \oplus \Z/2,2m-3)) \cong \Z/3\langle \iota_{2m-3}\rangle\) and it is straightforward to see that \(\iota_{2m-3}\) transgresses to \[b_{m-1} \in \co^{2m-2}(E[2]) \cong \co^{2m-2}(B(m,1^3))/(b_{m-2}).\] Furthermore \(\co^{2m-1}(K(\Z \oplus \Z/2,2m-3) \cong 0\), whence there are no nontrivial terms on the \((2m-1)\)-th diagonal on the \(E_\infty\)-page. In conclusion, \(\co^{2m-1}(E[3]) \cong 0\) and the result follows.  \\

\paragraph{\textbf{Computing the indeterminancy of lifts.}}\label{subsubsec:indet_cpspan3_odd}
Recall that the group \(\pi_{2m-2} \coloneqq \pi_{2m-2}(W(m,3))\) satisfies Proposition \ref{prop:piW(m,r)}. Let all the assumptions of Theorem \ref{thm:quaternaryobs_cpspan3_odd} be satisfied. Then there is a lift of \((\xi,\ell_1,\ell_2,\ell_3)\) to \(E[4]\) if and only if \(g^*(k_4) = 0\) in \(\co^{2m-1}(X;\pi_{2m-2})/\mathrm{Indet}\), where \(g \colon X \to E[3]\) is any lift of \((\xi,\ell_1,\ell_2,\ell_3)\) to \(E[3]\) and \(\mathrm{Indet} \subseteq \co^{2m-1}(X;\pi_{2m-2})\) denotes the indeterminacy of such lifts. We will briefly sketch the reason that \(\mathrm{Indet} = 0.\)

Write \(K \coloneqq K(\Z \oplus \Z/2, 2m-3)\). We have the following homotopy commutative diagram
\[\begin{tikzcd}
	K && K \\
	{K\times B(m-3,1^3)} & {K\times E[3]} & {E[3]} \\
	{B(m-3,1^3)} && {E[2]}
	\arrow[Rightarrow, no head, from=1-1, to=1-3]
	\arrow[from=1-1, to=2-1]
	\arrow[from=1-3, to=2-3]
	\arrow["{1\times q_3}", from=2-1, to=2-2]
	\arrow["\nu", curve={height=-18pt}, from=2-1, to=2-3]
	\arrow["\pi", from=2-1, to=3-1]
	\arrow["\mu", from=2-2, to=2-3]
	\arrow["{p_3}", from=2-3, to=3-3]
	\arrow["{\bar{s}}", curve={height=-12pt}, from=3-1, to=2-1]
	\arrow["{q_2}", from=3-1, to=3-3]
\end{tikzcd}\]
and an exact sequence
\[
\dots \longrightarrow \mathrm{ker}(q_3^*)\cap\co^{2m-1}(E[3])\overset{\nu^*}\longrightarrow \mathrm{ker}(\bar{s}^*)\cap\co^{2m-1}(K\times B(m-3,1^3))\overset{\tau_1}\longrightarrow \co^{2m}(E[2]) \longrightarrow \dots,
\]
where all cohomology groups have \(\pi_{2m-2}\) coefficients. As usual, we write \(\mu^*(k_4) = 1 \otimes k_4 + \nu^*(k_4)\), where
\[
\nu^*(k_4) \in \mathrm{ker}(\bar{s}) \cap \mathrm{ker}(\tau_1) \cap \co^{2m-1}(K\times B(m-3,1^3)). 
\]
But by direct calculation using the proof of Lemma \ref{lem:tau_q_3}, one sees that the group \[\mathrm{ker}(\bar{s}) \cap \mathrm{ker}(\tau_1) \cap \co^{2m-1}(K\times B(m-3,1^3))\] is trivial because the generator of $\co^{2m-3}(K)$ transgresses to \(b_{m-1}\). Whence 
\(\nu^*(k_4) = 0\). Thus \(\mu^*(k_4) = 1 \otimes k_4\), from which it follows that \(\mathrm{Indet} = 0\), as required.\\

\paragraph{\textbf{Conditions for the vanishing of the generalized quaternary obstruction.}}
Finally, we finish the proof of Theorem \ref{thm:quaternaryobs_cpspan3_odd}.

Supposing that the conditions of Lemma \ref{lem:tertiaryobs_cpspan3_odd} are satisfied, we can consider a lift \(g \colon X \to E[3]\) of \((\xi, \ell_1,\ell_2,\ell_3)\). The triviality of the indeterminacy subgroup computed in Section \ref{subsubsec:indet_cpspan3_odd} yields that the obstruction set \[\lobs^{2m-1}((\xi,\ell_1,\ell_2,\ell_3),q_{m,3},k_4)\subseteq \co^{2m-1}(X;\pi_{2m-2})\] is the singleton \(\{g^*k_4\}\). Moreover, by Lemma \ref{lem:k_4_cpspan3_odd}, any \(\Z/3\) part of \(k_4\) is zero. Thus, we require  \(g^*k_4 \in \co^{2m-1}(X;\Z/p)\) to be trivial, for \(p = 2,4,8\). From the description of \(k_4\) in Lemma \ref{lem:k_4_cpspan3_odd}, there is no simple way to relate \(g^*k_4\) with characteristic classes of the bundle \((\xi, \ell_1,\ell_2,\ell_3)\). Instead, we impose that \(\co^{2m-1}(X;\Z/p) \cong 0\), for \(p=2,4,8.\) Theorem \ref{thm:quaternaryobs_cpspan3_odd} now follows from the universal coefficient theorem for cohomology \cite[Theorem 10, p.\ 246]{spanier1989algebraic}.

\subsection{The Generalized Quinary Obstruction}\label{subsec:quinaryobs_cpspan3_odd}

\begin{thm}\label{thm:quinaryobs_cpspan3_odd}
Let \(m\geq 5\) be an odd integer and let all spaces and maps be as in \emph{(\ref{eq:mp_diag_cpspan3_odd})}. Suppose that the conditions of \emph{Theorem \ref{thm:quaternaryobs_cpspan3_odd}} are satisfied. Then the generalized quinary obstruction \[\lobs^{2m}((\xi,\ell_1,\ell_2,\ell_3),q_{m,3},k_5) \subseteq \co^{2m}(X;\Z)\] is the singleton set consisting of a nonzero rational multiple of \(c_{m}(\xi-\ell_1 \oplus \ell_2\oplus \ell_3)\). Hence \(\lobs^{2m}((\xi,\ell_1,\ell_2,\ell_3),q_{m,3},k_5) = \{0\}\) if and only if $c_{m}(\xi-\ell_1 \oplus \ell_2\oplus \ell_3)=0$.
\end{thm}
The rest of the section is dedicated to the proof of Theorem \ref{thm:quinaryobs_cpspan3_odd}.\\

Let all maps and spaces be as defined in (\ref{eq:mp_diag_cpspan3_odd}). Further let the assumptions of Theorem \ref{thm:quaternaryobs_cpspan3_odd} be satisfied. Then there is a lift of \((\xi,\ell_1,\ell_2,\ell_3)\) along \(q_{m,3}\) if and only if there exists a lift \(g \colon X \to E[4]\) of \((\xi,\ell_1,\ell_2,\ell_3)\) such that \(g^*k_4 = 0\) in \(\co^{2m}(X;\Z)/\mathrm{Indet}.\) By virtually identical arguments to the proof of Theorem \ref{thm:quaternaryobs_even}, one checks that \(\mathrm{Indet} = 0\). We omit the details for brevity. 

Now since \(\mathrm{Indet} = 0\), there is a unique quinary obstruction to lifting \((\xi,\ell_1,\ell_2,\ell_3)\). It is possible that methods analogous to those in Section \ref{subsec:quaternaryobs_cpspan3_even} may apply here; however, since \(\co^{2m}(X;\Z)\) is torsion free by assumption, we may ignore torsion and use basic tools of rational homotopy theory to determine this obstruction. Henceforth, all cohomology groups have rational coefficients unless otherwise indicated.

Recall that since the rationalization functor \((-)_0\) is a \(\mathbb{Q}\)-localization, it is exact \cite[Proposition 3.3]{atiyah2018introduction} and preserves pullbacks \cite[Exercise 4.10]{cornea2003lusternik}. Applying the functor \((-)_0\) to the Moore--Postnikov tower (\ref{eq:mp_diag_cpspan3_odd}) for the fibration \(q_{m,3} \colon B(m-3,1^3) \to B(m,1^3)\) yields the following Moore--Postnikov tower of the rationalization of \(q_{m,3}\):

\begin{equation}\label{eq:mp_diag_rational}
\begin{tikzcd}[ampersand replacement=\&]
	\&\& {E[5]_0} \\
	\&\& {E[4]_0} \& {K(\mathbb{Q},2m)} \\
	\&\& {E[3]_0} \& {0} \\
	\&\& {E[2]_0} \& {K(\mathbb{Q},2m-2)} \\
	\&\& {E[1]_0} \& {0} \\
	{W(m,3)_0} \& {B(m-3,1^3)_0} \& {B(m,1^3)_0} \& {K(\mathbb{Q},2m-4)} \\
	\\
	\&\& X_0
	\arrow["{(p_5)_0}", from=1-3, to=2-3]
	\arrow["{k_5\otimes \mathbb{Q}}", from=2-3, to=2-4]
	\arrow["{(p_4)_0}", from=2-3, to=3-3]
	\arrow["{k_4\otimes \mathbb{Q}}", from=3-3, to=3-4]
	\arrow["{(p_3)_0}", from=3-3, to=4-3]
	\arrow["{k_3\otimes \mathbb{Q}}", from=4-3, to=4-4]
	\arrow["{(p_2)_0}", from=4-3, to=5-3]
	\arrow["{k_2\otimes \mathbb{Q}}", from=5-3, to=5-4]
	\arrow["{(p_1)_0}", from=5-3, to=6-3]
	\arrow[from=6-1, to=6-2]
	\arrow["{(q_4)_0}"{description}, from=6-2, to=2-3]
	\arrow["{(q_3)_0}"{description}, from=6-2, to=3-3]
	\arrow["{(q_2)_0}"{description}, from=6-2, to=4-3]
	\arrow["{(q_1)_0}"{description}, from=6-2, to=5-3]
	\arrow["{(q_{m,3})_0}", from=6-2, to=6-3]
	\arrow["{k_1\otimes \mathbb{Q}}", from=6-3, to=6-4]
	\arrow["{(\xi,\ell_1,\ell_2,\ell_3)_0}"', from=8-3, to=6-3]
\end{tikzcd}
\end{equation}
\noindent Now \(k_5 \in \co^{2m}(E[4];\Z)\) is the characteristic element in the fibration 
\[
F[4] \longrightarrow B(m-3,1^3) \overset{q_4}\longrightarrow E[4].
\]
Moreover, \(F[4]\) is \((2m-2)\)-connected with \(\co^{2m-1}(F[4]) \cong \mathbb{Q}.\)
As in Section \ref{subsec:quaternaryobs_cpspan3_even}, for each \(i = 1,2,3,4\), let us denote by \(s_i \colon F[i] \to F[i-1]\) the induced maps between homotopy fibers of the maps \(p_i \colon E[i] \to E[i-1]\), where \(F[0] = W(m,3),\) \(E[0] = B(m,1^3)\), and \(q_0 = q_{m,3}\), c.f., Lemma \ref{lem:thomas_diagram}. Set \(s\coloneqq s_1 \circ s_2 \circ s_3 \circ s_4\) and \(p \coloneqq p_1 \circ p_2 \circ p_3 \circ p_4\).
We also recall that 
\[
\co^*(W(m,3)) \cong \Lambda_{\mathbb{Q}}[e_{2m-5},e_{2m-3},e_{2m-1}], 
\]
with \(\lvert e_i \rvert = i\), and 
\[
\co^*(BU(n)) \cong \mathbb{Q}[c_1,\dots, c_n],
\]
with \(\lvert c_j \rvert = 2j\). 

Then it is straightforward to establish the following lemma. 

\begin{lem}\label{lem:F[4]_iso}
The map \(s_0 \colon F[4]_0\longrightarrow W(m,3)_0\) induces an isomorphism 
\[
s_0^* \colon \co^{2m-1}(W(m,3)) \overset{\cong}\longrightarrow \co^{2m-1}(F[4]),
\]
in rational cohomology.
\end{lem}

Finally, consider the map between homotopy fibrations
\begin{equation*}
    \begin{tikzcd}
        F[4]_0 \arrow[r] \arrow[d, "s_0"]\ & B(m-3, 1^3)_0 \arrow[d, equal] \arrow[r, "(q_4)_0"] & E[4]_0 \arrow[d, "p_0"]\\
        W(m,3)_0 \arrow[r] & B(m-3, 1^3)_0 \arrow[r, "(q_{m,3})_0"] & B(m,1^3)_0
    \end{tikzcd}
\end{equation*}
and the induced map between rational Leray--Serre spectral sequences. The image of the transgression in the bottom homotopy fibration on the $2m$-th page
\begin{equation*}
    d_{2m}\colon \mathbb{Q} \langle e_{2m-1} \rangle \cong E_{2m}^{0, 2m-1} \longrightarrow E_{2m}^{2m, 0} \cong \co^{2m}(B(m,1^3))
\end{equation*}
is generated by $e_{2m-1} \longmapsto a_{m}\coloneqq c_m(\gamma_m \times 1^{\times 3} - 1 \times \gamma_1^{\times 3})$.
By the naturality of spectral sequences and Lemma \ref{lem:F[4]_iso}, we conclude that on the $2m$-th page of the spectral sequence of $q_4$ the element
\begin{equation*}
    s_0^*(e_{2m-1}) \in \co^{2m-1}(F[4])
\end{equation*}
transgresses to
\begin{equation*}
    p_0^*(a_m) = k_4 \otimes \mathbb{Q} \in \co^{2m}(E[4]).
\end{equation*}
Finally, due to the homotopy commutativity of lifts 
\[\begin{tikzcd}
	& & {E[4]} \\
	X & & {B(m,1^3)}
	\arrow["p", from=1-3, to=2-3]
	\arrow["g", from=2-1, to=1-3]
	\arrow["{(\xi,\ell_1,\ell_2,\ell_3)}"', from=2-1, to=2-3]
\end{tikzcd}\]
it follows that
\[
c_m(\xi-\ell_1\oplus \ell_2 \oplus \ell_3) = (p\circ g)^*(a_m) =  g^*(k_4\otimes \mathbb{Q}) \in \co^{2m}(X;\mathbb{Q}).
\]

In conclusion, for any lift \(g \colon X \to E[4]\) of \((\xi, \ell_1,\ell_2,\ell_3)\), the virtual Chern class \(c_m(\xi - \ell_1 \oplus \ell_2 \oplus \ell_3)\) is an integer multiple of the quinary obstruction \(g^*(k_4) \in \co^{2m}(X;\Z)\). But since \(\co^{2m}(X;\Z)\) is torsion free, \(g^*k_4=0\) if and only if \(c_m(\xi - \ell_1 \oplus \ell_2 \oplus \ell_3)=0.\) This completes the proof of Theorem \ref{thm:quinaryobs_cpspan3_odd}.

\begin{rem}
    We thank the anonymous referee for pointing out that the rational methods used in this section can be pushed further. That is, one can ask about rational line bundle decompositions of complex vector bundles. More specifically, consider a complex \(m\)-bundle \(\xi\) and \(r\) complex line bundles \(\ell_1,\dots,\ell_r\) over over a \(2m\)-dimensional CW complex \(X\). Towards constructing a rational Moore—Postnikov tower of the fibration \(q_{m,r}\), one computes the rational homotopy groups of the complex Stiefel manifold \(W(m,r)\) as follows: 
    \begin{equation}
        \pi_i(W(m,r))\otimes \mathbb Q = 
            \begin{cases}
            0 & 1 \leq i \leq 2(m-r) \text{ and } i \geq 2m\\
            \mathbb Q & 2(m-r)+1 \leq i \leq 2m-1
            \end{cases}
    \end{equation}
    Ultimately, since all the Moore—Postnikov truncations in the rational tower are products of rational Eilenberg—Maclane spaces, one deduces that the vanishing of the rational virtual Chern classes
    \[
        c_i(\xi - \ell_1 \oplus \cdots \oplus \ell_r) \otimes \mathbb Q \in \co^{2i}(X;\mathbb Q),
    \]
    for \(i> m-r\), is a necessary and sufficient condition for the existence of a map \(g\colon X \to B(m-r,1^r)\) making the following diagram commute
    \[
        \begin{tikzcd}[ampersand replacement=\&] \&\& {\mathrm{H}^*(BU(m-r,1^r);\mathbb{Q})} \\ \\ {\mathrm{H}^*(X;\mathbb{Q})} \&\& {\mathrm{H}^*(BU(m,1^r);\mathbb{Q})} \arrow["{g^*\otimes \mathbb{Q}}"', dashed, from=1-3, to=3-1] \arrow["{q_{m,r}^*\otimes \mathbb{Q}}"', from=3-3, to=1-3] \arrow["{(\xi,\ell_1,\dots,\ell_r)^*\otimes \mathbb{Q}}", from=3-3, to=3-1] \end{tikzcd}
    \]
    in rational cohomology. 
\end{rem}

\section{Acknowledgments}\label{subsec:acknowledgements}
The authors thank the anonymous referee for a careful reading of the manuscript and several helpful suggestions that improved the exposition.
The second author earnestly thanks Mark Grant for countless discussions about the contents of this paper, the invaluable suggestion to utilize rational homotopy theory in the proofs of Proposition \ref{prop:piW(m,r)} and Theorem \ref{thm:quinaryobs_cpspan3_odd}, and for a careful reading of the paper. The second author also thanks the following people for useful discussions: Markus Upmeier, Jarek K\k{e}dra, Cihan Bahran, Ran Levi, and John Oprea.  Lastly, both authors extend their gratitude to Pavle Blagojevi\'{c}.

\bibliographystyle{plain}
\bibliography{references.bib}
\end{document}